\documentclass{amsart}[12pt]
\usepackage[curve,matrix,arrow,frame,tips]{xy}
\usepackage{mathrsfs}
\usepackage{enumerate}
\usepackage{amssymb}
\usepackage{stmaryrd}
\usepackage{manfnt}
\usepackage[pdfborder={0 0 0}, pdftitle={A Real-global equivariant Segal--Becker splitting, explicit Brauer induction, and global Adams operations}, pdfauthor={Stefan Schwede}, pdfstartview={FitH}]{hyperref}

\DeclareMathOperator{\colim}{colim}
\DeclareMathOperator{\ad}{ad}
\DeclareMathOperator{\eig}{eig}
\DeclareMathOperator{\gl}{gl}
\DeclareMathOperator{\gr}{gr}
\DeclareMathOperator{\ev}{ev}
\DeclareMathOperator{\fc}{fc}
\DeclareMathOperator{\sa}{sa}
\DeclareMathOperator{\im}{im}
\DeclareMathOperator{\Gal}{Gal}
\DeclareMathOperator{\Sym}{Sym}
\DeclareMathOperator{\map}{map}
\DeclareMathOperator{\Id}{Id}
\DeclareMathOperator{\sh}{sh}
\DeclareMathOperator{\tr}{tr}
\DeclareMathOperator{\trc}{trace}
\DeclareMathOperator{\Vect}{Vect}
\DeclareMathOperator{\res}{res}

\newcommand{\bk}{{\mathscr{C}}}
\newcommand{\mCl}{{\mathbb C\text{l}}}
\newcommand{\mC}{{\mathbb C}}
\newcommand{\mF}{{\mathbb F}}
\newcommand{\mL}{{\mathbb L}}
\newcommand{\mN}{{\mathbb N}}
\newcommand{\mQ}{{\mathbb Q}}
\newcommand{\mR}{{\mathbb R}}
\newcommand{\mS}{{\mathbb S}}
\newcommand{\mZ}{{\mathbb Z}}
\newcommand{\Fin}{{{\mathcal F}in}}
\newcommand{\Ab}{{\mathcal A}b}
\newcommand{\Hc}{{\mathcal H}}
\newcommand{\Kc}{{\mathcal K}}
\newcommand{\Lc}{{\mathcal L}}
\newcommand{\Sc}{{\mathcal S}}
\newcommand{\Uc}{{\mathcal U}}
\newcommand{\bA}{{\mathbf A}}
\newcommand{\bGr}{\mathbf{Gr}}
\newcommand{\bL}{{\mathbf L}}
\newcommand{\bP}{{\mathbf P}}
\newcommand{\bU}{{\mathbf U}}
\newcommand{\bku}{{\mathbf{ku}}}
\newcommand{\bkr}{{\mathbf{kr}}}
\newcommand{\bKU}{{\mathbf{KU}}}
\newcommand{\bKR}{{\mathbf{KR}}}
\newcommand{\bBU}{{\mathbf{BU}}}
\newcommand{\bBUP}{{\mathbf{BUP}}}
\newcommand{\bBOP}{{\mathbf{BOP}}}
\newcommand{\iso}{\cong}
\newcommand{\sm}{\wedge}
\newcommand{\tensor}{\otimes}
\newcommand{\xra}{\xrightarrow}
\newcommand{\xla}{\xleftarrow}
\newcommand{\GH}{{\mathcal{GH}}}
\newcommand{\upi}{{\underline \pi}}
\newcommand{\td}[1]{\langle #1\rangle}
\newcommand{\gh}[1]{\llbracket #1\rrbracket}

\renewcommand{\to}{\longrightarrow}

\numberwithin{equation}{section}
\newtheorem{theorem}[equation]{Theorem}
\newtheorem{cor}[equation]{Corollary}
\newtheorem{prop}[equation]{Proposition}
 
\theoremstyle{definition}
\newtheorem{defn}[equation]{Definition}
\newtheorem{rk}[equation]{Remark}
\newtheorem{eg}[equation]{Example}
\newtheorem{construction}[equation]{Construction}

\newcommand{\Danger}{%
   \par\penalty-500\begingroup\parindent=0.0pt\clubpenalty=10000%
   \def\par{\endgraf\endgroup}\leavevmode
   \hangindent=2.4em\relax
   \hangafter=-2\relax
   \hbox to 0pt{\hss\relax
      \vbox to 0pt{\vskip-8pt\relax
        \hbox
          {\relax\dbend}
\vss}\kern.70em}}

\begin{document}

\title[Real-global Segal--Becker splitting]%
{A Real-global equivariant Segal--Becker splitting,\\
explicit Brauer induction, and global Adams operations}
 
\author{Stefan Schwede}
\email{schwede@math.uni-bonn.de}
\address{Mathematisches Institut, Universit{\"a}t Bonn, Germany}

\begin{abstract}
  We prove a splitting result in global equivariant homotopy theory
  that is a simultaneous refinement of the  Segal--Becker splitting
  and its `Real' and equivariant generalizations,
  and of the explicit Brauer induction of Boltje  and Symonds.
  We show that the morphism of ultra-commutative Real-global ring spectra
  from $\Sigma^\infty_+ B_{\gl}U(1)$ to the Real-global K-theory spectrum
  that classifies the tautological Real $U(1)$-representation
  admits a section on underlying Real-global infinite loop spaces.
  We prove that this global Segal--Becker splitting induces
  the classical Segal--Becker splittings on equivariant cohomology theories,
  and that it induces the Boltje--Symonds explicit Brauer induction on equivariant homotopy groups.
  As an application we rigidify the unstable Adams operations in Real-equivariant K-theory
  to global self-maps of the Real-global space $\bBUP$.\medskip
\end{abstract}

\maketitle

\setcounter{tocdepth}{1}
\tableofcontents

\section*{Introduction}

The purpose of this paper is to establish
a fundamental property of one of the key objects in equivariant homotopy theory,
the Real-global K-theory spectrum $\bKR$, representing equivariant complex K-theory and its
Real-equivariant generalization.
More specifically, we prove a splitting result in global
equivariant homotopy theory that is a simultaneous refinement and generalization
of the Segal--Becker splitting  \cite{becker:characteristic_and_K, segal:CP infty},
its equivariant extension \cite{crabb:brauer, iriye-kono:Segal-Becker KG},
its `Real' generalization \cite{nagata-nishida-toda:Segal-Becker for KR},
and of the explicit Brauer induction of Boltje \cite{boltje:canonical brauer}
and Symonds \cite{symonds-splitting}.
In the presence of both a Real and an equivariant direction, our splitting exhibits a subtle
deviation from additivity that is inconsistent with the
Real-equivariant version  \cite[Theorem 1']{iriye-kono:Segal-Becker KG}
of the Segal--Becker splitting.

Our main player is a morphism of ultra-commutative Real-global ring spectra
$\eta\colon\Sigma^\infty_+ \bP\to\bKR$ from the unreduced suspension spectrum
of the Real-global space $\bP$, a multiplicative model of the global classifying space of $U(1)$
and a global refinement of $\mC P^\infty$, to the Real-global K-theory spectrum.
We construct a section to the morphism of Real-global infinite loop spaces
$\Omega^\bullet(\eta)\colon\Omega^\bullet(\Sigma^\infty_+\bP)\to\Omega^\bullet(\bKR)$
that is a Real-global $\sigma$-loop map.
We prove that this {\em global Segal--Becker splitting}
induces the classical Segal--Becker splittings on equivariant cohomology theories,
and that it induces the Boltje--Symonds explicit Brauer induction on equivariant homotopy groups.
As an application we rigidify the unstable Adams operations in equivariant K-theory
to Real-global self-maps of the representing Real-global space $\bBUP$.\medskip

To put our results into perspective, we give a brief review of the history of the
Segal--Becker splitting.
Its origin is the statement that the morphism $\Sigma^\infty_+ \mC P^\infty\to K U$
from the suspension spectrum of the infinite complex projective space to
the complex K-theory spectrum that classifies the tautological line bundle 
over $\mC P^\infty$ has a section, by a one-fold loop map, after passing to infinite loop spaces.
The theorem implies in particular that the transformation of degree 0 cohomology
theories on spaces is a split epimorphism,
and several papers on the subject state the splitting in this form.
The original splitting theorem was proved by Segal in \cite{segal:CP infty}
for complex K-theory.
Becker gives a different proof in \cite{becker:characteristic_and_K}
that also provides a splitting for real and symplectic K-theory.
Segal's construction produces $p$-complete maps
for every prime $p$ that are then assembled via an arithmetic square;
all other sources on the subject use transfers in some incarnation to construct
the relevant splitting.
It seems to be well-known that the section to
$\Omega^\infty(\Sigma^\infty_+ \mC P^\infty)\to\Omega^\infty K U\simeq \mZ\times B U$
can be arranged as a loop map, but it does {\em not} deloop twice;
we recall an argument in Remark \ref{rk:no_double_deloop}.

An equivariant generalization of the Segal--Becker splitting
for finite groups was obtained by Iriye and Kono \cite[Theorem 1]{iriye-kono:Segal-Becker KG}.
Crabb's paper \cite{crabb:brauer} provides a particularly nice exposition
that works for all compact Lie groups; we review it
in Construction \ref{con:classical SB}, where we also extend it to the Real-equivariant context.
In \cite{nagata-nishida-toda:Segal-Becker for KR},
Nagata, Nishida and Toda prove a version of Segal's splitting
for Real K-theory in the sense of Atiyah \cite{atiyah:KR},
i.e., the K-theory made from complex vector bundles over spaces with involutions,
equipped with a fiberwise conjugate-linear involution.
On underlying non-equivariant spaces this recovers Segal's splitting
for complex K-theory; taking $C_2$-fixed points yields Becker's
splitting for real K-theory.
A version of the Segal--Becker splitting in motivic homotopy theory,
for algebraic K-theory in place of topological K-theory,
is provided in \cite{joshua-pelaez:motivic_SB}.

A different proof of the Real splitting was proposed by Kono \cite{kono:Segal-becker KR},
and soon thereafter extended by Iriye and Kono \cite[Theorem 1']{iriye-kono:Segal-Becker KG} 
to semi-direct products  $\tilde G=G\rtimes_\tau C$ of
a finite group $G$ by a multiplicative involution $\tau$.
However, as I explain in 
Remark \ref{rk:KonoIriye_inconsistency}, the paper \cite{kono:Segal-becker KR}
and the specific result \cite[Theorem 1']{iriye-kono:Segal-Becker KG}
are flawed by overlooking the non-additivity of the transfer construction
in the presence of a Real direction.\medskip

Now we give a more detailed outline of our own results.
Throughout the paper we write $C=\Gal(\mC/\mR)$ for the
Galois group of $\mC$ over $\mR$, a group with two elements.
Our results live in the unstable and stable $C$-global homotopy categories.
In this paper, an {\em augmented Lie group} is a continuous homomorphism
$\alpha\colon G\to C$ from a compact Lie group to the Galois group.
Augmented Lie groups are to $C$-global homotopy theory what compact Lie groups
are to global homotopy theory: every augmented Lie group encodes a way
to assign a $G$-homotopy type to every $C$-global homotopy type;
and restriction along all augmented Lie groups is jointly conservative.
The augmented Lie groups fall into two disjoint classes:
the trivially augmented groups together remember the underlying `non-Real' global homotopy type;
and the surjective augmentations encode the `truly Real' aspects.

The non-equivariant morphism $\Sigma^\infty_+ \mC P^\infty\to K U$
that classifies the tautological line bundle over $\mC P^\infty$ has a particularly nice and
prominent refinement to a morphism of $C$-global ring spectra
\[ \eta\  : \ \Sigma^\infty_+ \bP\ \to\ \bKR\ . \]
The global K-theory spectrum was introduced by Joachim \cite{joachim:coherences},
see also \cite[Construction 6.4.9]{schwede:global}
or Construction \ref{con:KR};
for every compact Lie group $G$, the underlying genuine $G$-spectrum of $\bKR$ represents
$G$-equivariant complex K-theory, see \cite[Theorem 4.4]{joachim:coherences} or \cite[Corollary 6.4.23]{schwede:global}.
The complex-conjugation involution on $\bKR$ was first investigated
by Halladay and Kamel \cite[Section 6]{halladay-kamel}, who show
that the resulting genuine $C$-spectrum models Atiyah's Real K-theory spectrum.
We show in Theorem \ref{thm:KU_deloops_BUP} that $\bKR$ deserves to be called
the {\em Real-global K-theory spectrum}\/:
for every augmented Lie group $\alpha\colon G\to C$,
the genuine $G$-spectrum  $\alpha^*(\bKR)$ represents
$\alpha$-equivariant Real K-theory $KR_\alpha$.

The $C$-global space $\bP$ is a specific global refinement of $\mC P^\infty$,
made from projective spaces of complex vector spaces,
see Construction \ref{con:P}.
The underlying $G$-equivariant homotopy type of $\bP$
is that of the projective space of a complete complex $G$-universe,
with involution by complex conjugation. 
It is a $C$-global classifying space, in the sense of \cite[Construction A.4]{schwede:stiefel},
for the extended circle group $\tilde U(1)=U(1)\rtimes C$
that we shall denote by $\tilde T$ throughout this paper;
so the unreduced suspension spectrum $\Sigma^\infty_+ \bP$
represents the functor $\pi_0^{\tilde T}$
on the $C$-global stable homotopy category, compare \cite[Theorem A.17]{schwede:stiefel}.
The morphism $\eta$ is extremely highly structured, and has a range of marvelous properties.
It is a morphism of ultra-commutative $C$-ring spectra that
sends the universal element in $\pi_0^{\tilde T}(\Sigma^\infty_+\bP)$
to the class of the tautological Real $\tilde T$-representation
in $\pi_0^{\tilde T}(\bKR)\iso RR(\tilde T)$.
As a morphism of ultra-commutative ring spectra, the effect of $\eta$
on equivariant homotopy groups is not only compatible with restriction,
inflations and transfers,
but also with products, multiplicative power operations and norms.
In \cite{schwede:global_Snaith}, I establish a
global refinement and generalization of Snaith's
celebrated theorem \cite{snaith:algebraic cobordism, snaith:localized stable},
saying that $K U$ can be obtained from $\Sigma^\infty_+\mC P^\infty$ by `inverting the Bott class':
the morphism $\eta\colon\Sigma^\infty_+ \bP \to \bKR$
is initial in the $\infty$-category of ultra-commutative ring spectra among morphisms
from $\Sigma^\infty_+ \bP$ that invert a specific family of
representation-graded equivariant homotopy classes in  $\pi_{\nu_n}^{U(n)}(\Sigma^\infty_+\bP)$,
the {\em pre-Bott classes}, for all $n\geq 1$.

The $C$-global space $\bU$ is a specific global refinement of the infinite unitary group,
made from the unitary groups of all hermitian inner product spaces,
see Construction \ref{con:U}.
The underlying $G$-equivariant homotopy type of $\bU$
is that of the unitary group of a complete complex $G$-universe,
with involution by complex conjugation. 
The $C$-global space $\bU$ features in a Real-global refinement of Bott
periodicity, a $C$-global equivalence $\Omega^\sigma\bU\sim\bBUP$,
see Theorem \ref{thm:BUP2OmegaU}.
In Construction \ref{con:SB-splitting} we use the $C$-global stable splitting
of $\Sigma^\infty_+\bU$ from \cite[Theorem 4.10]{schwede:stiefel}
to construct a morphism
\[ d \ : \ \bU\ \to \ \Omega^\bullet(\sh^\sigma(\Sigma^\infty_+\bP)) \]
in the unstable $C$-global homotopy category,
our $\sigma$-deloop of the global Segal--Becker splitting.
Here `$\sh^\sigma $' denotes the shift of an orthogonal $C$-spectrum
by the sign representation, see \cite[Construction 3.1.21]{schwede:global}.
Shifting by $\sigma$ is Real-globally equivalent to suspending by $S^\sigma$,
see \cite[Proposition 3.1.25 (ii)]{schwede:global}.
The splitting property is our first main result, to be proved
as Theorem \ref{thm:splits}:\medskip

{\bf Theorem A.} {\em  The composite 
  \[    \bU\ \xra{\ d\ } \ \Omega^\bullet(\sh^\sigma(\Sigma^\infty_+\bP))
    \ \xra{\Omega^\bullet(\sh^\sigma\eta)}\ \Omega^{\bullet}(\sh^\sigma\bKR)
  \]
  is a $C$-global equivalence.}
\medskip

As we also show, the composite
$\Omega^\bullet(\sh^\sigma\eta)\circ d$ in Theorem A is not just any $C$-global equivalence,
but it coincides with the `preferred infinite delooping' of $\bU$
established in Theorem \ref{thm:U2shKR}.
In fact, all the work in proving Theorem A goes into showing precisely this.
At the heart of the argument is a subtle connection between
the global stable splitting and the preferred delooping of $\bU$,
two features that are a priori unrelated.
As we show in Theorem \ref{thm:annihilate},
the adjoint $\Sigma^\infty\bU\to \sh^\sigma \bKR$ of the preferred infinite delooping
annihilates the higher terms
of the global stable splitting \eqref{eq:global_splitting}.\medskip

The $C$-global space $\bBUP$ is a Real-global refinement of $\mZ\times BU$ and its
involution by complex conjugation, see Construction \ref{con:Gr2BUP}.
By Theorem \ref{thm:BUP_represents}, $\bBUP$ represents Real-equivariant K-theory.
Just as $\mZ\times BU$ is the infinite loop space
of the topological K-theory spectrum $K U$, the global space $\bBUP$
`is' the Real-global infinite loop space of $\bKR$, see Theorem \ref{thm:KU_deloops_BUP}
for the precise statement.
Looping the morphism $d\colon\bU\to\Omega^\bullet(\sh^\sigma(\Sigma^\infty_+\bP))$
by the sign representation
and composing with the Bott periodicity equivalence $\bBUP\sim\Omega^\sigma\bU$
from \eqref{eq:define_gamma} yields another morphism
in the unstable $C$-global homotopy category
\[ c \ : \  \bBUP\  \to \ \Omega^\bullet(\Sigma^\infty_+ \bP ) \ ,\]
the {\em global Segal--Becker splitting} \eqref{eq:define_gSB}.
By design, the morphism $c$ is a Real-global $\sigma$-loop map.
Its splitting property, to be proved as Theorem \ref{thm:split},
is an easy consequence of Theorem A:\medskip

{\bf Theorem B.} {\em  The composite 
  \[     \bBUP \ \xra{\ c\ } \ \Omega^\bullet(\Sigma^\infty_+\bP)
    \ \xra{\Omega^\bullet(\eta)}\ \Omega^{\bullet}(\bKR)  \]
  is a $C$-global equivalence.}\medskip

Much like in Theorem A, the composite 
$\Omega^\bullet(\eta)\circ c$ in Theorem B is not just any $C$-global equivalence,
but it coincides with the `preferred infinite delooping' $\Theta\colon\bBUP\sim\Omega^\bullet(\bKR)$
established in Theorem \ref{thm:KU_deloops_BUP}.\medskip

Morphisms that admit a section tend to admit many different sections.
Our next result justifies that the global Segal--Becker splitting
$c\colon\bBUP\to\Omega^\bullet(\Sigma^\infty_+ \bP )$
is the `correct' splitting to $\Omega^\bullet(\eta)$
by showing that it refines the classical equivariant Segal--Becker splittings,
thereby also explaining the name.
The latter are defined at the level of equivariant cohomology theories;
we review them in Construction \ref{con:classical SB}
and \eqref{eq:K_alpha2P},
and we denote them by $\vartheta_{\alpha,A}$.
We show the following in Theorem \ref{thm:rigidifies}:\medskip

{\bf Theorem C.} {\em 
  For every augmented Lie group  $\alpha\colon G\to C$ and every finite $G$-CW-complex $A$,
  the map $[A,c]^\alpha\colon[A,\bBUP]^\alpha \to[A,\Omega^\bullet(\Sigma^\infty_+ \bP)]^\alpha$
 coincides with   the composite
  \[  [A,\bBUP]^\alpha \ \iso \ KR_\alpha(A) \ \xra{\vartheta_{\alpha,A}}\
    \ [A,\Omega^\bullet(\Sigma^\infty_+ \bP)]^\alpha \ .\]
}\smallskip

The unspecified isomorphism between  the group $[A,\bBUP]^\alpha$
and the Real-equivariant K-group $K R_\alpha(A)$
will be recalled in Construction \ref{con:K_G_via_spc}.
The group $[A,\Omega^\bullet(\Sigma^\infty_+ \bP)]^\alpha$ is isomorphic to
the group of morphisms, in the $G$-equivariant stable homotopy category,
from $\Sigma^\infty_+ A$ to the unreduced suspension spectrum of 
a classifying $G$-space for $\alpha$-equivariant Real line bundles.

\Danger There is a potentially surprising subtlety in
the construction of the equivariant Segal--Becker splitting
for surjectively augmented Lie groups.
In the previous literature such as \cite{crabb:brauer}, the construction
is mostly discussed without any `Real' direction,
i.e., without any augmentations to $C=\Gal(\mC/\mR)$,
and then this subtlety is invisible.
Without the Real direction, we are considering complex $G$-vector bundles,
and the splitting based on transfers is additive for Whitney sum
of vector bundles, see for example \cite[Lemma 2.6]{crabb:brauer}.
The extension from vector bundles to the Grothendieck group is
thus straightforward from the universal property of the latter.
As I show in Remark \ref{rk:nonadditive}, the direct extension of
the Segal--Becker splitting to surjectively augmented
Lie groups is {\em not additive} for the Whitney sum of Real-equivariant vector bundles.
Since additivity fails, it is less obvious how to meaningfully extend
the construction to virtual vector bundles, i.e., the Grothendieck group.
What saves us is that the Segal--Becker splitting enjoys a property
somewhat weaker than additivity, spelled out in Corollary \ref{cor:vartheta_nonadd}.
This `weak additivity' still allows a well-defined extension to the
Grothendieck group, see \eqref{eq:K_alpha2P}.
I consider this non-additivity a feature, and not a bug, of the
Real-equivariant Segal--Becker splitting; it is intimately tied to our
global Segal--Becker splitting being a $\sigma$-loop map,
but not an ordinary loop map, whenever there is an honest Real direction.\medskip

Our Real-global splittings in Theorems B and C imply
or correct the above mentioned splitting results by specialization,
i.e., by applying a suitable forgetful functor from Real-global to $G$-equivariant homotopy theory:
the underlying non-equivariant splitting of Theorem B recovers Segal's and Becker's original result,
and Theorem A even provides a delooping of the splitting.
Forgetting the Real direction and passing to the unstable $G$-homotopy category
for a finite group $G$ 
yields the equivariant splitting of Iriye--Kono \cite{iriye-kono:Segal-Becker KG}.
For general compact Lie groups, we obtain the equivariant Segal--Becker splitting
in Crabb's paper \cite{crabb:brauer};
and again, Theorem A provides a delooping of the equivariant splitting.

Keeping the Real direction and forgetting the equivariance essentially
recovers the theorem of Nagata, Nishida and Toda
\cite[Theorem 1.1]{nagata-nishida-toda:Segal-Becker for KR}.
There are two caveats: for one, the paper \cite{nagata-nishida-toda:Segal-Becker for KR}
works with reduced versions of the Real K-theory and $\Sigma^\infty_+ \bP$-cohomology.
And for the other, the sections we obtain are not additive, see Example \ref{eg:c_not_additive}.
The construction in \cite{nagata-nishida-toda:Segal-Becker for KR}
is sufficiently different from ours so that I do not know
how the splitting of Nagata, Nishida and Toda is related to ours.
For general surjectively augmented Lie groups, our splitting
fixes the issues of \cite[Theorem 1']{iriye-kono:Segal-Becker KG}
insofar as we obtain a natural section, albeit not an additive one;
see Remark \ref{rk:KonoIriye_inconsistency} for more information.
I do not know if {\em additive} natural sections exist in the Real-equivariant
context for surjectively augmented Lie groups;
I have no evidence supporting such an expectation.\medskip

On equivariant homotopy groups, the global Segal--Becker splitting induces
the so-called {\em explicit Brauer induction}
of Boltje \cite{boltje:canonical brauer} and Symonds \cite{symonds-splitting}.
We review this result and the history of explicit Brauer induction
in Remark \ref{rk:Brauer induction}. The following will be
proved as Corollary \ref{cor:c_on_RG}.\medskip

{\bf Theorem D.} {\em 
  For every compact Lie group $G$, the composite
  \[ R(G) \ \iso\  \pi_0^G(\bBUP) \  \xra{\pi_0^G(c)} \
 \pi_0^G(\Sigma^\infty_+ \bP) \ \iso \ \bA(T,G) \]
is the Boltje--Symonds explicit Brauer induction.}\medskip

As an application of the global Segal--Becker splitting,
we construct global rigidifications of the unstable Adams operations on equivariant K-theory.
In \eqref{eq:define_Adams} we define the {\em $n$-th global Adams operation}
\[  \Upsilon^n \ : \ \bBUP \ \to \ \bBUP \ ,\]
for $n\geq 1$. These global Adams operations are morphisms
in the unstable $C$-global homotopy category that arise as Real-global $\sigma$-loop maps,
the deloopings being certain endomorphisms of $\bU$.
The following result, proved as Theorem \ref{thm:Adams},
justifies the name of the global Adams operations:\medskip

{\bf Theorem E.} {\em 
  For every augmented Lie group $\alpha\colon G\to C$,
  every finite $G$-CW-complex $A$ and every $n\geq 1$, the following
  square commutes:
  \[ \xymatrix@C=15mm{
  [A,\bBUP]^\alpha \ar[d]_\iso \ar[r]^-{[A,\Upsilon^n]^\alpha} & [A,\bBUP]^\alpha\ar[d]^\iso \\
      KR_\alpha(A) \ar[r]_-{\psi^n} &      KR_\alpha(A)
    } \]
Here $\psi^n$ is the $n$-th Adams operation on Real-equivariant K-theory.
}\smallskip

This paper concludes with two appendices. Appendix \ref{app:A}
develops some pieces of unstable $C$-global homotopy theory that are not
covered in the existing sources \cite[Appendix A]{schwede:stiefel} and \cite{barrero}.
In Appendix \ref{app:B} we extend various results about
global K-theory to the Real-global context.\medskip

{\bf Conventions.}
Throughout this paper we let $C=\Gal(\mC/\mR)$ denote the
Galois group of $\mC$ over $\mR$.
We use the models of \cite[Appendix A]{schwede:stiefel} and \cite[Appendix A]{barrero}
to represent unstable and stable $C$-global homotopy
types. So $C$-global spaces are represented by orthogonal $C$-spaces,
see Definition \ref{def:orthogonal_C-space},
relative to the notion of $C$-global equivalence introduced in
\cite[Definition A.2]{schwede:stiefel} and \cite[Definition 3.2]{barrero}.
And $C$-global spectra are represented by orthogonal $C$-spectra,
relative to the notion of $C$-global equivalence introduced in
\cite[Definition A.6]{schwede:stiefel}.
Since the non-identity element of the group $C$
always acts by some sort of complex conjugation,
and the $C$-actions in this paper always encode conjugate-linear phenomena, 
we shall also use the term `Real-global' as synonymous for `$C$-global'.
Our global Segal--Becker splitting and its deloop will thus be morphisms
in the unstable Real-global homotopy category, i.e., the localization of the category
of orthogonal $C$-spaces at the class of $C$-global equivalences.

We shall always write $\sigma$ for the {\em sign representation}
of $C$, i.e., the 1-dimensional orthogonal representation on $\mR$
with generator multiplying by $-1$. If $\alpha\colon G\to C$ is an augmented Lie group,
we abuse notation and also write $\sigma$ for the orthogonal $G$-representation
obtained by restriction along $\alpha$, i.e., the elements in the kernel of $\alpha$
act by the identity, and all others act by $-1$.

We will make frequent use of the {\em pre-Euler classes} associated with representations.
If $V$ is an orthogonal representation of a compact Lie group $G$, we write
\begin{equation}\label{eq:preEuler}
  a_V\ \in \pi_{-V}^G(\mS)    
\end{equation}
for the representation-graded $G$-equivariant homotopy class represented by
the inclusion $\{0,\infty\}=S^0\to S^V$ into the one-point compactification of $V$.
We shall use various standard properties
without further notice, such as the compatibility under restriction
along continuous homomorphisms of compact Lie groups, the multiplicativity
property $a_{V\oplus W} = a_V\cdot a_W$, and the fact that $a_V=0$ whenever
$V^G\ne 0$.
\smallskip

{\bf Acknowledgments.}
This paper owes a lot to many discussion with Greg Arone over the course of
almost a decade, and some key steps are based on suggestions of his;
in fact, we had earlier considered making this project a joint paper.
Some of the work for this paper was done while the author spent the summer term 2023 on
sabbatical at Stockholm University, with financial support from the Knut and Alice Wallenberg
Foundation; I would like thank Greg Arone and SU for the hospitality and stimulating
atmosphere during this visit.
I am grateful to Markus Hausmann for several enlightening conversations
about the contents of this paper.

The author acknowledges support by the DFG Schwerpunktprogramm 1786
`Homotopy Theory and Algebraic Geometry' (project ID SCHW 860/1-1)
and by the Hausdorff Center for Mathematics
at the University of Bonn (DFG GZ 2047/1, project ID 390685813).

\section{The \texorpdfstring{$C$}{C}-global stable splitting of \texorpdfstring{$\bU$}{U}}

A key player in this paper is the $C$-global ultra-commutative monoid $\bU$
made from unitary groups,  see \cite[Example 2.37]{schwede:global} or Construction \ref{con:U}.
Our definition of the global Segal--Becker splitting
depends on a $C$-global stable splitting of $\bU$:
In \cite{schwede:stiefel}, I construct certain  morphisms
$s_k\colon\Sigma^\infty (\bGr_k)^{\ad(k)}\to\Sigma^\infty_+\bU$ in the $C$-global stable
homotopy category such that the combined morphism
\begin{equation}\label{eq:global_splitting}
  \sum s_k \ : \ \bigvee_{k\geq 0}\Sigma^\infty (\bGr_k)^{\ad(k)}\ \xra{\ \sim \ } \ \Sigma^\infty_+\bU     \end{equation}
is a $C$-global equivalence, see \cite[Theorem 4.10]{schwede:stiefel}.
The splitting \eqref{eq:global_splitting}
is a global-equivariant refinement of Miller's stable splitting
\cite{miller} of the infinite unitary group.  
In this section we review the splitting,
and then translate it into an interpretation
of the group $\gh{\Sigma^\infty\bU,X}^C$ of stable $C$-global morphisms
in terms of $\ad(k)$-graded $\tilde U(k)$-equivariant
homotopy groups of $X$, see Theorem \ref{thm:gh_from_U}.\medskip

  The {\em extended unitary group}
  \[ \tilde U(k)\ =\ U(k)\rtimes C \]
  is the semi-direct product of the group $C=\Gal(\mC/\mR)$ acting on the unitary group
  \[ U(k)\ = \  \{ A\in M_k(\mC)\colon A\cdot \bar A^t= E_k\} \]
  by coordinatewise complex conjugation.
The orthogonal $C$-space $\bGr_k$ is made from Grassmannians
of complex $k$-planes, with involution by complex conjugation; see Construction \ref{con:Gr_k}.
And $(\bGr_k)^{\ad(k)}$ denotes the global Thom space over $\bGr_k$ associated with the
adjoint representation $\ad(k)$ of the extended unitary group $\tilde U(k)$;
see \cite[Example 3.12]{schwede:stiefel}.

The $C$-global stable splitting morphism $s_k\colon(\bGr_k)^{\ad(k)}\to\Sigma^\infty_+\bU$
ultimately stems from a specific $\tilde U(k)$-equivariant stable splitting of
the `top cell' in $U(k)^{\ad}$.
We review this splitting now, following Crabb's exposition in \cite[page 39]{crabb:U(n) and Omega}.
In \cite{schwede:global} and \cite{schwede:stiefel} we use different conventions
on whether the suspension coordinate in an orthogonal suspension spectrum is
written on the left or right of the argument.
In this paper, we adopt the convention of  \cite{schwede:global},
with the suspension coordinate written on the left;
this entails some minor changes to formulas in \cite{schwede:stiefel},
moving some suspension coordinates to the other side.

\begin{construction}[Splitting the top cell off $U(k)^{\ad}$]
  We write
  \[ \ad(k)\ = \ \{ X\in M_k(\mC)\colon X= -\bar X^t\} \]
  for the $\mR$-vector space of skew-hermitian complex $k\times k$ matrices.
  The unitary group $U(k)$ acts on $\ad(k)$ by conjugation,
  and the group $C$ acts by coordinatewise complex conjugation.
  Together with the euclidean inner product $\td{X,Y}=\trc(\bar X^t\cdot Y)=-\trc(X\cdot Y)$,
  these data make $\ad(k)$ into an orthogonal $\tilde U(k)$-representation;
  this action witnesses $\ad(k)$ as the adjoint representation of $\tilde U(k)$,
  whence the name.

  The {\em Cayley transform} is the $\tilde U(k)$-equivariant open embedding 
  \[  \ad(k)\ \to \ U(k)^{\ad}  \ , \quad X\ \longmapsto \  (X-1) (X+1)^{-1} \]
  onto the subspace of $U(k)$ of those matrices that do not have +1 as an eigenvalue.
  The associated collapse map
  \[  U(k)^{\ad}\ \to \ S^{\ad(k)}    \]
  admits a section in the stable homotopy category of genuine $\tilde U(k)$-spectra, as follows. 
  We write
  \[ \sa(k)\ = \ \{ Z\in M_k(\mC)\colon Z= \bar Z^t\} \]
  for the $\mR$-vector space of hermitian complex $k\times k$ matrices.
  Much like for $\ad(k)$, the unitary group $U(k)$ acts on $\sa(k)$ by conjugation,
  the group $C$ acts by coordinatewise complex conjugation,
  and an invariant inner product is given by
  $\td{Z,Z'}=\trc(\bar Z^t\cdot Z')=\trc(Z\cdot Z')$;
  these data make $\sa(k)$ into another orthogonal $\tilde U(k)$-representation.
  A basic linear algebra fact, sometimes referred to as  `polar decomposition',
  is that the $\tilde U(k)$-equivariant map
  \begin{equation} \label{eq:define_phi_k}
  \phi_k \ : \ \sa(k)\times U(k)^{\ad} \ \to \ M_k(\mC) =\sa(k)\oplus \ad(k)
    \ , \quad (Z,A)\mapsto A\cdot \exp(-Z)       
  \end{equation}
   is an open embedding onto the general linear group $G L_k(\mC)$;
  for a proof, see for example \cite[Proposition B.17]{schwede:stiefel}.
  This open embedding has an associated $\tilde U(k)$-equivariant collapse map
  \begin{equation}\label{eq:t_k_representative}
  t_k\ : \ S^{\sa(k)\oplus\ad(k)} \ \to\ S^{\sa(k)}\sm U(k)^{\ad}_+
  \end{equation}
  that is a stable section to the previous collapse map $U(k)^{\ad}\to S^{\ad(k)}$,
  see the argument after the proof of Theorem 1.8 in \cite[page 39]{crabb:U(n) and Omega},
  or the proof of \cite[Theorem 4.7]{schwede:stiefel}.
\end{construction}

Next we recall how the $\tilde U(k)$-equivariant stable splitting \eqref{eq:t_k_representative}
gives rise to a specific equivariant homotopy class $\tau_k$ in
$\pi_{\ad(k)}^{\tilde U(k)}(\Sigma^\infty_+ \bU)$; this class in turn characterizes the
$C$-global splitting morphism
$s_k\colon\Sigma^\infty(\bGr_k)^{\ad(k)}\allowbreak\to \Sigma^\infty_+ \bU$
by the relation \eqref{eq:define_s_k}.\medskip

Throughout this paper, an {\em augmented Lie group}
is a continuous homomorphism $\alpha\colon G\to C$
from a compact Lie group $G$ to $C=\Gal(\mC/\mR)$.
A recurring example is the extended unitary group,
augented by the projection $\tilde U(k)=U(k)\rtimes C\to C$.
For an orthogonal $C$-spectrum $X$ we will often abbreviate
the equivariant homotopy groups $\pi_*^G(\alpha^*(X))$ to $\pi_*^G(X)$
when the augmentation $\alpha$ is clear from the context.

\begin{construction}[The $C$-global splitting morphism]
  We write $\nu_k$ for the {\em tautological Real representation}
  of the extended unitary group $\tilde U(k)$.
  By definition, this is the vector space $\mC^k$ with tautological $U(k)$-action, and with
  $C$ acting by coordinatewise complex conjugation.
  If $W$ is a hermitian inner product space,
  we write $u W$ for the underlying $\mR$-vector space, endowed with
  the euclidean inner product $\td{v,w}=\text{Re}(v,w)$,
  the real part of the hermitian inner product.
  The map
  \begin{equation} \label{eq:def_zeta^W}
    \zeta^k\ : \ \nu_k \ \to \ \mC\tensor_\mR u (\nu_k) = u(\nu_k)_\mC\ , \quad
    w \longmapsto (1\tensor w  - i \tensor i w)/\sqrt{2}
  \end{equation}
  is a $U(k)$-equivariant $\mC$-linear isometric embedding.
  The map $\zeta^k$ commutes with the complex conjugations,
  which acts diagonally on the target, through complex conjugation on $\mC$ and on $\nu_k$.
  So conjugation by $\zeta^k$ and extension by the identity
  on the orthogonal complement of its image
  is a continuous $\tilde U(k)$-equivariant group monomorphism
  \begin{equation} \label{eq:def_zeta^k}
    \zeta^k_* \ : \ U(k)^{\ad} = U(\nu_k)\ \to \  \bU(u(\nu_k))\ .
    \end{equation}
  Here $U(k)$ acts on $\bU(u(\nu_k))$ via its action on $u(\nu_k)$
  and the functorially of $\bU$; and complex conjugation acts diagonally,
  through its action on $\bU$ and on $\nu_k$.

  The $\tilde U(k)$-equivariant collapse map $t_k$  was defined in \eqref{eq:t_k_representative}.
  In \cite[Construction 4.4]{schwede:stiefel}, we define a class
  \begin{equation}\label{eq:define_tau_k}
    \tau_k \ \in \ \pi_{\ad(k)}^{\tilde U(k)}(\Sigma^\infty_+ \bU)     
  \end{equation}
  (there denoted $\td{t_{k,0}}$)
  as the one represented by the $\tilde U(k)$-map
  \begin{align*}
    S^{\nu_k\oplus\sa(k)\oplus \ad(k)} \
    \xra{S^{\nu_k}\sm t_k} \  &S^{\nu_k\oplus\sa(k)}\sm U(k)^{\ad}_+
  \xra{S^{\nu_k\oplus\sa(k)}\sm(\zeta^k_*)_+} \ S^{\nu_k\oplus\sa(k)} \sm \bU(u(\nu_k))_+\\
      \xra{S^{\nu_k\oplus\sa(k)}\sm\bU(i_1)_+} \  & S^{\nu_k\oplus\sa(k)}\sm \bU(u(\nu_k)\oplus\sa(k))_+
      = (\Sigma^\infty_+\bU)(u(\nu_k)\oplus\sa(k))\ .
  \end{align*}
  Here $i_1\colon u(\nu_k)\to u(\nu_k)\oplus \sa(k)$ is the embedding of the first summand.
  In \cite{schwede:stiefel}, we consider $\tau_k$ as an equivariant
  stable homotopy class of the $k$-th stage of the eigenspace filtration of $\bU$,
  but now we work in the ambient orthogonal space $\bU$.

  The {\em tautological class}
  \[ e_{\tilde U(k),\ad(k)}\ \in \ \pi_{\ad(k)}^{\tilde U(k)}(\Sigma^\infty (\bGr_k)^{\ad(k)})  \]
  is defined in \cite[(A.16)]{schwede:stiefel}.
  By \cite[Theorem A.17 (i)]{schwede:stiefel}, the pair
  $(\Sigma^\infty (\bGr_k)^{\ad(k)},e_{\tilde U(k),\ad(k)})$ represents the functor 
  $\pi_{\ad(k)}^{\tilde U(k)}\colon \GH_C\to \Ab$ on the $C$-global stable homotopy category.
  The $C$-global splitting morphism
  $s_k\colon\allowbreak\Sigma^\infty(\bGr_k)^{\ad(k)}\to \Sigma^\infty_+ \bU$
  is defined in \cite[(4.6)]{schwede:stiefel}
  by the property that it takes the tautological class
  to $\tau_k$, i.e., by the relation
  \begin{equation}\label{eq:define_s_k}
 (s_k)_*(e_{\tilde U(k),\ad(k)})\ = \ \tau_k\ .    
  \end{equation}
\end{construction}

In this paper, we shall mostly work with the reduced suspension spectrum
$\Sigma^\infty\bU$ (as opposed to the unreduced one),
and with the `reduced' version of the class $\tau_k$.
We endow $\bU$ with the intrinsic basepoint
consisting of the multiplicative units.
We write $\bU_+$ for $\bU$ with an additional basepoint added.
This comes with based maps $\bU_+\to S^0$ and  $\bU_+\to \bU$;
the first of these takes $\bU$ to the non-basepoint of $S^0$,
and the second is the identity on $\bU$ and maps the extra basepoint to the
intrinsic basepoint.
We write
\[ \varrho\ :\ \Sigma^\infty_+\bU\ \to\  \Sigma^\infty S^0=\mS \text{\qquad and\qquad}
  q\ :\ \Sigma^\infty_+\bU\ \to\  \Sigma^\infty\bU\]
for the morphisms induced on reduced suspension $C$-spectra.
The combined morphism
\begin{equation}\label{eq:split_off_basepoint}
   (\varrho,q)\ : \ \Sigma^\infty_+\bU\ \xra{\ \sim \ }\ \mS\times(\Sigma^\infty\bU)
\end{equation}
is then a $C$-global equivalence.

The next proposition shows that  for all $k\geq 1$, the class $\tau_k$
defined in \eqref{eq:define_tau_k}
lies in the augmentation ideal, i.e., the kernel of
the homomorphism $\varrho_*\colon \pi_\star^{\tilde U(k)}(\Sigma^\infty_+ \bU)\to
\pi_\star^{\tilde U(k)}(\mS)$.

\begin{prop}\label{prop:tau_augments}
  For every $k\geq 1$, the relation $\varrho_*(\tau_k)=0$
  holds in $\pi_{\ad(k)}^{\tilde U(k)}(\mS)$.
\end{prop}
\begin{proof}
  Every matrix in $\sa(k)$ is unitarily diagonalizable with real eigenvalues.
  We write $\sa^{\geq 0}(k)$ for the closed $\tilde U(k)$-invariant subspace
  of $\sa(k)$ consisting of those hermitian matrices all of whose eigenvalues
  are greater or equal to $0$.  
  The continuous $\tilde U(k)$-equivariant map
  \begin{equation}\label{eq:prehomotopy}
 \sa^{\geq 0}(k)\times [0,\infty)\ \to \ \sa^{\geq 0}(k) \ , \quad
    (Z,t)\ \mapsto \ Z \ +\  t\cdot E_k     
  \end{equation}
  is proper. Indeed, if $Z\in \sa^{\geq 0}(k)$ has eigenvalues $x_1,\dots,x_k$,
  and $t\geq 0$, then $Z+ t E_k$ is again hermitian,
  has eigenvalues $x_1+t,\dots,x_k+t$, and thus also lies in $\sa^{\geq 0}(k)$.
  So
  \[ |\!|(Z,t)|\!|\ = \ \sqrt{x_1^2+\dots+x_k^2+t^2} \ \leq \
    \sqrt{(x_1+t)^2+\dots+(x_k+t)^2} = \ |\!|Z+t E_k|\!| \ . \]
  As a proper map, \eqref{eq:prehomotopy} extends continuously
  to the one-point compactifications
  \[ ( \sa^{\geq 0}(k)\cup\{\infty\})\sm [0,\infty]\ \to \ \sa^{\geq 0}(k)\cup\{\infty\} 
\ ;\]
  this map is a $\tilde U(k)$-equivariant contracting homotopy
  of the space $\sa^{\geq 0}(k)\cup\{\infty\}$.

  The singular value decomposition of complex matrices shows that every
  $X\in M_k(\mC)$ is of the form $X=A\cdot Z$ for some $A\in U(k)$
  and some $Z\in \sa^{\geq 0}(k)$. The matrix $Z$ in this decomposition is unique,
  namely the only matrix in $\sa^{\geq 0}(k)$ such that $Z^2=\bar X^t\cdot X$.
  So the composite
  \begin{equation}\label{eq:comp}
    \sa^{\geq 0}(k)\ \xra{\text{incl}}\  M_k(\mC)\ \to \ U(k)\backslash M_k(\mC)     
  \end{equation}
  is a continuous bijection, where the right hand side denotes the orbit space
  by the $U(k)$-action by left multiplication.
  The map
  \[ M_k(\mC)\ \to \ \sa^{\geq 0}(k)\ , \quad X \ \longmapsto \sqrt{\bar X^t\cdot X} \]
  is continuous and invariant under left multiplication by unitary matrices.
  So it descends to a continuous map on $U(k)\backslash M_k(\mC)$,
  showing that the composite \eqref{eq:comp} is in fact a homeomorphism.

  The conjugation action of $\tilde U(k)$ on $M_k(\mC)$ passes to a
  well-defined action on the orbit space  $U(k)\backslash M_k(\mC)$
  and the map \eqref{eq:comp} is $\tilde U(k)$-equivariant for this induced action
  on the target, and the conjugation action on the source.
  As an equivariant homeomorphism, the map \eqref{eq:comp} extends
  to a $\tilde U(k)$-equivariant homeomorphism between the one-point
  compactifications.
  We showed above that  $\sa^{\geq 0}(k)\cup\{\infty\}$ is equivariantly contractible;
  the upshot is that the one-point compactification
  \[ (U(k)\backslash M_k(\mC))\cup\{\infty\}\ = \  U(k)\backslash S^{M_k(\mC)}\]
  is also $\tilde U(k)$-equivariantly contractible.

  By inspection of definitions, the composite
  \begin{equation} \label{eq:rho_t}
  S^{M_k(\mC)}\  \xra{\ t_k\ }\ S^{\sa(k)}\sm U(k)_+ \ \xra{S^{\sa(k)}\sm \varrho} \ S^{\sa(k)}   
  \end{equation}
  is given by
  \[ X \ \longmapsto \
    \begin{cases}
      -\ln(\sqrt{\bar X^t\cdot X}) & \text{ for $X\in G L_k(\mC)$, and}\\
   \quad   \infty & \text{ otherwise.}
    \end{cases}
  \]
  In particular, this composite is invariant under left translation by elements of $U(k)$,
  so it factors through the orbit space $U(k)\backslash S^{M_k(\mC)}$.
 Since the latter is equivariantly contractible, the composite \eqref{eq:rho_t}
  is $\tilde U(k)$-equivariantly null-homotopic.
  Since the defining representative of the class $\varrho_*(\tau_k)$
  factors through a suspension of the composite \eqref{eq:rho_t},
  this proves that $\varrho_*(\tau_k)=0$.
\end{proof}

We set
\begin{equation}\label{eq:define_omega_k}
  \omega_k\ =\ q_*(\tau_k)\ \in \pi_{\ad(k)}^{\tilde U(k)}(\Sigma^\infty\bU) \ .
\end{equation}
For $k\geq 1$, the class $\tau_k$ lies in the augmentation ideal
by Proposition \ref{prop:tau_augments},
so its projection $\omega_k$ to the reduced suspension spectrum does not lose any information.

The following representability result is a fairly direct consequence of the
stable splitting \eqref{eq:global_splitting}.
We shall use it to construct $C$-global stable morphisms with source $\Sigma^\infty\bU$,
and to check commutativity of diagrams whose initial object is $\Sigma^\infty\bU$.
We let $\gh{-,-}^C$ denote the group of morphisms in the $C$-global stable
homotopy category.

\begin{theorem}\label{thm:gh_from_U}
  For every $C$-global spectrum $X$, the evaluation map
\[  \gh{\Sigma^\infty\bU,X}^C\ \xra{\ \iso \ } \ {\prod}_{k\geq 1}\, \pi_{\ad(k)}^{\tilde U(k)}(X)\ , \quad
  f \ \longmapsto \ (f_*(\omega_k))_{k\geq 1} \]
is an isomorphism.
\end{theorem}
\begin{proof}
  The splitting \eqref{eq:global_splitting} proved in \cite[Theorem 4.10]{schwede:stiefel}
  and the representability property of $(\Sigma^\infty (\bGr_k)^{\ad(k)},\allowbreak e_{\tilde U(k),\ad(k)})$
  together provide the natural isomorphism
  \begin{equation}\label{eq:full_iso}
  \gh{\Sigma^\infty_+\bU,X}^C\ \xra{\ \iso\ } \ {\prod}_{k\geq 0}\, \pi_{\ad(k)}^{\tilde U(k)}(X) \ ,\quad
    f\longmapsto ( f_*(\tau_k))_{k\geq 0}\ .       
  \end{equation}
  The $C$-global equivalence \eqref{eq:split_off_basepoint}
  induces another isomorphism
  \[ \gh{\mS,X}^C\times \gh{\Sigma^\infty\bU,X}^C\ \xra{\ \iso\ } \
    \gh{\Sigma^\infty_+\bU,X}^C\ ,\quad (a,b)\ \longmapsto\ a\circ\varrho + b\circ q\ .\]
  The morphism $\varrho\colon\Sigma^\infty_+\bU\to \mS$ sends the class $\tau_0$
  to $1\in \pi_0(\mS)$, so the composite
  \[ \gh{\mS,X}^C\ \xra{\ \varrho^*\ } \
    \gh{\Sigma^\infty_+\bU,X}^C\ \xra{f\mapsto f_*(\tau_0)} \ \pi_0^C(X) \]
  is an isomorphism. Moreover, $q_*(\tau_0)=0$ and  $q_*(\tau_k)=\omega_k$ for $k\geq 1$,
  so the isomorphism \eqref{eq:full_iso} restricts to an isomorphism
  as in the statement of the theorem, where now the factor indexed by $k=0$ is omitted.
\end{proof}

We shall also need a certain relation between the
class $\tau_k$ and the $k$-fold exterior product of copies of $\tau_1$,
see Theorem \ref{thm:Phi-relations-T^k}.
The ultra-commutative multiplication on $\bU$ induces a
$C$-global ring spectrum structure on the unreduced suspension spectrum $\Sigma^\infty_+\bU$.
This, in turn, induces products on the equivariant homotopy groups of  $\Sigma^\infty_+\bU$.
We set
\[ \tau_1\times\dots\times \tau_1\ = \ p_1^*(\tau_1)\cdot\ldots\cdot p_k^*(\tau_1) \ \in \
\pi_{k \sigma}^{\tilde T^k}(\Sigma^\infty_+\bU)\ ,\]
where $p_i\colon \tilde T^k\to\tilde T$ is the projection to the $i$-th factor.
We have also identified the sign representation
with $\ad(1)$ by sending $x\in\sigma$ to $i\cdot x\in\ad(1)=i\cdot\mR$.

\begin{construction}
  We define the diagonal embedding
  \[ \Delta \ : \ \mC^k \ \to \ M_k(\mC)\text{\qquad by\qquad}
    \Delta(z_1,\dots,z_k)\ =\  \begin{pmatrix}
      z_1 & 0 & 0 & \dots & 0 \\
      0 & z_2 & 0 & \dots & 0 \\
      \vdots & \ddots &\ddots  & \ddots & \vdots \\
      0 & \dots & 0 &  z_{k-1} & 0  \\
        0 & \dots & 0 & 0 &  z_k
    \end{pmatrix}\ .\]
  Whenever convenient we use $\Delta$ to identify
  $T^k$ with its image in $U(k)$, the diagonal maximal torus.
  The diagonal embedding extends to an embedding of the semi-direct products
  by complex conjugation,
  \[ \tilde T^k\ = \ T^k\rtimes C \ \to \  U(k)\rtimes C \ = \ \tilde U(k) \ .\]

We let $D\subset M_k(\mC)$ denote the $\tilde T^k$-invariant
$\mC$-subspace of lower subdiagonal matrices, i.e., of the form:
\begin{equation}\begin{aligned} \label{eq:D}
  \begin{pmatrix}
    0 & 0 & 0 & \dots & 0  \\
    z_{2,1} & 0 & 0 & \dots & 0 \\
    z_{3,1} & z_{3,2} &0 & \dots & 0 \\
    \vdots & \vdots &\ddots & \ddots & \vdots \\
      z_{k,1} & z_{k,2} &\dots & z_{k,k-1} & 0
    \end{pmatrix}     
    \end{aligned}  \end{equation}
  The maps
  \begin{alignat}{3}\label{eq:D_and_ad_sa}
    k\sigma\oplus D \ &\to \ \ad(k) \ , &\quad
                         (x_1,\dots,x_k; A) \ &\longmapsto \ &
                        \Delta(i x_1,\dots, i x_k)\  &+\  A - \bar A^t \\
    \mR^k\oplus D \ &\to \ \sa(k)\ ,   &\quad
                         (y_1,\dots,y_k; A) \ &\longmapsto \ &
                         \Delta(y_1,\dots, y_k)\ \  &+\  A + \bar A^t \nonumber
  \end{alignat}
  are $\mR$-linear and $\tilde T^k$-equivariant isomorphisms.
  The first one identifies $D$ with the subspace of off-diagonal matrices
  $\ad(k)-\ad(k)^{T^k}$ in $\ad(k)$, as orthogonal $\tilde T^k$-representations.
  And the second one identifies $D$ with the subspace of off-diagonal matrices
  $\sa(k)-\sa(k)^{T^k}$ in $\sa(k)$.
\end{construction}

The class $\tau_k$  in  $\pi_{\ad(k)}^{\tilde U(k)}(\Sigma^\infty_+ \bU)$
was defined in \eqref{eq:define_tau_k}.
And $a_D\in\pi_{-D}^{\tilde T^k}(\mS)$ denote the pre-Euler class \eqref{eq:preEuler}
of the $\tilde T^k$-representation $D$.

\begin{theorem}\label{thm:Phi-relations-T^k}
  The relation 
  \[ 
   a_D^2\cdot \res^{\tilde U(k)}_{\tilde T^k}(\tau_k)\ = \
    a_D\cdot(\tau_1\times\dots\times \tau_1)
  \]
  holds in the group $\pi^{\tilde T^k}_{k\sigma-D}(\Sigma^\infty_+ \bU)$.
\end{theorem}
\begin{proof}
  The diagonal map $\Delta$ identifies $\ad(1)^k$ with $\ad(k)^{T^k}$,
  and it identifies $\sa(1)^k$ with $\sa(k)^{T^k}$.
  We denote all these restrictions by $\Delta$, too.
  The following diagram commutes:
  \[ \xymatrix@C=15mm{  
       (\sa(1)\times U(1))^k\ar[r]^-{(\phi_1)^k} \ar[d]_{\text{shuffle}}^\iso &
       (\sa(1)\oplus \ad(1))^k\ar[d]^{\text{shuffle}}\\
       \sa(1)^k\times U(1)^k \ar[d]_{\Delta\times\Delta}&
       \sa(1)^k\oplus \ad(1)^k\ar[d]^{\Delta\times\Delta}\\
        \sa(k)\times U(k)^{\ad}\ar[r]_-{\phi_k} &
      \sa(k)\oplus \ad(k)      } \]
  So the diagram of collapse maps associated with the horizontal open embeddings commutes, too.
  Hence also the left part of the following diagram of based continuous $\tilde T^k$-maps
  commutes:
  \[ \xymatrix@C=13mm{  
    S^{\nu_k}\sm (S^{\sa(1)\oplus\ad(1)})^{\sm k}\ar[r]^-{S^{\nu_k}\sm t_1^{\sm k}}
    \ar[d]_{\text{shuffle}}^\iso &
      S^{\nu_k}\sm (S^{\sa(1)}\sm U(1)_+)^{\sm k}\ar[d]_\iso^{\text{shuffle}}
      \ar[rr]^-{S^{\nu_k}\sm (S^{\sa(1)}\sm(\zeta^1_*)_+)^{\sm k}} &&
      S^{\nu_k}\sm (S^{\sa(1)}\sm \bU(u(\nu_1))_+)^{\sm k}\ar[d]^{\text{shuffle}}\\
         S^{\nu_k\oplus \sa(1)^k\oplus\ad(1)^k}\ar[d]_{S^{\nu_k}\sm\Delta\sm\Delta} &
         S^{\nu_k}\sm S^{\sa(1)^k}\sm U(1)^k_+\ar[rr]^-{S^{\nu_k}\sm S^{\sa(1)^k}\sm(\zeta^1_*)^k_+}\ar[d]^{S^{\nu_k}\sm\Delta\sm\Delta_+}
       &&  S^{\nu_k}\sm S^{\sa(1)^k}\sm \bU(u(\nu_1))^k_+ \ar[d]^{S^{\nu_k}\sm\Delta\sm\mu^{(k)}_+} \\
      S^{\nu_k\oplus\sa(k)\oplus\ad(k)}\ar[r]^-{S^{\nu_k}\sm t_k}
      & S^{\nu_k\oplus \sa(k)}\sm U(k)^{\ad}_+\ar[d]_{S^{\nu_k\oplus\sa(k)}\sm (\bU(i_1)\circ\zeta^k_*)_+}
      \ar[rr]_-{S^{\nu_k\oplus\sa(k)}\sm (\zeta^k_*)_+}&&
      S^{\nu_k\oplus\sa(k)}\sm \bU(u(\nu_k))_+ \ar[d]^{S^{\nu_k\oplus\sa(k)}\sm\bU(i_1)_+}\\
      &(\Sigma^\infty_+\bU)(u(\nu_k)\oplus\sa(k)) \ar@{=}[rr]
    && S^{\nu_k\oplus\sa(k)}\sm \bU(u(\nu_k)\oplus \sa(k))_+ }
  \]
 The iterated multiplication morphism 
  \[ \mu^{(k)}\ : \ \bU(u(\nu_1))\times\dots\times \bU(u(\nu_1))\ \to \
    \bU(u(\nu_1)\oplus\dots\oplus u(\nu_1)) = \bU(u(\nu_k))\]
  of $\bU$ is given by orthogonal direct sum.
  So it participates in a commutative diagram:
  \[ \xymatrix@C=20mm{
      U(1)\times\dots\times U(1) \ar[d]_{\Delta}\ar[r]^-{\zeta^1_*\times\dots\times\zeta^1_*}&
      \bU(u(\nu_1))\times\dots\times \bU(u(\nu_1))
      \ar[d]^{\mu^{(k)}} \\
       U(k)^{\ad}\ar[r]_-{\zeta^k_*}  & \bU(u(\nu_k))
    } \]
  This implies the commutativity of the middle right part of the above diagram.
  
  Now we can wrap up.
  Under the isomorphisms \eqref{eq:D_and_ad_sa}, the maps
   $\Delta\colon S^{\sa(1)^k}\to S^{\sa(k)}$ and $\Delta\colon S^{\ad(1)^k}\to S^{\ad(k)}$
  both represents the pre-Euler class $a_D$.
  The clockwise composite  in the above commutative diagram
  thus represents the class $a_D\cdot(\tau_1\times\dots\times\tau_1)$.
  And the counter clockwise composite represents the class
  $a_D^2\cdot \res^{\tilde U(k)}_{\tilde T^k}(\tau_k)$,
  so the diagram witnesses the desired relation.
\end{proof}

\section{Relations among equivariant stable homotopy classes in \texorpdfstring{$\bU$}{U}}

The purpose of this section is to establishing some crucial relations between
certain families of $\tilde U(k)$-equivariant homotopy classes of
the $C$-global spectrum $\Sigma^\infty_+\bU$, namely the classes
$\tau_k$ defined in \eqref{eq:define_tau_k},
their reduced cousins $\omega_k$ from \eqref{eq:define_omega_k},
and certain classes $v_k$ and $u_k$ that we introduce in \eqref{eq:define_v_k}
and \eqref{eq:define_u_k} below.
The main result here is the formula of Theorem \ref{thm:omega_vs_u},
which will eventually be used to identify the effect
of our global Segal--Becker splitting on equivariant homotopy groups
and on equivariant cohomology theories.

\begin{construction}
For $k\geq 0$, the diagonal matrix $-E_k=\Delta(-1,\dots,-1)$
is a $\tilde U(k)$-fixed point of $U(k)^{\ad}$.
Thus 
\[ \zeta^k_*(-E_k)\ \in \ \bU(u(\nu_k)) \]
is again a $\tilde U(k)$-fixed point, where 
$\zeta^k_* \colon U(k)^{\ad} \to \bU(u(\nu_k))$ was defined in \eqref{eq:def_zeta^k}.
We write
\begin{equation} \label{eq:define_v_k}
 v_k \ \in \pi_0^{\tilde U(k)}(\Sigma^\infty_+ \bU)   
\end{equation}
for its stabilization, i.e., the class represented by the $\tilde U(k)$-equivariant map
\[ S^{\nu_k}\ \xra{-\sm  \zeta^k_*(-E_k)}\ S^{\nu_k}\sm \bU(u(\nu_k))_+ \ = \ (\Sigma^\infty_+\bU)(u(\nu_k)) \ . \]
\end{construction}

For $0\leq j\leq k$, we write $U(k-j,j)$ for the block subgroup of $U(k)$,
consisting of the matrices of the form
$\left(  \begin{smallmatrix} A & 0 \\ 0 & B \end{smallmatrix}\right)$
for $(A,B)\in U(k-j)\times U(j)$.
We write $\tilde U(k-j,j)=U(k-j,j)\rtimes C$ for the
semi-direct product with $C$ acting by coordinatewise complex conjugation.
The class $\tau_j$ was defined in \eqref{eq:define_tau_k},
and $a_{\ad(j)}\in\pi_{-\ad(j)}^{\tilde U(j)}(\mS)$ is the pre-Euler class
\eqref{eq:preEuler} of the adjoint representation.

\begin{theorem}\label{thm:v_intermsof_tau}
  For every $k\geq 1$, the relation
  \[ v_k\ = \
    \sum_{0\leq j\leq k}  \tr_{\tilde U(k-j,j)}^{\tilde U(k)}(1\times(a_{\ad(j)}\cdot\tau_j))    \]
  holds in $\pi_0^{\tilde U(k)}(\Sigma^\infty_+ \bU)$.
\end{theorem}
\begin{proof}
  The map $\phi_k$ defined in \eqref{eq:define_phi_k}
  is an open embedding with image $G L_k(\mC)$.
  Every matrix $Y$ in $\sa(k)\cap G L_k(\mC)$ is unitarily diagonalizable with
  non-zero real eigenvalues.
  If $i$ is the number of positive eigenvalues of $Y$, counted with multiplicity,
  it is thus of the form
  \begin{equation}\label{eq:jth}
  Y\ = \ B\cdot\left(  \begin{smallmatrix} \exp(-Z) & 0 \\ 0 & -\exp(-Z') \end{smallmatrix}\right)\cdot B^{-1}
        \end{equation}
  for some $B\in U(k)$ and $(Z,Z')\in \sa(i)\times\sa(k-i)$.
  Moreover, this representation is unique up to changing $(B,Z,Z')$
  to $(B\cdot (A\oplus A')^{-1}, {^A Z}, {^{A'}Z'})$
  for some $(A,A')\in U(i)\times U(k-i)$. 
  This shows that the map
  \[ g_i \ :\ \tilde U(k)\times_{\tilde U(i,k-i)} (\sa(i)\times\sa(k-i))\ \to \ \sa(k)\cap G L_k(\mC)\]
  sending the equivalence class $[B;Z,Z']$ to \eqref{eq:jth}
  is a homeomorphism onto the open and closed subspace 
of matrices with exactly $i$ positive eigenvalues.
Altogether, this shows that the
following commutative square of $\tilde U(k)$-equivariant maps is a pullback:
  \[ \xymatrix{
      \coprod_{0\leq i\leq k}\, \tilde U(k)\times_{\tilde U(i,k-i)} \sa(i)\times\sa(k-i)\ar[rr]^-{\coprod g_i}\ar[d] &&
      \sa(k)\ar[d]^{(-,0)}\\
      \sa(k)\times U(k)^{\ad} \ar[r]_-{\phi_k}^-\iso & GL_k(\mC) \ar[r] & \sa(k)\oplus \ad(k)  } \]
The left vertical map is given on the $i$-th summand by
the $\tilde U(k)$-equivariant extension of the map
\[ \sa(i)\times \sa(k-i)\ \to \ \sa(k)\times U(k)^{\ad}\ ,\quad
  (Z,Z')\longmapsto \
  \left(\left(\begin{smallmatrix} Z & 0 \\ 0 & Z'\end{smallmatrix}\right),
  \left(  \begin{smallmatrix} E_i & 0 \\ 0 &  -E_{k-i}\end{smallmatrix}\right)\right)
  \ .\]
The pullback yields a commutative diagram for the collapse maps associated
with the horizontal open embeddings; after smashing with $S^{\nu_k}$,
this becomes the commutative left part of the following diagram:
    \[ \xymatrix{
      S^{\nu_k\oplus\sa(k)}\ar[d]_{-\sm 0}\ar[r]^-{S^{\nu_k}\sm \bigvee g_i^\natural}\ar[d] &
      S^{\nu_k}\sm \left(\bigvee_{0\leq i\leq k} \tilde U(k)\ltimes_{\tilde U(i,k-i)} S^{\sa(i)\oplus\sa(k-i)}\right)\ar[d] \ar@/^2pc/[dr]  \\
      S^{\nu_k\oplus\sa(k)\oplus \ad(k)}\ar[r]_-{S^{\nu_k}\sm t_k} & S^{\nu_k\oplus\sa(k)}\sm U(k)^{\ad}_+\ar[r]_-{S^{\nu_k\oplus\sa(k)}\sm(\bU(i_1)\circ \zeta^k_*)_+}
      & S^{\nu_k\oplus\sa(k)}\sm \bU(u(\nu_k)\oplus\sa(k))_+}
   \]
The counter clockwise composite represents the class $a_{\ad(k)}\cdot \tau_k$.
So the diagram witnesses the relation
\[ a_{\ad(k)}\cdot \tau_k\ = \ \sum_{0\leq i\leq k} f_i\ , \]
where $f_i$ is represented by the suspension by $S^{\nu_k}$ of the composite
\begin{align*}
  S^{\sa(k)}\
  &\xra{\ g_i^\natural\ } \ \tilde U(k)\ltimes_{\tilde U(i,k-i)} S^{\sa(i)\oplus\sa(k-i)}\
  \xra{\ \xi^\flat_i\ } \ S^{\sa(k)}\sm \bU(u(\nu_k)\oplus\sa(k))_+ \ .
\end{align*}
The second map $\xi^\flat_i$ is the $\tilde U(k)$-equivariant extension of
\begin{align*}
\xi_i \ : \    S^{\sa(i)\oplus\sa(k-i)}\ &\to \qquad
  S^{\sa(k)}\sm \bU(u(\nu_k)\oplus\sa(k))_+ \\
  (Z,Z')\quad &\longmapsto \
 \left(  \begin{smallmatrix} Z & 0 \\ 0 & Z'
 \end{smallmatrix}\right)\sm (\bU(i_1)\circ\zeta^k_*)  \left(\begin{smallmatrix} E_i & 0 \\ 0 & -E_{k-i}\end{smallmatrix}\right) \ .
\end{align*}
The collapse map $g_i^\natural$ associated with $g_i$
is closely related to, but different from, the collapse map
that features in the definition of the transfer $\tr_{\tilde U(i,k-i)}^{\tilde U(k)}$.
The latter collapse map, discussed for example in  \cite[(3.2.10)]{schwede:global},
is based on a different open embedding, namely the $\tilde U(k)$-equivariant extension
of the map
\[ \sa(i)\times\sa(k-i) \ \to \ \sa(k)\ , \quad (Z,Z')\ \longmapsto \
  \left(\begin{smallmatrix} E_i + Z & 0 \\ 0 & -E_{k-i}+ Z'\end{smallmatrix}\right) \]
restricted to a small open disc around $(0,0)$.
The differential of this second embedding at $(0,0)$ is the block embedding
$\sa(i)\times\sa(k-i)\to \sa(k)$, while
the differential of the embedding $g_i$ at $(0,0)$ is the map
\[ \sa(i)\times\sa(k-i)\ \to\ \sa(k)\ , \quad (Z,Z)\ \longmapsto
  \left(\begin{smallmatrix} -Z & 0 \\ 0 &  Z'\end{smallmatrix}\right)\ . \]
The differentials of these two embeddings differ by the sign in the first block.
We write $\epsilon_i\in\pi_0^{\tilde U(i)}(\mS)$ for the unit represented by
the map $S^{\sa(i)}\to S^{\sa(i)}$, $Z\mapsto -Z$.
Then the composite $f_i$ differs from the transfer of $1\times v_{k-i}$
by multiplication by the equivariant homotopy class  $\epsilon_i\times 1$
in $\pi_0^{\tilde U(i,k-i)}(\mS)$.
In other words:
\[ f_i\ =\ \tr_{\tilde U(i,k-i)}^{\tilde U(k)}\left(\epsilon_i\times v_{k-i} \right) \ .\]
Altogether we have thus shown the relation
\begin{equation}\label{eq:eq:tau_intermsof_v}
 a_{\ad(k)}\cdot \tau_k\ = \
    \sum_{0\leq i\leq k}  \tr_{\tilde U(i,k-i)}^{\tilde U(k)}(\epsilon_i\times v_{k-i})    
\end{equation}
in $\pi_0^{\tilde U(k)}(\Sigma^\infty_+ \bU)$.

Now we are only a few algebraic manipulations away from the desired formula.
We consider the morphism of orthogonal $C$-spectra
$\varrho\colon \Sigma^\infty_+ \bU\to \mS$
induced by the based map $\bU_+\to S^0$ that collapses $\bU$ to the non-basepoint.
It satisfies $\varrho_*(v_{k-i})=1$ in $\pi_0^{\tilde U(i,k-i)}(\mS)$. So
\[
\varrho_*\left( \epsilon_i\times v_{k-i}\right)
\ = \ \epsilon_i\times \varrho_*(v_{k-i})\ =\ \epsilon_i\times 1\ .
\]
Because $\varrho_*(\tau_k)=0$ by Proposition \ref{prop:tau_augments},
applying $\varrho_*$ to \eqref{eq:eq:tau_intermsof_v} yields
\begin{equation}\label{cor:sum_eps_vanish}
    \sum_{0\leq i\leq k}  \tr_{\tilde U(i,k-i)}^{\tilde U(k)}(\epsilon_i\times 1)\ 
  = \  \sum_{0\leq i\leq k}  \tr_{\tilde U(i,k-i)}^{\tilde U(k)}
    \left(\varrho_*\left(\epsilon_i\times v_{k-i}\right)\right)\
 = \ \varrho_*( a_{\ad(k)}\cdot \tau_k)\ =\ 0 
    \end{equation}
whenever $k\geq 1$. Now we deduce
  \begin{align*}
    \sum_{0\leq j\leq k}  \tr_{\tilde U(k-j,j)}^{\tilde U(k)}&(1\times (a_{\ad(j)}\cdot\tau_j))\
   = _{\eqref{eq:eq:tau_intermsof_v}} \     \sum_{0\leq i\leq j\leq k}  \tr_{\tilde U(k-j,j)}^{\tilde U(k)}(1\times\tr_{\tilde U(i,j-i)}^{\tilde U(j)}(\epsilon_i\times v_{j-i}))\\
    &= \     \sum_{0\leq i\leq j\leq k}  \tr_{\tilde U(i,k-j,j-i)}^{\tilde U(k)}(\epsilon_i\times 1\times v_{j-i})\\
                                    &= \  \sum_{0\leq d\leq k}\   \sum_{0\leq i\leq k-d}  \tr_{\tilde U(k-d,d)}^{\tilde U(k)}(\tr_{\tilde U(i,k-d-i)}^{\tilde U(k-d)}(\epsilon_i\times 1)\times v_d)\
   =  _{\eqref{cor:sum_eps_vanish}} \ v_k\ .
  \end{align*}
  The third equation is the variable substitution $d=j-i$.
\end{proof}

We write
  \begin{equation}\label{eq:define_Cayley}
    \mathfrak c\ : \  S^\sigma \ \xra{\ \iso\ } \  U(1)\ , \quad \mathfrak c(x)=(x+i)(x-i)^{-1} 
  \end{equation}
  for the Cayley transform; it is $C$-equivariant for the sign action on the source,
  and complex conjugation on the target.
We identify the sign representation
with $\ad(1)$ by sending $x\in\sigma$ to $i\cdot x\in\ad(1)=i\cdot\mR$.
We have $\sa(1)=\mR$ with trivial $\tilde T$-action.
Then the collapse map $t_1\colon S^{\sa(1)\oplus\ad(1)}\to S^{\sa(1)}\sm U(1)_+$
becomes a map $t_1\colon S^{1\oplus\sigma}\to S^1\sm U(1)_+$.

\begin{prop}\label{prop:t_1andCayley}
  The composite $(S^1\sm q)\circ t_1\colon S^{1\oplus\sigma}\to S^1\sm U(1)$
  is $C$-equivariantly homotopic
  to the suspension of the Cayley transform \eqref{eq:define_Cayley}.
\end{prop}
\begin{proof}
  The Cayley transform is a homeomorphism, so we may show that the composite
  \begin{equation}\label{eq:a_composite}
   S^{1\oplus\sigma} \ \xra{\ t_1\ } \ S^1\sm U(1)_+\ \xra{S^1\sm q}\
    S^1\sm U(1) \ \xra[\iso]{S^1\sm \mathfrak c^{-1}} \ S^{1\oplus\sigma}     
  \end{equation}
  is $C$-equivariantly homotopic to the identity.
  This is the case if and only if the underlying non-equivariant
  morphism and the restriction to $C$-fixed points are homotopic to the respective identity maps.

  The restriction of \eqref{eq:a_composite} to $C$-fixed points is the composite
  \begin{equation}\label{eq:fixed_composite}
    S^1 \ \xra{(t_1)^C}\ S^1\sm \{\pm 1\}_+ \ \xra{S^1\sm q} \ S^1\sm \{\pm 1\}\
    \xra[\iso]{S^1\sm (\mathfrak c^{-1})^C} \ S^1 \ .    
  \end{equation}
  Expanding all formulas shows that this composite 
  collapses the contractible subset $[0,\infty]$ of $S^1$
  to the basepoint and is given on $\mR_{<0}$ by the formula
  \[ f \ : \ \mR_{<0}\ \to \ \mR \ , \quad f(y)\ =\  -\ln(-y) \ .\]
  So the composite \eqref{eq:fixed_composite} is indeed homotopic to the identity.

  The map \eqref{eq:a_composite} itself
  collapses the contractible subset $[0,\infty]\times\{0\}$ of $S^2$
  to the basepoint and factors through a homeomorphism
  \[  S^2/ ([0,\infty]\times\{0\}) \ \iso \ S^2\ .\]
  So the composite \eqref{eq:a_composite} is a non-equivariant homotopy equivalence.
  Expanding all formulas shows that the composite \eqref{eq:a_composite}
  is given on $\mR\times\mR_{>0}$ by 
  \[ F\ : \ \mR\times\mR_{>0}\ \to \ \mR^2 \ , \quad
          F(x,y) \ = \ \left(-\ln(\sqrt{x^2+y^2}),\ x/y+\sqrt{(x/y)^2+1} \right)\ .\]
  So \eqref{eq:a_composite} fixes the point $(0,1)$,
  and it is smooth near $(0,1)$ with differential 
  $D_{(0,1)}F = \left(\begin{smallmatrix} 0 & -1\\ 1 &\ 0 \end{smallmatrix}\right)$.
  Since this differential has determinant $1$,
  the composite \eqref{eq:a_composite} is indeed
  non-equivariantly homotopic to the identity.
  This concludes the proof of the claim that
  $(S^1\sm q)\circ t_1$ is $C$-equivariantly homotopic to $S^1\sm\mathfrak c$.
\end{proof}

  For $k\geq 1$, we define a class 
  \begin{equation}\label{eq:define_u_k}
    u_k\ \in\ \pi_\sigma^{\tilde U(k)}(\Sigma^\infty\bU)   
  \end{equation}
  by stabilizing the class in $\pi_\sigma^{\tilde U(k)}(\bU,1)$ represented by the composite
  \begin{equation}\label{eq:delta_k}
  \delta_k \ : \ S^\sigma \ \xra[\iso]{\ \mathfrak c\ }\  U(1)\ \xra{\ \partial\ }\
    U(k)^{\ad}\ \xra{\zeta^k_*}\     \bU(u(\nu_k)) \ .  
  \end{equation}
    The map $\partial\colon U(1)\to U(k)$ is the diagonal map sending
  an element of $U(1)$ to the constant diagonal matrix in $U(k)$, and $\zeta^k_*$
  was defined in \eqref{eq:def_zeta^k}.
  And `stabilizing' means that $u_k$ is represented by the composite
  \[  S^{\nu_k\oplus \sigma} \ \xra{S^{\nu_k}\sm\delta_k}\   S^{\nu_k}\sm \bU(u(\nu_k))
    \ = \ (\Sigma^\infty_+\bU)(u(\nu_k))\ .\]
  The unstable representative $\delta_k$ of $u_k$ satisfies $\delta(0)=\zeta_*^k(-E_k)$,
  which is the unstable representative for the class $v_k$ defined in \eqref{eq:define_v_k}.
  So comparison of the definitions shows the relation
  \begin{equation}\label{eq:u_vs_v}
    a_\sigma\cdot u_k \ = \ q_*(v_k)
  \end{equation}
  in the group $\pi_0^{\tilde U(k)}(\Sigma^\infty\bU)$,
  where $a_\sigma$ is the pre-Euler class \eqref{eq:preEuler} of the sign representation.

We write $\overline\ad(k)$ for the {\em reduced adjoint representation}
of $\tilde U(k)$, i.e., the $\tilde U(k)$-subrepresentation
consisting of those matrices in $\ad(k)$ whose trace is 0.
Similarly,
\[ \overline\sa(k)\ = \ \{Z\in \sa(k)\colon \trc(Z)=0\} \ .\]
The classes $a_{\overline\sa(k)}$ and $a_{\overline\ad(k)}$
are the associated pre-Euler classes \eqref{eq:preEuler}.
The class $\omega_k$ was defined in \eqref{eq:define_omega_k}.

\begin{theorem}\label{thm:top_Phi-relations}
  For every $k\geq 1$, the relation 
  \[ 
   a_{\overline\sa(k)}\cdot a_{\overline\ad(k)}\cdot\omega_k\ = \
   a_{\overline\sa(k)}\cdot u_k  \]
  holds in the group $\pi^{\tilde U(k)}_{\sigma-\overline\sa(k)}(\Sigma^\infty\bU)$.
\end{theorem}
\begin{proof}
  The map $\partial\colon\mC\to M_k(\mC)$, $\partial(x)=\Delta(x,\dots,x)$
  takes $U(1)$ isomorphically onto the center of $U(k)$,
  it identifies $\ad(1)$ with $\ad(k)^{U(k)}$,
  and it identifies $\sa(1)$ with $\sa(k)^{U(k)}$.
  We denote all these restrictions by $\partial$, too.
  The following diagram commutes:
  \[ \xymatrix@C=15mm{  
       \sa(1)\times U(1)\ar[r]^-{\phi_1} \ar[d]_{\partial\times\partial}&
       \sa(1)\oplus \ad(1) \ar[d]^{\partial\times\partial}\\
        \sa(k)\times U(k)^{\ad}\ar[r]_-{\phi_k} &
      \sa(k)\oplus \ad(k)      } \]
  So the diagram of collapse maps associated with the horizontal open embeddings commutes, too.
  Thus the left part of the following diagram of based continuous $\tilde U(k)$-maps
  commutes:
  \[ \xymatrix@C=20mm{  
         S^{\nu_k\oplus \sa(1)\oplus\ad(1)}\ar[d]_{S^{\nu_k}\sm\partial\sm\partial} \ar[r]^-{S^{\nu_k}\sm t_1}&
         S^{\nu_k\oplus\sa(1)}\sm U(1)_+\ar[r]^-{S^{\nu_k\oplus\sa(1)}\sm(\zeta^k_*\circ\partial\circ q)}\ar[d]^{S^{\nu_k}\sm\partial\sm\partial_+}
         &  S^{\nu_k\oplus\sa(1)}\sm \bU(u(\nu_k))\ar[d]^{S^{\nu_k}\sm\partial\sm\bU(u(\nu_k))} \\
      S^{\nu_k\oplus\sa(k)\oplus\ad(k)}\ar[r]^-{S^{\nu_k}\sm t_k}
      & S^{\nu_k\oplus \sa(k)}\sm U(k)^{\ad}_+\ar[d]_{S^{\nu_k\oplus\sa(k)}\sm (\bU(i_1)\circ\zeta^k_*\circ q)}
      \ar[r]_-{S^{\nu_k\oplus\sa(k)}\sm (\zeta^k_*\circ q)}&
      S^{\nu_k\oplus\sa(k)}\sm \bU(u(\nu_k))\ar[d]^{S^{\nu_k\oplus\sa(k)}\sm\bU(i_1)}\\
      &(\Sigma^\infty\bU)(u(\nu_k)\oplus\sa(k)) \ar@{=}[r]
    & S^{\nu_k\oplus\sa(k)}\sm \bU(u(\nu_k)\oplus \sa(k)) }
  \]
  The map $\partial\colon S^{\sa(1)}\to S^{\sa(k)}$
  represents the pre-Euler class $a_{\overline\sa(k)}$.
  Proposition \ref{prop:t_1andCayley} shows that
  the composite $(S^1\sm q)\circ t_1\colon S^{1\oplus\sigma}\to S^1\sm U(1)$
  is $C$-equivariantly homotopic
  to the suspension of the Cayley transform \eqref{eq:define_Cayley}.
  So the clockwise composite  in the above commutative diagram
  represents the class $a_{\overline\sa(k)}\cdot u_k$.
  And the counter clockwise composite represents the class
  $a_{\overline\sa(k)}\cdot a_{\overline\ad(k)}\cdot \omega_k$,
  so the diagram witnesses the desired relation.
\end{proof}

\Danger 
In the next proposition
that we work in the ordinary, non-extended unitary group $U(k,l)$,
as opposed to the extended unitary group $\tilde U(k,l)$.
This is important because the classes 
\[  \res^{\tilde U(k+l)}_{\tilde U(k,l)}(u_{k+l}) \text{\qquad and\qquad}
  p_1^*(u_k) \ + \  p_2^*(u_l) \]
are {\em different} in the group $\pi_\sigma^{\tilde U(k,l)}(\Sigma^\infty \bU)$.
This feature is another manifestation of the non-additivity of certain
pieces of structure for surjectively augmented Lie groups,
such as the non-additivity in Example \ref{eg:c_not_additive} or Remark \ref{rk:nonadditive}.

\begin{prop}\label{prop:u_is_additive}
For all $k,l\geq 1$, the relation  
\[  \res^{\tilde U(k+l)}_{U(k,l)}(u_{k+l}) \ = \
  p_1^*(\res^{\tilde U(k)}_{U(k)}(u_k)) \ + \  p_2^*(\res^{\tilde U(l)}_{U(l)}(u_l)) \]
 holds in the group $\pi_1^{U(k,l)}(\Sigma^\infty \bU)$,
 where $p_1\colon U(k,l)\to U(k)$ and $p_2\colon U(k,l)\to U(l)$
 are the projections to the blocks.
\end{prop}
\begin{proof}
  The map $\delta_{k+l}$ defined in  \eqref{eq:delta_k} is the product of the
  maps $\delta_k$ and $\delta_l$, in the sense that it factors as the composite
  \[ S^\sigma \ \xra{(\delta_k,\delta_l)}\  \bU(u(\nu_k))\times \bU(u(\nu_l))\ \xra{\mu_{\nu_k,\nu_l}}\
    \bU(u(\nu_{k+l})) \ .\]
  So the relation $[\delta_{k+l}] = p_1^*[\delta_k] \cdot p_2^*[\delta_l]$
  holds in the group $\pi_\sigma^{\tilde U(k,l)}(\bU,1)$, where the multiplication
  is formed under the group structure
  arising from the ultra-commutative multiplication of $\bU$.
  The subgroup $U(k,l)$ of $\tilde U(k,l)$ acts trivially on the sign representation,
  so restriction yields the relation
  \[  \res^{\tilde U(k,l)}_{U(k,l)}[\delta_{k+l}] \ = \
    p_1^*(\res^{\tilde U(k)}_{U(k)}[\delta_k]) \cdot p_2^*(\res^{\tilde U(l)}_{U(l)}[\delta_l])  \]
  in $\pi_1^{U(k,l)}(\bU,1)$.  
  The Eckmann--Hilton argument shows that the group structure on 
  $\pi_1^{U(k,l)}(\bU,1)$ from the multiplication of $\bU$
  agrees with that as a fundamental group, i.e., by concatenation of loops.
  The stabilization map
 \[ \sigma^{U(k,l)} \ : \ \pi_1^{U(k,l)}(\bU,1)\ \to \ \pi_1^{U(k,l)}(\Sigma^\infty\bU) \]
 is a group homomorphism for the concatenation of loops
 on the source and on the target, where it is the group structure coming from stability,
 i.e., the usual addition on equivariant stable homotopy groups.
 This proves the claim.
\end{proof}

The following theorem is the main result of this section.
The classes $\omega_j$ are defined in \eqref{eq:define_omega_k},
and $a_{\overline\ad(j)}$ is the pre-Euler class \eqref{eq:preEuler}
of the reduced adjoint representation of $\tilde U(j)$.

\begin{theorem}\label{thm:omega_vs_u}
  For every $k\geq 1$, the relation
  \[ u_k\ = \
    \sum_{1\leq j\leq k}  \tr_{\tilde U(k-j,j)}^{\tilde U(k)}(p_2^*(a_{\overline\ad(j)}\cdot\omega_j))    \]
  holds in $\pi_\sigma^{\tilde U(k)}(\Sigma^\infty\bU)$,
  where $p_2\colon \tilde U(k)\to\tilde U(k-j,j)$ is the projection to the second block.
\end{theorem}
\begin{proof}
  We argue by induction over $k$.
  For $k=1$, we have $\overline\ad(1)=0$, so the claim reduces to $u_1=\omega_1$,
  which holds by Proposition \ref{prop:t_1andCayley}.
  For the rest of the proof we assume that $k\geq 2$. 

  In the next step we show that the desired relation holds after restriction
  to the subgroup $U(i,k-i)$ for every $1\leq i\leq k-1$.
  All groups involved in this part of the argument augment trivially to $C$,
  so they act trivially on the sign representation $\sigma$.
  We abbreviate
  \[ \bar u_k\ =\ \res^{\tilde U(k)}_{U(k)}(u_k)\text{\qquad and\qquad}
    w_k\ =\ \res^{\tilde U(k)}_{U(k)}(a_{\overline\ad(k)}\cdot\omega_k)\ ,\]
  both classes lying in $\pi_1^{U(k)}(\Sigma^\infty\bU)$.
For $1\leq j\leq k$, the double coset formula for $\res^{U(k)}_{U(i,k-i)}\circ\tr_{U(k-j,j)}^{U(k)}$
  established in \cite[Proposition 1.3]{schwede:Chern} yields
\begin{align*}
    \res^{U(k)}_{U(i,k-i)}&(\tr_{U(k-j,j)}^{U(k)}( p_2^*(w_j)))  \\
  &= \ \sum_{0,i-j\leq d\leq i,k-j} \tr^{U(i,k-i)}_{U(d,i-d,k-j-d,j-i+d)}(\gamma_d^*(\res^{U(k-j,j)}_{U(d,k-j-d,i-d,j-i+d)}( p_2^*(w_j))))
\end{align*}
in the group $\pi_1^{U(i,k-i)}(\Sigma^\infty\bU)$, where $\gamma_d$ is a certain
permutation matrix in $U(k)$.
  For $i-j< d<i$, we have $\overline\ad(j)^{U(i-d,j-i+d)}\ne 0$,
  thus $\res^{\tilde U(j)}_{U(i-d,j-i+d)}(a_{\overline\ad(j)})=0$, and hence
  \[ \res^{\tilde U(j)}_{U(i-d,j-i+d)}(w_j)\ =\ \res^{\tilde U(j)}_{U(i-d,j-i+d)}(a_{\overline\ad(j)})\cdot
     \res^{\tilde U(j)}_{U(i-d,j-i+d)}(\omega_j)\ = \ 0 \ . \]
So we can omit all summands with $i-j<d < i$. 
For $d=i-j$, the summation bounds force $j\leq i$,
and $\gamma_{i-j}$ is the matrix associated with the permutation
that shuffles the sets $\{i-j+1,\dots,k-j\}$ and $\{k-j+1,\dots, k\}$; so 
\[ \tr^{U(i,k-i)}_{U(i-j,j,k-i)}(\gamma_{i-j}^*(\res^{U(k-j,j)}_{U(i-j,k-i,j)}(p_2^*(w_j))))\
= \  p_1^*(\tr_{U(i-j,j)}^{U(i)}(p_2^*(w_j)))\ .\]
For $d=i$, the summation bounds force $j\leq k-i$, and
$\gamma_i$ is the identity matrix; so
\[  \tr^{U(i,k-i)}_{U(i,k-j-i,j)}(\gamma_i^*(\res^{U(k-j,j)}_{U(i,k-j-i,j)}(p_2^*(w_j)))) \ = \
 p_2^*(\tr_{U(k-j-i,j)}^{U(k-i)}(p_2^*(w_j)))\  . \]
We sum these relations over all $j\in\{1,\dots,k\}$ to obtain
\begin{align*}
  \sum_{1\leq j\leq k}
   \res^{U(k)}_{U(i,k-i)}&(\tr_{U(k-j,j)}^{U(k)}( p_2^*(w_j)))  \\
  &= \ \sum_{1\leq j\leq i} p_1^*(\tr_{U(i-j,j)}^{U(i)}(p_2^*(w_j)))\
    +\  \sum_{1\leq j\leq k-i} p_2^*(\tr_{U(k-i-j,j)}^{U(k-i)}(p_2^*(w_j)))\\
  &=\   p_1^*(\bar u_i) \ + \  p_2^*(\bar u_{k-i}) \
    = \ \res^{\tilde U(k)}_{U(i,k-i)}( u_k)\ .
\end{align*}
The second equation is the inductive hypothesis,
and the final equation is Proposition \ref{prop:u_is_additive}.

In the next step we show that the desired relation holds after
restriction to the group $U(k)$.
The space $U(k)^{\ad}$ is $U(k)$-equivariantly connected, in the sense that for every subgroup
$G$ of $U(k)$, the fixed point space $U(k)^G$, i.e., the centralizer of $G$ in $U(k)$,
is path-connected. So by \cite[Theorem 3.3.15 (i)]{schwede:global},
elements in $\pi_1^{U(k)}(\Sigma^\infty \bU)$
are detected by geometric fixed points for all closed subgroups of $U(k)$.
In other words: we may show the desired relation after applying $\Phi^G$
for all $G\leq U(k)$. If the subgroup $G$ is subconjugate to $U(i,k-i)$
for some $1\leq i\leq k-1$, then the relation holds after restriction
to $G$ by the first step, and hence also after taking $G$-geometric fixed points.
If the subgroup $G$ is not subconjugate to $U(i,k-i)$
for any $1\leq i\leq k-1$, then the $G$-action on $\nu_k$ is irreducible,
so the centralizer of $G$ in $U(k)$ equals the center of $U(k)$.
Then $\overline\sa(k)^G=\overline\ad(k)^G=0$, and hence $\Phi^G(a_{\overline\sa(k)})=\Phi^G(a_{\overline\ad(k)})=1$.
Since $G$ is not subconjugate to $U(k-j,j)$ for $1\leq j\leq k-1$, we have
$\Phi^G\circ\tr_{U(k-j,j)}^{U(k)}=0$, see \cite[Theorem 3.4.2 (ii)]{schwede:global}. Thus
\begin{align*}
 \Phi^G(u_k)\  = \ \Phi^G(a_{\overline\ad(k)}\cdot u_k)\
                   &= \ \Phi^G(a_{\overline\sa(k)}\cdot a_{\overline\ad(k)}\cdot \omega_k)\\
  &= \ \sum_{1\leq j\leq k}  \Phi^G(\tr_{\tilde U(k-j,j)}^{\tilde U(k)}(p_2^*(a_{\overline\ad(j)}\cdot\omega_j)))\ .
\end{align*}
The second equation is Theorem \ref{thm:top_Phi-relations}.
This concludes the proof that the relation holds after restriction to $U(k)$.

Now we complete the argument.
The decomposition $\ad(k)\iso\overline\ad(k)\oplus\sigma$ as
orthogonal $\tilde U(k)$-representations induces an isomorphism of based orthogonal $C$-spaces
\[ (\bGr_m)^{\ad(m)}\ \iso \ (\bGr_m)^{\overline\ad(m)\oplus\sigma}\ \iso \
  (\bGr_m)^{\overline\ad(m)}\sm S^\sigma\ .\]
So the global stable splitting \eqref{eq:global_splitting} provides an isomorphism
    \[ \pi_\sigma^{\tilde U(k)}(\Sigma^\infty \bU)\ \iso \ 
      \bigoplus_{m\geq 1}   \pi_\sigma^{\tilde U(k)}(\Sigma^\infty (\bGr_m)^{\ad(m)})
      \ \iso \ 
      \bigoplus_{m\geq 1}   \pi_0^{\tilde U(k)}(\Sigma^\infty(\bGr_m)^{\overline\ad(m)}) \ . \]
Classes in $\pi_\sigma^{\tilde U(k)}(\Sigma^\infty\bU)$
    are thus detected by geometric fixed points for all closed subgroups
    of $\tilde U(k)$. A caveat is that if $G\leq U(k)$, then $\sigma^G=\mR$,
    and the geometric fixed points live in $\Phi_1^G(\Sigma^\infty \bU)$;
    and if $G$ maps onto $C$, then $\sigma^G=0$, and the geometric fixed points
    live in $\Phi_0^G(\Sigma^\infty\bU)$.
    
    If the subgroup $G$ is contained in $U(k)$, then the relations holds after restriction
    to $G$ by the previous step, and hence also after taking $G$-geometric fixed points.    
    If $G$ is not contained in $U(k)$,  then $\sigma^G=0$,
    and thus $\Phi^G(a_\sigma)=1$. Thus
\begin{align*}
  \Phi^G(u_k)\
  &= \ \Phi^G(a_\sigma\cdot u_k)\\
 _{\eqref{eq:u_vs_v}}  &= \ \Phi^G(q_*(v_k))\
                         = \ 
   \sum_{1\leq j\leq k}  \Phi^G(a_\sigma\cdot\tr_{\tilde U(k-j,j)}^{\tilde U(k)}( p_2^*( a_{\overline\ad(j)}\cdot \omega_j)))\\
  &= \  \sum_{1\leq j\leq k}  \Phi^G(\tr_{\tilde U(k-j,j)}^{\tilde U(k)}(p_2^*( a_{\overline\ad(j)}\cdot \omega_j)))\ .
\end{align*}
The third equation is Theorem \ref{thm:v_intermsof_tau},
plus the fact that $q_*(\tau_0)=0$, and the relation
\[ q_*(1\times (a_{\ad(j)}\cdot \tau_j))\ = \
  q_*(p_2^*(a_{\ad(j)}\cdot \tau_j))\ = \
 p_2^*( a_{\ad(j)}\cdot q_*(\tau_j))\ = \
  a_\sigma\cdot p_2^*(a_{\overline\ad(j)}\cdot \omega_j)\]
for all $j\geq 1$.
\end{proof}

\section{The interplay of the global splitting and the eigenspace morphism}

In this section we establish a subtle connection between
two a priori unrelated features of the ultra-commutative monoid $\bU$,
namely its $C$-global stable splitting and 
its preferred infinite delooping.
As we show in Theorem \ref{thm:annihilate},
the adjoint $\Sigma^\infty\bU\to \sh^\sigma \bKR$ of the preferred infinite delooping
$\bU\sim \Omega^\bullet(\sh^\sigma \bKR)$ from Theorem \ref{thm:U2shKR}
annihilates all the higher terms
of the stable global splitting \eqref{eq:global_splitting}.
This fact is fundamental for all other results in this paper.\medskip

For an augmented Lie group $\alpha\colon G\to C$ and
a compact $G$-space $A$, we write $KR_\alpha(A)$ for
the Real $\alpha$-equivariant K-group of $A$,
the Grothendieck group of Real $\alpha$-equivariant vector bundles,
see Construction \ref{con:K_G_via_spc}.
If $A$ is endowed with a $G$-fixed basepoint, we write
$\widetilde{KR}_\alpha(A)$ for the reduced K-group, the kernel of restriction to the basepoint
$KR_\alpha(A)\to KR_\alpha(\ast)$. The multiplicative structure by tensor product
of Real-equivariant vector bundles makes the groups
$KR_\alpha(A)$ and $\widetilde{KR}_\alpha(A)$ into modules over
$KR_\alpha(\ast)=RR(\alpha)$, the Real representation ring of $\alpha\colon G\to C$.

As before we let $\tilde T^k=T^k\rtimes C$ denote the
semi-direct product of the $k$-torus by the complex conjugation action,
augmented to $C$ by the projection.
We call an $RR(\tilde T^k)$-module {\em Euler-torsion-free} 
if the Euler class of every irreducible Real $\tilde T^k$-representa\-tion
with trivial fixed points acts injectively on the module.

\begin{prop}\label{prop:EulerTorsionFree}
  For every $k\geq 1$ and $l\geq 0$, the $RR(\tilde T^k)$-module
    $\widetilde{KR}_{\tilde T^k}(S^{l\sigma})$ is Euler-torsion-free.
\end{prop}
\begin{proof}
  We start by showing that for every abelian group $M$, the $RR(\tilde T^k)$-module
  $RR(\tilde T^k)\tensor M$ is Euler-torsion-free.
  The ring $RR(\tilde T^k)$ is a Laurent polynomial ring
  \[ RR(\tilde T^k)\ = \ \mZ[x_1^{\pm},\dots,x_k^{\pm}]\ ,\]
  with $x_i=p_i^*(\nu_1)$ for $p_i\colon\tilde T^k\to \tilde T$
  the projection to the $i$-th factor.
  Indeed, the forgetful ring homomorphism
  \[ \res^{\tilde T^k}_{T^k}\ : \ RR(\tilde T^k)\ \to \ RU(T^k) \ = \ \mZ[x_1^{\pm},\dots,x_k^{\pm}] \]
  is injective by \cite[page 13]{atiyah-segal:completion}, and the ring 
  $RU(T^k)$ is generated by the underlying unitary representations of the
  Laurent monomials in the $x_i$. So this restriction map is an isomorphism.
  
  The irreducible Real $\tilde T^k$-representations $\lambda$ correspond to the
  monomial units $x_1^{i_1}\cdot\ldots\cdot x_k^{i_k}$
  for $i_1,\dots, i_k\in \mZ$; and $\lambda$ has trivial fixed points
  precisely when not all $i_j$ equal to $0$.
  Moreover, the Euler class of $\lambda$ is the element
  \[ e_\lambda \ = \  1-x_1^{i_1}\cdot\ldots\cdot x_k^{i_k}\ .\]
  So for $\lambda^{\tilde T^k}=0$, the underlying abelian group of $RR(\tilde T^k)/(e_\lambda)$
  is free, and hence $\text{Tor}(RR(\tilde T^k)/(e_\lambda),M)=0$.
  This means that multiplication by $e_\lambda$ is injective on 
  $RR(\tilde T^k)\tensor M$, so $RR(\tilde T^k)\tensor M$ is Euler-torsion-free.  

  Now we turn to the proof of the proposition.
  All irreducible Real $\tilde T^k$-representations
  are 1-dimensional, and hence of real type, i.e., their only automorphisms
  are scalars from $\mR$.
  So for every compact space $A$ with trivial $\tilde T^k$-action,
  \cite[Proposition 8.1]{atiyah-segal:completion} provides
  an isomorphism
  \[ RR(\tilde T^k)\tensor KO(A)\ \xra{\ \iso \ } \ KR_{\tilde T^k}(A) \ .  \]
  Applying this to $A=S^n$ and the restriction to its basepoint
  yields a commutative square
  \[
    \xymatrix{ RR(\tilde T^k)\tensor KO(S^n)\ar[r]^-\iso\ar[d] & KR_{\tilde T^k}(S^n)\ar[d]\\
    RR(\tilde T^k)\tensor KO(\ast)\ar[r]_-\iso & KR_{\tilde T^k}(\ast)}
  \]
  in which both horizontal maps are isomorphisms.
  So the upper map restricts to an isomorphism between the vertical kernels
  $RR(\tilde T^k)\tensor \widetilde{KO}(S^n)$ and $\widetilde{KR}_{\tilde T^k}(S^n)$.

  Now we choose an $m\geq 0$ such that $l\leq 8 m$.
  The orthogonal $\tilde T^k$-representation $1\oplus\sigma$
  underlies a Real $\tilde T^k$-representation (with trivial $T^k$-action).
  So  Bott periodicity   \cite[Theorem (5.1)]{atiyah:Bott and elliptic}
  for the Real $\tilde T^k$-representation $l\cdot(1\oplus\sigma)$
  and the 8-fold Bott periodicity of Real equivariant K-theory
  provide isomorphisms of $RR(\tilde T^k)$-modules
  \[    RR(\tilde T^k)\tensor \widetilde{KO}(S^{8m-l})\ \iso \
    \widetilde{KR}_{\tilde T^k}(S^{8m-l})\ \iso \
    \widetilde{KR}_{\tilde T^k}(S^{8 m +l\sigma})\ \iso \
    \widetilde{KR}_{\tilde T^k}(S^{l\sigma})\ .  \]
  We showed above that this $RR(\tilde T^k)$-module is Euler-torsion-free,
  so this completes the proof.
\end{proof}

The connective global K-theory spectrum $\bku$ is defined in
\cite[Construction 3.6.9]{schwede:global},
generalizing a configuration space model of Segal \cite{segal:K-homology}
to the global equivariant context. It comes with a multiplicative
involution $\psi$ by complex conjugation that enhances it
to the connective Real-global K-theory spectrum $\bkr=(\bku,\psi)$.
We recall the definition in Construction \ref{con:kr}.
The {\em eigenspace morphism} of based orthogonal $C$-spaces
\[ \eig \ : \ \bU \ \to \ \Omega^\bullet(\sh^\sigma  \bkr)   \]
is defined in \cite[(6.3.26)]{schwede:global}; 
we recall the construction in \eqref{eq:eig}.
Here `$\sh^\sigma $' denotes the shift of an orthogonal $C$-spectrum
by the sign representation, see \cite[Construction 3.1.21]{schwede:global}.
Shifting by $\sigma$ is Real-globally equivalent to suspending by $S^\sigma$,
see \cite[Proposition 3.1.25 (ii)]{schwede:global}.
And $\Omega^\bullet$ is the functor from orthogonal $C$-spectra to based orthogonal $C$-spaces
that is right adjoint to the reduced suspension spectrum functor,
see \cite[Construction 4.1.6]{schwede:global};
the orthogonal $C$-space $\Omega^\bullet X$ models the underlying
‘Real-global infinite loop space’ of a Real-global spectrum $X$.
As the name suggest, the eigenspace morphism assigns to a unitary automorphism
the configuration of eigenvalues and eigenspaces;
the shift coordinate in $\sh^\sigma \bkr$ is the place that stores the eigenvalues.
We write 
\begin{equation}\label{eq:adjoint_eig}
  \eig^\flat \ : \ \Sigma^\infty \bU \ \to \ \sh^\sigma  \bkr  
\end{equation}
for the adjoint of the eigenspace morphism.

The global K-theory spectrum $\bKU$
was introduced by Joachim \cite[Definition 4.3]{joachim:coherences}
as a commutative orthogonal ring spectrum.
Joachim showed in \cite[Theorem 4.4]{joachim:coherences}
that the genuine $G$-spectrum underlying the global spectrum $\bKU$ represents
$G$-equivariant complex K-theory; another proof can be found in 
\cite[Corollary 6.4.23]{schwede:global}.
Joachim's orthogonal ring spectrum can be enhanced to an orthogonal $C$-ring spectrum $\bKR$
by suitably incorporating complex conjugation, see \cite[Section 6]{halladay-kamel};
this is the {\em periodic Real-global K-theory spectrum}.
We review the definition in Construction \ref{con:KR},
and we show in Theorem \ref{thm:KU_deloops_BUP} that $\bKR$ represents
Real-equivariant K-theory for augmented Lie groups, justifying the name.
A morphism of commutative orthogonal $C$-ring spectra
\[ j \ : \ \bkr \ \to \ \bKR \]
from connective to periodic Real-global K-theory
is defined in \cite[Corollary 6.4.13]{schwede:global};
we review the definition in Construction \ref{con:kr2KR}.

\begin{theorem}\label{thm:annihilate}
  For every $k\geq 2$, the morphism of orthogonal $C$-spectra
  $(\sh^\sigma j)\circ \eig^\flat\colon \Sigma^\infty\bU\to\sh^\sigma\bKR$
  annihilates the class $\omega_k\in\pi_{\ad(k)}^{\tilde U(k)}(\Sigma^\infty\bU)$
  defined in \eqref{eq:define_omega_k}.
\end{theorem}
\begin{proof}
  We recall that
  \[ \varrho\ :\ \Sigma^\infty_+\bU\ \to\  \Sigma^\infty_+ \ast=\mS \text{\quad and\quad}
    q\ :\ \Sigma^\infty_+\bU\ \to\  \Sigma^\infty\bU\]
  denote the morphisms induced on reduced suspension $C$-spectra
  by the based maps $\bU_+\to S^0$ and  $\bU_+\to \bU$
  that, respectively, map $\bU$ to the non-basepoint of $S^0$,
  and are the identity on $\bU$.  
  By Proposition \ref{prop:tau_augments},
   the class $\tau_1$ in $\pi_{\ad(1)}^{\tilde T}(\Sigma^\infty_+ \bU)$
   belongs to the augmentation ideal, i.e., $\varrho_*(\tau_1)=0$.
   So the class $\tau_1\times\dots\times \tau_1$ lies in the $k$-th power of
   the augmentation ideal, where $k$ is the number of factors.
  Proposition \ref{prop:eig_additive} (ii) shows that
  the map $(\eig^\flat\circ q)_*$ annihilates
  the square of the augmentation ideal of $\pi_\star^{\tilde T^k}(\Sigma^\infty_+\bU)$.
  In particular,
   \[ (\eig^\flat\circ q)_*(\tau_1\times\dots\times\tau_1)\ = \ 0 \]
   for $k\geq 2$.

  We let $D\subset M_k(\mC)$ denote the Real $\tilde T^k$-representation
  of lower subdiagonal matrices \eqref{eq:D},
  and $a_D\in\pi_{-D}^{\tilde T^k}(\mS)$ is its pre-Euler class \eqref{eq:preEuler}.
  We use the isomorphism \eqref{eq:D_and_ad_sa} to identify $\ad(k)$ and $k\sigma\oplus D$
  as orthogonal $\tilde T^k$-representations.
  Using Theorem \ref{thm:Phi-relations-T^k} we obtain the relation
    \begin{align*}
      a_D^2\cdot \eig^\flat_*(\res^{\tilde U(k)}_{\tilde T^k}(\omega_k))\
   _{\eqref{eq:define_omega_k}} &= \  (\eig^\flat\circ q)_*(a_D^2\cdot\res^{\tilde U(k)}_{\tilde T^k}(\tau_k))\\
    &= \  (\eig^\flat\circ q)_*(a_D\cdot(\tau_1\times\dots\times \tau_1))\\
     & = \  a_D\cdot((\eig^\flat\circ q)_*(\tau_1\times\dots\times \tau_1))\ =\ 0
    \end{align*}
  in $\pi_{k\sigma-D}^{\tilde T^k}(\sh^\sigma\bkr)$,  for all $k\geq 2$.

    The representation $D$
    inherits the structure of a Real $\tilde T^k$-representation from $M_k(\mC)$.
  Real-equivariant Bott periodicity \cite[Theorem (5.1)]{atiyah:Bott and elliptic}
  provides an associated Bott class  $\beta_D\in\pi_D^{\tilde T^k}(\bKR)$
  with corresponding Euler class $e_D=\beta_D\cdot a_D$ in $\pi_0^{\tilde T^k}(\bKR)$.
  Then
  \[  e_D^2\cdot \res^{\tilde U(k)}_{\tilde T^k}(((\sh^\sigma j)\circ\eig^\flat)_*(\omega_k)) \
    = \ \beta_D^2\cdot (\sh^\sigma j)_*( a_D^2\cdot \eig^\flat_*(\res^{\tilde U(k)}_{\tilde T^k}(\omega_k)))  \ = \  0  \]
  in $\pi^{\tilde T^k}_{k\sigma +D}(\sh^\sigma\bKR)$.
  Real-equivariant Bott periodicity for $D$ also provides an isomorphism
  of $\pi_0^{\tilde T^k}(\bKR)$-modules 
  \[  \pi^{\tilde T^k}_{k\sigma +D}(\sh^\sigma\bKR)\ \iso \
    \pi^{\tilde T^k}_{k\sigma}(\sh^\sigma\bKR)\ \iso \ \pi^{\tilde T^k}_{(k-1)\sigma}(\bKR) \ .  \]
  By Theorem \ref{thm:KU_deloops_BUP} (ii),
  the ring $\pi_0^{\tilde T^k}(\bKR)$ is isomorphic to the Real representation
  ring $RR(\tilde T^k)$, in a way that identifies the
  $\pi_0^{\tilde T^k}(\bKR)$-module $\pi^{\tilde T^k}_{(k-1)\sigma}(\bKR)$
  with the $RR(\tilde T^k)$-module $\widetilde{KR}_{\tilde T^k}(S^{(k-1)\sigma})$.
  This $RR(\tilde T^k)$-module 
  is Euler-torsion-free by Proposition \ref{prop:EulerTorsionFree}.
  The $\tilde T^k$-representation $D$ has trivial fixed points,
  and thus multiplication by the Euler class $e_D$
  is injective on $\pi^{\tilde T^k}_{(k-1)\sigma}(\bKR)$. So we deduce that
  \[ \res^{\tilde U(k)}_{\tilde T^k}(((\sh^\sigma  j)\circ \eig^\flat)_*(\omega_k))\ = \ 0
  \]
  in $\pi^{\tilde T^k}_{k\sigma + D}(\sh^\sigma  \bKR)\iso\pi^{\tilde T^k}_{(k-1)\sigma + D}(\bKR)$.
  We let $\overline\ad(k)$ denote the reduced adjoint representation of $\tilde U(k)$,
  i.e., the subrepresentation of $\ad(k)$ of matrices with trivial trace.
  Any isomorphism $\ad(k)\iso \sigma\oplus\overline\ad(k)$ provides an identification
  \[ \pi^{\tilde U(k)}_{\ad(k)}(\sh^\sigma  \bKR)\ \iso \
     \pi^{\tilde U(k)}_{\sigma\oplus\overline\ad(k)}(\sh^\sigma  \bKR)\ \iso\
    \pi^{\tilde U(k)}_{\overline\ad(k)}(\bKR)\ .   \]
  The restriction homomorphism
  \[ \res^{\tilde U(k)}_{\tilde T^k}\ : \  \widetilde{KR}_{\tilde U(k)}(S^{\overline\ad(k)})\ \to \
    \widetilde{KR}_{\tilde T^k}(S^{\overline\ad(k)}) \]
  is split injective, see for example \cite[Proposition (5.2)]{atiyah:Bott and elliptic}.
  So also the restriction homomorphism
\[ 
    \res^{\tilde U(k)}_{\tilde T^k}\ : \  \pi^{\tilde U(k)}_{\overline\ad(k)}(\bKR)\ \to \
    \pi^{\tilde T^k}_{\overline\ad(k)}(\bKR)  = \pi^{\tilde T^k}_{(k-1)\sigma+ D}(\bKR)    
 \]
  is injective.
  Hence $((\sh^\sigma  j)\circ \eig^\flat)_*(\omega_k)=0$
  in $\pi^{\tilde U(k)}_{\ad(k)}(\sh^\sigma  \bKR)\iso\pi^{\tilde U(k)}_{\overline\ad(k)}(\bKR)$.
\end{proof}

\begin{rk}
  The fact that the morphism
  $(\sh^\sigma j)\circ\eig^\flat\colon\Sigma^\infty\bU \to\sh^\sigma\bKR$
  annihilates  the class $\omega_k$ for all $k\geq 2$ is crucial for all further results
  in this paper.
  We alert the reader that the eigenspace morphism 
  $\eig^\flat\colon\Sigma^\infty\bU \to\sh^\sigma\bkr$ itself
  does {\em not} annihilate the class $\omega_k$ for any $k\geq 2$.
  At this point one might want to recall that the name `connective global K-theory'
  has to be taken with a grain of salt, in that the morphism $j\colon \bkr\to\bKR$
  is an equivariant connective cover for {\em finite} groups
  by Theorems 6.3.27 and 6.4.21 of \cite{schwede:global}, but not generally
  for compact Lie groups of positive dimension.

  In fact, the composite
  \[ \Sigma^\infty\bU \ \xra{\eig^\flat}\ \sh^\sigma\bkr\ \xra{\sh^\sigma\dim}\
    \sh^\sigma (S p^\infty) \]
  does not annihilate $\omega_k$ for $k\geq 2$.
  Here $S p^\infty$ denotes the orthogonal spectrum, with trivial $C$-action,
  made from infinite symmetric products of spheres, compare \cite[Example 5.3.10]{schwede:global}.
  And $\dim\colon\bkr\to Sp^\infty$ is the dimension homomorphism to
  the infinite symmetric product spectrum, defined in \cite[Example 6.3.36]{schwede:global}.
  The spectrum $Sp^\infty$ is $\Fin$-globally equivalent to the Eilenberg--MacLane spectrum
  for the constant global functor $\underline{\mZ}$, see Propositions 5.3.9 and 5.3.12
  of \cite{schwede:global}. However, for compact Lie groups $G$ of positive dimension,
  the groups $\pi_*^G(S p^\infty)$ are typically not concentrated in dimension~0,
  and the ring $\pi_0^G(S p^\infty)$ need not be isomorphic to $\mZ$.
  For example,
  $\pi_1^T(S p^\infty)\iso\mQ$, see \cite[Theorem 5.3.16]{schwede:global},
  and the abelian group $\pi_0^{S U(2)}(S p^\infty)$ has rank 2,
  see \cite[Example 4.16]{schwede:infinite}.
 
  We show the non-vanishing for $k=2$, the other cases being similar but slightly more involved.
  The class $u_k$ from \eqref{eq:define_u_k}
  is represented by the suspension by $S^{\nu_k}$ of the composite
  \[ \delta_k =  \zeta^k_*\circ \partial\circ\mathfrak c \ : \ S^\sigma\  \to\  \bU(u(\nu_k))\ .\]
  This map sends $x\in S^\sigma$ to the unitary automorphism of $u(\nu_k)_\mC$
  that is multiplication by $\mathfrak c(x)$ on the image of the $\mC$-linear monomorphism
  $\zeta^k\colon \nu_k\to u(\nu_k)_\mC$, and the identity on its orthogonal complement.
  The morphism $\eig^\flat\colon\Sigma^\infty\bU \to\sh^\sigma\bkr$
  extracts eigenvalues and eigenspaces, and turns the eigenvalues into a suspension coordinate
  via $\mathfrak c^{-1}$;
  and as the name suggests, the morphism $\dim\colon\bkr\to Sp^\infty$
  takes a configuration of vector spaces to the configuration of the dimensions.
  So the class $((\sh^\sigma\dim)\circ \eig^\flat)_*(u_k)$
  is represented by the map
  \[   S^{\nu_k\oplus\sigma}\  \to \  S p^\infty(S^{\nu_k\oplus\sigma}) \ = \
     \sh^\sigma(S p^\infty)(u(\nu_k)) \ , \quad
    x  \ \longmapsto \ k\cdot x\ ,\]
  the point $x\in S^{\nu_k\oplus\sigma}$ with multiplicity $k$.
  This map represents $\sh^\sigma(k\cdot 1)$,
  the $k$-fold multiple of the shifted multiplicative unit
  in $\pi^{\tilde U(k)}_\sigma(\sh^\sigma (S p^\infty))$, and thus
  \[ ((\sh^\sigma\dim)\circ \eig^\flat)_*(u_k)\ = \  \sh^\sigma(k\cdot 1)\ . \]
  Theorem \ref{thm:omega_vs_u} for $k=2$ and the fact that $\omega_1=u_1$ provide the relation
  \[  a_{\overline\ad(2)}\cdot \omega_2 \ = \ u_2- \tr_{\tilde U(1,1)}^{\tilde U(2)}(p_2^*(u_1)) \  .\]
  Thus
  \begin{align*}
    a_{\overline\ad(2)}\cdot ((\sh^\sigma\dim)\circ\eig^\flat)_*(\omega_2)\
    & =\   ((\sh^\sigma\dim)\circ\eig^\flat)_*(u_2)\ - \  \tr_{\tilde U(1,1)}^{\tilde U(2)}(p_2^*(((\sh^\sigma\dim)\circ\eig^\flat)_*(u_1)))    \\
    &=\    \sh^\sigma\left(2 -\tr_{\tilde U(1,1)}^{\tilde U(2)}(1) \right)\ .
  \end{align*}
  By the remark immediately before \cite[Example 4.16]{schwede:infinite},
  the classes $1$ and $\tr_{U(1,1)}^{U(2)}(1)$ are linearly independent in the group
  $\pi_0^{U(2)}(S p^\infty)$.
  Hence the classes
  $1$ and $\tr_{\tilde U(1,1)}^{\tilde U(2)}(1)$ are linearly independent in the group
  $\pi_0^{\tilde U(2)}(S p^\infty)$.
  Thus the class $((\sh^\sigma\dim)\circ\eig^\flat)_*(\omega_2)$ is non-zero.
\end{rk}

\section{The global Segal--Becker splitting}

In this section we construct the global Segal--Becker splitting
$c\colon \bBUP\to\Omega^\bullet(\Sigma^\infty_+\bP)$ in the unstable Real-global
homotopy category, see \eqref{eq:define_gSB}.
This morphism comes into existence as a $C$-global $\sigma$-loop map,
the delooping being the morphism
$d\colon\bU\to \Omega^\bullet(\sh^\sigma(\Sigma^\infty_+\bP))$
defined in \eqref{eq:define_d}.
That the morphisms $d$ and $c$ are indeed sections to
$\Omega^\bullet(\sh^\sigma\eta)\colon\Omega^\bullet(\sh^\sigma(\Sigma^\infty_+\bP))\to
\Omega^{\bullet}(\sh^\sigma\bKR)$
and to $\Omega^\bullet(\eta)\colon\Omega^\bullet(\Sigma^\infty_+\bP)\to
\Omega^{\bullet}(\bKR)$, respectively,
is proved in Theorems \ref{thm:splits} and \ref{thm:split}.
In Corollary \ref{cor:c_on_P} we show that the composite
$c\circ h\colon \bP\to\Omega^\bullet(\Sigma^\infty_+\bP)$
of the global Segal--Becker splitting with the morphism
$h\colon\bP\to \bBUP$ that represents the inclusion of line bundles into
virtual vector bundles is the unit of the adjunction $(\Sigma^\infty_+,\Omega^\bullet)$.

\begin{construction}[The Real-global ultra-commutative monoid $\bP$]\label{con:P}
  We recall the orthogonal $C$-space $\bP$ made from complex projective spaces,
  compare \cite[(2.3.20)]{schwede:global}, a specific ultra-commutative model
  for the Real-global classifying space of the extended circle group
  $\tilde T=\tilde U(1)=U(1)\rtimes C$.
  In \cite{schwede:global}, we use the notation $\bP^\mC$ to distinguish this version
  made from complex projective spaces from the real version.
  In this paper, the real version plays no role, so we simplify notation
  and drop the superscript `$\mC$'.
  The value of $\bP$ at the inner product space $V$ is
  \[ \bP(V)\ = \ P(\Sym(V_\mC)) \ ,\]
  the complex projective space of the symmetric algebra of the complexification.
  The structure map $\bP(\varphi)\colon\allowbreak \bP(V)\to\bP(W)$
  induced by a linear isometric embedding $\varphi\colon V\to W$
  takes a complex line to its image under
  $\Sym(\varphi_\mC)\colon \Sym(V_\mC)\to \Sym(W_\mC)$.
  The involution  $\psi(V)\colon\bP(V)\to\bP(V)$
  is induced by complex conjugation on $\Sym(V_\mC)$,
  exploiting that also conjugate-linear maps take $\mC$-subspaces to $\mC$-subspaces.

  As a Real-global classifying space for $\tilde T$,
  the orthogonal $C$-space $\bP$ represents the functor $\pi_0^{\tilde T}$ on the
  unstable Real-global homotopy category.
  We will later need to refer to the universal element,
  the tautological class, so we recall it here.
  We consider the $\tilde T$-invariant complex line
  \[ \mL\ = \ \mC\cdot(1\tensor 1 - i\tensor i) \ =  \
  \im(\zeta^1\colon\nu_1\to u(\nu_1)_\mC)\ \subset \ u(\nu_1)_\mC\ ,\]
  where $\zeta^1$ is defined in \eqref{eq:def_zeta^W}.
  We also view $\mL$ as a complex line in the symmetric algebra via the embedding
  $u(\nu_1)_\mC\to \Sym(u(\nu_1)_\mC)$ as the linear summand. 
  Then $\mL$ is a $\tilde T$-fixed point of $P(\Sym(u(\nu_1)_\mC))=\bP(u(\nu_1))$;
  the unstable and stable {\em tautological classes}
  \begin{equation}  \label{eq:define_u_T}
       u_{\tilde T}\ \in \ \pi_0^{\tilde T}(\bP)\text{\qquad and\qquad}
    e_{\tilde T}\ \in \ \pi_0^{\tilde T}(\Sigma^\infty_+\bP)
     \end{equation}
     are, respectively, its homotopy class and
     the class represented by the $\tilde T$-equivariant map
     \[  S^{\nu_1}\ \xra{-\sm\mL} \
    S^{\nu_1}\sm \bP(u(\nu_1))_+ \ = \ (\Sigma^\infty_+\bP)(u(\nu_1))\ . \]  
  Then the pair $(\bP,u_{\tilde T})$ represent
  the functor $\pi_0^{\tilde T}$ on the unstable Real-global homotopy category;
  and the pair $(\Sigma^\infty_+ \bP,e_{\tilde T})$ 
  represent the functor $\pi_0^{\tilde T}$ on the Real-global stable homotopy category.
\end{construction}

\begin{construction}[The morphism $\eta\colon\Sigma^\infty_+ \bP\to\bKR$]
  The morphism of non-equivariant spectra from $\Sigma^\infty_+ \mC P^\infty$ to $K U$ 
  that classifies the tautological complex line bundle over $\mC P^\infty$
  has a particularly nice and prominent Real-global refinement
  $\eta\colon\Sigma^\infty_+ \bP\to\bKR$,
  defined as a composite of two morphisms of ultra-commutative $C$-ring spectra
  \[ \Sigma^\infty_+ \bP\ \xra{\ \mu\ }\  \bkr\ \xra{\ j \ }\ \bKR \ .\]
  The first morphism $\mu$ is the inclusion of the `rank 1' part
  in the rank filtration, compare \cite[Construction 6.3.40]{schwede:global};
  its value at an inner product space $V$ is the $C$-equivariant map
\[  \mu(V)\ : \
    (\Sigma^\infty_+ \bP)(V) = S^V\sm P(\Sym(V_\mC))_+\ \to\  \bk(\Sym(V_\mC),S^V)= \bkr(V) \ ,
    \quad     v\sm L \longmapsto [L;v]\ . \]
  We review the definition of the morphism $j\colon\bkr\to\bKR$ in Construction \ref{con:kr2KR}.
  
  The underlying morphism of Real-global spectra of $\eta$ classifies
  the tautological $\tilde T$-representation:
  under the preferred identification of $\pi_0^{\tilde T}(\bKR)$
  with the Real representation ring $RR(\tilde T)$
  given by Theorem \ref{thm:KU_deloops_BUP} (ii),
  the stable tautological class \eqref{eq:define_u_T} maps to the class of the
  tautological $\tilde T$-representation, i.e.,
  \[ \eta_*(e_{\tilde T})\ = \ [\nu_1] \quad \text{in $\pi_0^{\tilde T}(\bKR)\iso R R(\tilde T)$.} \]
Because $\eta$ is a morphism of ultra-commutative $C$-ring spectra,
its effect on equivariant homotopy groups is not only compatible with restriction,
inflations and transfers,
but also with multiplicative power operations and norms.
In \cite{schwede:global_Snaith}, I establish a global equivariant generalization of
Snaith's celebrated theorem \cite{snaith:algebraic cobordism, snaith:localized stable},
saying that $K U$ can be obtained from $\Sigma^\infty_+\mC P^\infty$ by `inverting the Bott class'.
\end{construction}

\begin{construction}[The $\sigma$-deloop of the global Segal--Becker splitting]\label{con:SB-splitting}
  A natural morphism $\lambda^\sigma_X\colon X\sm S^\sigma\to\sh^\sigma X$
  for orthogonal $C$-spectra is defined in \cite[(3.1.23)]{schwede:global}.
  This morphism is a $C$-global equivalence by \cite[Theorem 3.1.25 (ii)]{schwede:global}.
  Theorem \ref{thm:gh_from_U} lets us define $C$-global morphisms from $\Sigma^\infty\bU$
by specifying
their values on the classes $\omega_k$ in $\pi_{\ad(k)}^{\tilde U(k)}(\Sigma^\infty\bU)$
defined in \eqref{eq:define_omega_k}.
So we let
\begin{equation} \label{eq:define_d^flat}
  d^\flat\ :\  \Sigma^\infty\bU\ \to \ \sh^\sigma(\Sigma^\infty_+\bP)
\end{equation}
be the unique morphism in the $C$-global stable homotopy category such that 
\[ d^\flat_*(\omega_k)\ = \
  \begin{cases}
    (\lambda^\sigma_{\Sigma^\infty_+\bP})_*(e_{\tilde T}\sm S^\sigma)   & \text{ for $k=1$, and}\\
    \quad  0   & \text{ for $k\geq 2$.}\\
  \end{cases}
\]
For $k=1$ we have implicitly identified the sign representation
with $\ad(1)$ by sending $x\in\sigma$ to $i\cdot x\in\ad(1)$.
Informally speaking, the morphism $d^\flat$ is the projection onto
the summand $k=1$ in the stable splitting \eqref{eq:global_splitting}.

As a Quillen adjoint functor pair for the $C$-global model structures,
the pair $(\Sigma^\infty,\Omega^\bullet)$ derives to an adjoint functor pair at the
  level of $C$-global homotopy categories.
  The stable morphism $d^\flat$ from \eqref{eq:define_d^flat}
  is thus adjoint to an unstable morphism in the homotopy
  category of based $C$-global spaces
  \begin{equation}\label{eq:define_d}
    d \ : \ \bU \ \to \ \Omega^\bullet(\sh^\sigma(\Sigma^\infty_+\bP)) \ .
  \end{equation}
  This is our $\sigma$-deloop of the global Segal--Becker splitting.
\end{construction}

We can now prove Theorem A from the introduction.
  
\begin{theorem}\label{thm:splits}
   The composite 
  \[    \bU\ \xra{\ d\ } \ \Omega^\bullet(\sh^\sigma(\Sigma^\infty_+\bP))
    \ \xra{\Omega^\bullet(\sh^\sigma \eta)}\
    \Omega^{\bullet}(\sh^\sigma \bKR)  \]
  equals the morphism $\Omega^\bullet(\sh^\sigma j)\circ\eig\colon\bU\to \Omega^\bullet(\sh^\sigma\bKR)$,
  and is thus a $C$-global equivalence.
\end{theorem}
\begin{proof}
  We claim that the following diagram commutes in the $C$-global stable homotopy category:
  \begin{equation}\begin{aligned}\label{eq:U2shKR}
\xymatrix@C=15mm{
    \Sigma^\infty \bU  \ar[r]_-{\eqref{eq:define_d^flat}}^-{d^\flat}
    \ar[d]^{\eqref{eq:adjoint_eig}}_{\eig^\flat}&
      \sh^\sigma(\Sigma^\infty_+\bP)\ar[d]^{\sh^\sigma\eta} \\
      \sh^\sigma\bkr \ar[r]_-{\sh^\sigma j} & \sh^\sigma\bKR   }      
    \end{aligned}\end{equation}
  By Theorem \ref{thm:gh_from_U}, it suffices to show that
  both composites agree on the classes $\omega_k$ for all $k\geq 1$.
  For $k\geq 2$, we have $d^\flat_*(\omega_k)=0$ by definition,
  and $((\sh^\sigma j)\circ \eig^\flat)_*(\omega_k)=0$ by Theorem \ref{thm:annihilate}.
  So both composites in the diagram \eqref{eq:U2shKR} annihilate the class $\omega_k$.

  For $k=1$ the following diagram of based $\tilde T$-maps commutes by inspection of definitions:
   \[ \xymatrix@C=5mm{
       S^{\nu_1\oplus\sigma}\ar[rr]^-{S^{\nu_1}\sm\mathfrak c}_-{\eqref{eq:define_Cayley}} \ar[d]_{-\sm\mL\sm-} &&
       S^{\nu_1}\sm U(1)\ar[rr]^-{S^{\nu_1}\sm\zeta^1_*}_-{\eqref{eq:def_zeta^k}} &&
       S^{\nu_1}\sm \bU(u(\nu_1))\ar@{=}[r] & (\Sigma^\infty\bU)(u(\nu_1))\ar[d]^{\eig^\flat(u(\nu_1))}\\
       S^{\nu_1}\sm \bP(u(\nu_1))_+\sm S^\sigma\ar[rr]_-{(\lambda^\sigma_{\Sigma^\infty_+\bP})(\nu_1)}&&
       S^{\nu_1\oplus\sigma}\sm \bP(u(\nu_1)\oplus\sigma)_+\ar[rr]_-{(\sh^\sigma\mu)(\nu_1)}&&
       \bkr(u(\nu_1)\oplus\sigma) \ar@{=}[r]& (\sh^\sigma\bkr)(u(\nu_1))    &
     } \]
   Indeed, both composites send $x\in S^{\nu_1\oplus\sigma}$
  to the one-element configuration
\[ [\mL,x]\ \in \ \bk(\Sym((u(\nu_1)\oplus\sigma)_\mC),S^{\nu_1\oplus\sigma})\ =\ (\sh^\sigma\bkr)(u(\nu_1))\ , \]
  the point $x$ labeled by the complex line
  $\mL=\mC\cdot (1\tensor 1-i\tensor i)\subset u(\nu_1)_\mC$,
  embedded via $u(\nu_1)_\mC\to (u(\nu_1)\oplus\sigma)_\mC\to \Sym((u(\nu_1)\oplus\sigma)_\mC)$.
   This diagram witnesses the relation
   \[ ((\sh^\sigma \mu)\circ\lambda^\sigma_{\Sigma^\infty_+\bP})_*(e_{\tilde T}\sm S^\sigma) \ = \
   \eig^\flat_*(u_1)\]
 in the group $\pi_\sigma^{\tilde T}(\sh^\sigma\bkr)$.
 Since  $\eta=j\circ\mu$ by definition, we obtain
    \begin{align*}
      ((\sh^\sigma \eta)\circ d^\flat)_*(\omega_1)\
    &= \ ((\sh^\sigma j)\circ(\sh^\sigma \mu)\circ\lambda^\sigma_{\Sigma^\infty_+\bP})_*(e_{\tilde T}\sm S^\sigma)\
    = \ ((\sh^\sigma j)\circ \eig^\flat)_*(u_1)\  .
  \end{align*}
  Since $u_1=\omega_1$ by Proposition \ref{prop:t_1andCayley},
  this concludes the proof that the diagram \eqref{eq:U2shKR} commutes.
  Passing to adjoints for the adjunction $(\Sigma^\infty,\Omega^\bullet)$,
  turns the commutative diagram \eqref{eq:U2shKR} into the desired relation
  $\Omega^\bullet(\sh^\sigma j)\circ\eig=\Omega^\bullet(\sh^\sigma\eta)\circ d\colon\bU\to \Omega^\bullet(\sh^\sigma \bKR)$.
   The composite $\Omega^\bullet(\sh^\sigma  j)\circ\eig$
   is a $C$-global equivalence by Theorem \ref{thm:U2shKR},
   hence so is $\Omega^\bullet(\sh^\sigma\eta)\circ d$.
 \end{proof}

Now we construct the actual global Segal--Becker splitting
$c\colon \bBUP\to\Omega^\bullet(\Sigma^\infty_+\bP)$ by looping the morphism 
$d \colon\bU \to  \Omega^\bullet(\sh^\sigma(\Sigma^\infty_+\bP))$ by the sign representation,
and exploiting the Real-global Bott periodicity equivalence $\bBUP\sim \Omega^\sigma\bU$.

\begin{construction}[The global Segal--Becker splitting]
  The ultra-commutative $C$-monoid $\bBUP$
  is defined in \cite[Example 2.4.33]{schwede:global},
  and we recall the construction in \ref{con:Gr2BUP}.
  In Theorem \ref{thm:BUP2OmegaU} we establish a Real-global form of
  Bott periodicity, summarized by the Real-global equivalence of ultra-commutative $C$-monoids
  \[ \gamma\ : \ \bBUP\ \xra{\ \sim \ } \ \Omega^\sigma\bU  \]
  defined in \eqref{eq:define_gamma}.
  We then define the {\em  global Segal--Becker splitting}
  \begin{equation} \label{eq:define_gSB}
    c \ : \ \bBUP\ \to \ \Omega^\bullet(\Sigma^\infty_+\bP)   
  \end{equation}
  as the unique morphism in the unstable based $C$-global homotopy category
  that makes the following diagram commute:
  \begin{equation} \begin{aligned}\label{eq:define_c}
 \xymatrix{
    \bBUP\ar[rrr]^-c \ar[d]_{\gamma}^{\sim}
    &&& \Omega^\bullet(\Sigma^\infty_+\bP) \ar[d]^{\Omega^\bullet(\tilde\lambda^\sigma_{\Sigma^\infty_+\bP})}_{\sim}\\
    \Omega^\sigma\bU\ar[rr]_-{\Omega^\sigma d}
    &&  \Omega^\sigma(\Omega^\bullet(\sh^\sigma(\Sigma^\infty_+\bP)))\ar[r]_-\iso
    & \Omega^\bullet(\Omega^\sigma(\sh^\sigma(\Sigma^\infty_+\bP)))
    }       
    \end{aligned}
  \end{equation}
  The unnamed isomorphism swaps the order in which the loops are taken,
  and $\tilde\lambda^\sigma_X\colon X\to\Omega^\sigma(\sh^\sigma X)$
  is the adjoint of $\lambda^\sigma_X\colon X\sm S^\sigma\to\sh^\sigma X$.
\end{construction}

The next theorem verifies that the
global Segal--Becker splitting is indeed a splitting of the morphism
$\Omega^\bullet(\eta)\colon \Omega^\bullet(\Sigma^\infty_+\bP)\to\Omega^{\bullet}(\bKR)$.
In Theorem \ref{thm:rigidifies} will we show that
our global Segal--Becker realizes the classical equivariant Segal--Becker splittings
at the level of cohomology theories, thereby justifying its name.

\begin{theorem}\label{thm:split}
  The composite 
  \[     \bBUP\ \xra{\ c\ } \ \Omega^\bullet(\Sigma^\infty_+\bP)
    \ \xra{\Omega^\bullet(\eta)}\ \Omega^{\bullet}(\bKR)  \]
  equals the morphism $\Theta\colon\bBUP\to \Omega^\bullet(\bKR)$
  defined in \eqref{eq:define_Theta}, and is thus a $C$-global equivalence.
\end{theorem}
\begin{proof}
  Several definitions and a previous result assemble into a commutative diagram:
    \[ \xymatrix@C=12mm{
         \bBUP    \ar `d[dddr]_\Theta[dddrrr] \ar[dr]^{\gamma}_{\sim}\ar[rrr]^-c
        &&& \Omega^\bullet(\Sigma^\infty_+\bP) \ar[d]_{\Omega^\bullet(\tilde\lambda^\sigma_{\Sigma^\infty_+\bP})}^{\sim} \ar@<4ex>@/^4pc/[ddd]^(.3){\Omega^\bullet(\eta)} \\
                 & \Omega^\sigma\bU\ar[r]^-{\Omega^\sigma d} \ar@/_1pc/[dr]^\sim_(.4){\Omega^\sigma(\Omega^\bullet(\sh^\sigma j)\circ\eig)\quad} 
        &    \Omega^\sigma(\Omega^\bullet(\sh^\sigma(\Sigma^\infty_+\bP))) \ar[r]^-\iso\ar[d]^{\Omega^\sigma(\Omega^\bullet(\sh^\sigma\eta))}
        &    \Omega^\bullet( \Omega^\sigma(\sh^\sigma(\Sigma^\infty_+\bP))) 
      \ar[d]_{\Omega^\bullet(\Omega^\sigma(\sh^\sigma\eta))}\\
            && \Omega^\sigma(\Omega^{\bullet}(\sh^\sigma\bKR))\ar[r]_\iso&\Omega^{\bullet}( \Omega^\sigma(\sh^\sigma\bKR))\\
  &&&\Omega^\bullet (\bKR)\ar[u]^{\Omega^\bullet(\tilde\lambda^\sigma_{\bKR})}_\sim\\
                         }  \]
  The parts of the diagram involving the morphisms  $\Theta$ and $c$
  commute by the respective definitions \eqref{eq:define_Theta} and \eqref{eq:define_c}.
  The inner triangle starting at $\Omega^\sigma\bU$ commutes by Theorem \ref{thm:splits}.
  The right part of the diagram commutes by naturality of the $\tilde\lambda^\sigma_X$-morphisms.
  The morphism $\Theta$ is a $C$-global equivalence by Theorem \ref{thm:KU_deloops_BUP} (i),
  hence so is the composite $\Omega^\bullet(\eta)\circ c$.
\end{proof}

\begin{rk}[No section deloops twice]\label{rk:no_double_deloop}
  Since $\Omega^\bullet(\eta)\colon \Omega^\bullet(\Sigma^\infty_+\bP)\to \Omega^\bullet(\bKR)$
  is a Real-global infinite loop map with an unstable section,
  one can wonder how often one can deloop an unstable section.
  Our construction of the section $c\colon \bBUP\to \Omega^\bullet(\Sigma^\infty_+\bP)$
  presents it as a $C$-global $\sigma$-loop map, the deloop being the morphism
  $d\colon \bU\to \Omega^\bullet(\sh^\sigma(\Sigma^\infty_+\bP))$.
  If we forget the $C$-action and pass to underlying global morphisms,
  these data witness the global Segal--Becker splitting as global loop map.
  However, one cannot do better than this, not even non-equivariantly, as we now recall.

  The H-space structure on the infinite unitary group $U$ coming from Bott periodicity
  coincides with the one from the group structure. Under the Pontryagin product,
  the mod 2 homology $H_*(U;\mF_2)$ is an exterior $\mF_2$-algebra on classes
  $a_i\in H_{2i+1}(U;\mF_2)$ for $i\geq 0$.
  In contrast, $H_*(\Omega^\infty(\Sigma^\infty_+ \mC P^\infty\sm S^1);\mF_2)$
  is a polynomial $\mF_2$-algebra on the iterated Kudo--Araki operations
  on a basis of $\tilde H_*(\mC P^\infty_+\sm S^1;\mF_2)$,
  see for example \cite[Theorem 5.1]{dyer-lashof}.
  So the epimorphism of commutative graded $\mF_2$-algebras
  \[  (\Omega^\infty(\eta\sm S^1))_* \ : \ 
    H_*(\Omega^\infty(\Sigma^\infty_+ \mC P^\infty\sm S^1);\mF_2)\ \to \
    H_*(\Omega^\infty(KU\sm S^1);\mF_2)\ \iso \   H_*(U;\mF_2)  \]
  does not admit a multiplicative section.  Hence the map
  $\Omega^\infty(\eta\sm S^1)\colon\Omega^\infty(\Sigma^\infty_+\mC P^\infty\sm S^1)\to
  \Omega^\infty(K U\sm S^1)$ does not have a section that is an H-map, much less a loop map.
\end{rk}

  The $\tilde U(k)$-equivariant linear embedding $\zeta^k\colon\nu_k\to u(\nu_k)_\mC$
  was defined in \eqref{eq:def_zeta^W}.
  Its image is a $\tilde U(k)$-invariant linear subspace of dimension $k$,
  and thus a $\tilde U(k)$-fixed point of $Gr_k^\mC(u(\nu_k)_\mC)=\bGr_k(u(\nu_k))$.
    We write 
  \[ \{\nu_k\}\ = \ [\im(\zeta^k)]\ \in\ \pi_0^{\tilde U(k)}(\bGr_k)\subset\pi_0^{\tilde U(k)}(\bGr) \]
  for its homotopy class.
The next proposition determines the image of this class under the composite
\[ \bGr \ \xra{\ i\ }\ \bBUP \ \xra{\ c\ }\ \Omega^\bullet(\Sigma^\infty_+\bP)\ . \]

\begin{theorem}\label{thm:special_vartheta}
  For every $k\geq 1$, the relation
\[ (c\circ i)_*\{\nu_k\}\ =\ \tr_{\tilde U(k-1,1)}^{\tilde U(k)}(1\times e_{\tilde T}) \]
  holds in the group $\pi_0^{\tilde U(k)}(\Sigma^\infty_+\bP)$.
\end{theorem}
\begin{proof}
  The following diagram commutes by the definition \eqref{eq:define_c} of the morphism
  $c\colon\bBUP\to \Omega^\bullet(\Sigma^\infty_+\bP)$ from the morphism
  $d\colon\bU\to \Omega^\bullet(\sh^\sigma(\Sigma^\infty_+\bP))$,
  which in turn was defined as the adjoint of
  $d^\flat\colon\Sigma^\infty\bU\to\sh^\sigma(\Sigma^\infty_+\bP)$:
    \[ \xymatrix{ \pi_0^{\tilde U(k)}(\bGr)\ar[r]^-{i_*}\ar@/_1pc/[ddr]_{\beta_*} &
      \pi_0^{\tilde U(k)}(\bBUP)\ar[rr]^-{c_*}\ar[dd]_\iso^{\gamma_*} && \pi_0^{\tilde U(k)}(\Omega^\bullet(\Sigma^\infty_+\bP))\ar@{=}[r]&
            \pi_0^{\tilde U(k)}(\Sigma^\infty_+\bP)  \ar[d]^-{-\sm S^\sigma}_-\iso \\
&&&& \pi_\sigma^{\tilde U(k)}(\Sigma^\infty_+\bP \sm S^\sigma) \ar[d]_{\iso}^{(\lambda^\sigma_{\Sigma^\infty_+\bP})_*} \\            & \pi_0^{\tilde U(k)}(\Omega^\sigma\bU)\ar@{=}[r]& 
            \pi_\sigma^{\tilde U(k)}(\bU)\ar[r]^-{\sigma^{\tilde U(k)}} \ar@/_1pc/[rr]_(.3){d_*} &
      \pi_\sigma^{\tilde U(k)}(\Sigma^\infty\bU) \ar[r]^-{d^\flat_*} &
      \pi_\sigma^{\tilde U(k)}(\sh^\sigma(\Sigma^\infty_+\bP))
    } \]
  The map $\sigma^{\tilde U(k)}\colon\pi_\sigma^{\tilde U(k)}(\bU)\to\pi_\sigma^{\tilde U(k)}(\Sigma^\infty\bU)$
  is the stabilization map \cite[(3.3.12)]{schwede:global}.

  By inspection of definitions,
  the map $\beta_* \colon \pi_0^{\tilde U(k)}(\bGr)\to \pi_0^{\tilde U(k)}(\Omega^\sigma\bU)$
  defined in \eqref{eq:define_beta}
  sends $\{\nu_k\}$ to the homotopy class of the map
  $\delta_k=\zeta^k_*\circ\mathfrak c\circ\partial\colon S^\sigma \to\bU(u(\nu_k))$.
  The class $u_k$ defined in \eqref{eq:define_u_k}
    is represented by the suspension by $S^{\nu_k}$ of $\delta_k$, so
  \[ \sigma^{\tilde U(k)}(\beta_*\{\nu_k\})\ = \ u_k\ . \]
  The commutative diagram thus shows the relation
  $(\lambda^\sigma_{\Sigma^\infty_+\bP})_*((c\circ i)_*\{\nu_k\}\sm S^\sigma) = d^\flat_*(u_k)$
  in  $\pi_\sigma^{\tilde U(k)}(\sh^\sigma(\Sigma^\infty_+\bP))$.
  The defining property of the morphism $d^\flat$ and Theorem \ref{thm:omega_vs_u} then yield
  \begin{align*}
   (\lambda^\sigma_{\Sigma^\infty_+\bP})_*( (c\circ i)_*\{\nu_k\}\sm S^\sigma)\  = \  d^\flat_*(u_k)\
    &= \  \sum_{1\leq j\leq k}  d_*^\flat(\tr_{\tilde U(k-j,j)}^{\tilde U(k)}(p_2^*(a_{\overline\ad(j)}\cdot\omega_j))\\
    &= \  \sum_{1\leq j\leq k} \tr_{\tilde U(k-j,j)}^{\tilde U(k)}(p_2^*(a_{\overline\ad(j)}\cdot d_*^\flat(\omega_j))\\
    &= \  \tr_{\tilde U(k-1,1)}^{\tilde U(k)}(p_2^*((\lambda^\sigma_{\Sigma^\infty_+\bP})_*(e_{\tilde T}\sm S^\sigma)))
    \\
    &= \
     (\lambda^\sigma_{\Sigma^\infty_+\bP})_*( \tr_{\tilde U(k-1,1)}^{\tilde U(k)}(1\times e_{\tilde T})\sm S^\sigma)\ .
  \end{align*}
  Suspension by $S^\sigma$ and the map $(\lambda^\sigma_{\Sigma^\infty_+\bP})_*$
  are bijective, so this proves the desired relation.
\end{proof}

The pair $(\bP,u_{\tilde T})$ represents the functor $\pi_0^{\tilde T}$
on the unstable $C$-global homotopy category. So we can define 
\begin{equation}\label{eq:define_h}
 h\ :  \ \bP\ \to\ \bBUP   
\end{equation}
  as the unique morphism  such that $h_*(u_{\tilde T})=i_*\{\nu_1\}$
  in $\pi_0^{\tilde T}(\bBUP)$.
  The morphism $h$ represents the inclusion of line bundles into
  virtual vector bundles.

\begin{cor}\label{cor:c_on_P}
   The composite 
  \[ \bP \ \xra{\ h\ } \  \bBUP \ \xra{\ c\ }\ \Omega^\bullet(\Sigma^\infty_+\bP) \]
  is the unit of the adjunction $(\Sigma^\infty_+,\Omega^\bullet)$.
\end{cor}
\begin{proof}
  For $k=1$, Theorem \ref{thm:special_vartheta}
  specializes to $(c\circ i)_*\{\nu_1\} =e_{\tilde T}$ in $\pi_0^{\tilde T}(\Sigma^\infty_+\bP)$.
  So
  \[ (c\circ h)_*(u_{\tilde T}) \ = \ (c\circ i)_*\{\nu_1\} \ = \ e_{\tilde T}\ .  \]
  The adjunction unit also takes the unstable tautological
  class $u_{\tilde T}$ to the stable tautological class $e_{\tilde T}$.
  Since $(\bP,u_{\tilde T})$ represents the functor $\pi_0^{\tilde T}$ on the unstable
  $C$-global homotopy category, this proves the claim.
\end{proof}

We end this section with a discussion of the additivity, or rather the failure
thereof, of the global Segal--Becker splitting.
If we forget the Real direction and pass to underlying
`non-Real' global spaces, then the sign action on $S^\sigma$ disappears,
and $\sigma$-loops become ordinary loops.
By the Eckmann--Hilton argument, the loop addition then coincides
with the abelian monoid structure from any ultra-commutative multiplication,
and loop maps are automatically additive.
In particular, the morphism of global spaces underlying
the Real-global Segal--Becker splitting is a global loop map,
and so the induced map $[A,c]^G\colon [A,\bBUP]^G\to[A,\Omega^\bullet(\Sigma^\infty_+\bP)]^G$
is additive whenever $G$ is a compact Lie group with trivial augmentation to $C$.
\Danger However, the map $[A,c]^\alpha$ induced by the Real-global Segal--Becker splitting
is {\em not} generally additive for surjective augmentations $\alpha\colon G\to C$.
In fact, additivity already fails for $G=C$ with identity augmentation,
and for $A=\ast$, see Example \ref{eg:c_not_additive}.
The rest of this section aims to explain this more carefully,
including some quantification of the failure of additivity
in Theorem \ref{thm:formula_c_*} (ii).\medskip

Our tool to quantify the deviation from additivity
is a certain piece of natural structure on $\sigma$-loop objects.
In contrast to the fundamental group $\pi_1(X,x)=[S^1,X]_*$
in the non-equivariant context, the set $[S^\sigma,X]_*^C$
has no group structure that is natural for based $C$-equivariant maps in $X$.
Equivalently, the $\sigma$-sphere does not admit an
equivariant `pinch map' $S^\sigma\to S^\sigma\vee S^\sigma$,
i.e., such that the composite with each of the two projections is
equivariantly homotopic to the identity.
The $C$-map $m\colon S^\sigma\to S^\sigma\vee S^\sigma$ we are about to define
is a partial remedy of this defect.
If we forget the $C$-action, the underlying non-equivariant
homotopy class of $m$ represents the element $x y^{-1}x$ in the free group
$\pi_1(S^1\vee S^1,\ast)$, where $x$ and $y$ denote the classes of the left and right
wedge summand inclusions.

\begin{construction}[A binary operation on $\Omega^\sigma X$]
  We let $X$ be a based orthogonal $C$-space.
  We define a specific binary operation \eqref{eq:define_m^*}
  on the $\sigma$-loop space $\Omega^\sigma X$.
  We consider the $C$-equivariant based map 
  \begin{align*}
    m'\ : \ U(1)\ &\to\ U(1)\vee U(1) \ , \quad
    m'(z)\ = \ 
                \begin{cases}
                  (\phantom{-}z^4,1) & \text{ if $\text{Re}(z)\geq 0$, and}\\
                  (-z^2,2) &\text{ if $\text{Re}(z)\leq 0$.}
                \end{cases}
  \end{align*}
  The second coordinate in the formula for $m'(z)$ specifies in which of the two wedge summands
  of $U(1)\vee U(1)$ the respective point lies.
  The map $m'$ is clearly $C$-equivariant for the action by complex conjugation on all
  instances of $U(1)$.
  We define a $C$-equivariant map $m\colon S^\sigma\to S^\sigma\vee S^\sigma$
  by conjugating $m'$ with the Cayley transform \eqref{eq:define_Cayley}, i.e., as the composite 
  \[ S^\sigma \ \xra[\iso]{\ \mathfrak c\ } U(1) \ \xra{\ m' \ }\ U(1)\vee U(1) \
    \xra[\iso]{\mathfrak c^{-1}\vee c^{-1}}\ S^\sigma\vee S^\sigma\ . \]
  Precomposition with $m\colon S^\sigma\to S^\sigma\vee S^\sigma$ then yields a morphism
  \begin{equation}\label{eq:define_m^*}
    m^*\ : \    (\Omega^\sigma X)\times (\Omega^\sigma X)\ \iso \ \map_*( S^\sigma\vee S^\sigma,X) 
    \ \xra{\map_*(m,X)}    \  \map_*( S^\sigma,X)= \Omega^\sigma X
    \ .
  \end{equation}
\end{construction}

We write $\epsilon\colon S^\sigma\to S^\sigma$
for the sign involution, i.e., $\epsilon(x)=-x$.
For every based orthogonal $C$-space $X$, it induces an
involution $\epsilon^*\colon \Omega^\sigma X\to\Omega^\sigma X$ by precomposition.

\begin{prop}\label{prop:effect_of_m}
  Let $X$ be a based orthogonal $C$-space.
  Then the composite
  \[ \Omega^\sigma X\ \xra{\ (\Id,\Id)\ }\ 
    (\Omega^\sigma X)\times (\Omega^\sigma X)\ \xra{m^*} \ \Omega^\sigma X  \]
  is equivariantly homotopic to the identity, and the composite
  \[ \Omega^\sigma X\ \xra{\ (\ast,\Id)\ }\ 
    (\Omega^\sigma X)\times (\Omega^\sigma X)\ \xra{m^*} \ \Omega^\sigma X  \]
  is equivariantly homotopic to $\epsilon^*\colon\Omega^\sigma X\to \Omega^\sigma X$.
\end{prop}
\begin{proof}
  We exploit that $C$-equivariant self-maps of $U(1)$, for the action by complex conjugation,
  are characterized up to equivariant based homotopy by their value on the fixed point $-1$
  and by the degree of the underlying non-equivariant map.
  The composite
  \[ U(1) \ \xra{\ m' \ }\ U(1)\vee U(1) \ \xra{\ \nabla\ }\ U(1) \]
  fixes $-1$ and has underlying degree $+1$, so it is
  equivariantly based homotopic to the identity, where $\nabla$ denotes the fold map.
    The composite
  \[ U(1) \ \xra{\ m' \ }\ U(1)\vee U(1) \ \xra{\ p_2\ }\ U(1) \]
  fixes $-1$ and has underlying degree $-1$, so it is
  equivariantly based homotopic to complex conjugation, where $p_2$ denotes the
  projection to the second wedge summand.
  After conjugation with the Cayley transform, these properties become that
  facts that the two composites
  \[ S^\sigma\ \xra{\ m \ }\ S^\sigma\vee S^\sigma \ \xra{\ \nabla\ }\ S^\sigma \text{\qquad and\qquad}
    S^\sigma\ \xra{\ m \ }\ S^\sigma\vee S^\sigma \ \xra{\ p_2\ }\ S^\sigma \]    
  are equivariantly based homotopic to the identity and
  to the sign involution $\epsilon\colon S^\sigma\to S^\sigma$, respectively.
  The claim follows by applying $\map_*(-,X)$ to the maps and homotopies.
\end{proof}

We abuse notation and also write $\epsilon\in \pi_0^C(\mS)$
for the $C$-equivariant stable homotopy class
of the sign involution $\epsilon\colon S^\sigma\to S^\sigma$.
This element satisfies $\epsilon^2=1$ and
is related to the transfer by $\tr_{\{1\}}^C(1)=1-\epsilon$.
The equivariant homotopy sets $[A,E]^\alpha$ are defined in \eqref{eq:[-,-]}.
Part (ii) of the next theorem refers to the
module structure of the group 
\[  [A,\Omega^\bullet Y]^\alpha\ \iso\ \pi_0^G(\map(A,\alpha^*(Y)))\]
over the ring $\pi_0^G(\mS)$.

The following theorem is straightforward for trivially augmented Lie groups.
Indeed, in that case the $G$-map underlying $c\colon\bBUP\to \Omega^\bullet(\Sigma^\infty_+\bP)$
is a loop map, and thus induces an additive map 
$c_* \colon [A,\bBUP]^G \to [A,\Omega^\bullet(\Sigma^\infty_+\bP)]^G$.
Moreover, $\res^C_{\{1\}}(\epsilon)=-1$, so if $\alpha\colon G\to C$ is the trivial
homomorphism, then $\alpha^*(\epsilon)=-1$.
The formula of part (ii) of the following theorem thus becomes $c_*(2x+y)=2 c_*(x)+c_*(y)$.
In contrast, if $\alpha\colon G\to C$ is surjective,
then $\alpha^*(\epsilon)\ne -1$ in $\pi_0^G(\mS)$,
and $c_* \colon [A,\bBUP]^\alpha \to [A,\Omega^\bullet(\Sigma^\infty_+\bP)]^\alpha$
need not be additive, see Example \ref{eg:c_not_additive}.

\begin{theorem}\label{thm:formula_c_*}
  Let $\alpha\colon G\to C$ be an augmented Lie group, and let $A$ be
  a finite $G$-CW-complex.
  \begin{enumerate}[\em (i)]
  \item  Let $\psi\colon \bU\to\bU$ be a morphism in the based $C$-global homotopy category.
    Then the map
  \[ (\Omega^\sigma \psi)_* \ : \  [A,\Omega^\sigma\bU]^\alpha \ \to \ [A,\Omega^\sigma\bU]^\alpha\]
  satisfies the relation
  \[  (\Omega^\sigma\psi)_*(2x-y)\ = \ 2\cdot (\Omega^\sigma\psi)_*(x) - (\Omega^\sigma\psi)_*(y) \]
  for all $x,y\in[A,\Omega^\sigma\bU]^\alpha$.
  \item 
    The map
    \[ c_* \ : \  [A,\bBUP]^\alpha \ \to \ [A,\Omega^\bullet(\Sigma^\infty_+\bP)]^\alpha\]
  satisfies the relation
  \[  c_*(2x+y)\ = \ (1-\alpha^*(\epsilon))\cdot c_*(x)\ +\ c_*(y) \]
  for all $x,y\in[A,\bBUP]^\alpha$.
\end{enumerate}
\end{theorem}
\begin{proof}
  We start with a preliminary observation.
  We let $M$ be an abelian group endowed with a group homomorphism $\mu\colon M\times M\to M$
  such that $\mu(x,x)=x$ for all $x\in M$.
  Then
  \begin{equation}\label{eq:pseudo_in_Ab}
          \mu(x,y)\ = \ \mu((x,x)-(0,x)+(0,y))\ = \ x- \mu(0,x) + \mu(0,y)
  \end{equation}
  for all $x,y\in M$, by the homomorphism property.

  (i) The binary operation $m^*$ on $\Omega^\sigma X$ defined in \eqref{eq:define_m^*}
  is natural for morphisms of Real-global spaces in $X$. In particular, it is
  natural for the multiplication morphism $\bU\boxtimes \bU\to \bU$.
  Hence the map 
  \[ m^* \ = \ [A,m^*]^\alpha\ : \
    [A,\Omega^\sigma\bU]^\alpha\times [A,\Omega^\sigma\bU]^\alpha \ \to \ [A,\Omega^\sigma\bU]^\alpha \]
  is a homomorphism of abelian groups.
  Proposition \ref{prop:effect_of_m} shows that $m^*(x,x)=x$ and
  $m^*(0,x)=[A,\epsilon^*]^\alpha(x)$.
  Relation \eqref{eq:pseudo_in_Ab} and Proposition \ref{prop:u_on_Omega^sigma U}
  thus yield
  \[ m^*(x,y)\ =\ x -m^*(0,x) + m^*(0,y)\ = \
     x -[A,\epsilon^*]^\alpha(x) + [A,\epsilon^*]^\alpha(y)\ = \ 2x-y \]
  for all $x,y\in [A,\Omega^\sigma\bU]^\alpha$.
  The map $(\Omega^\sigma \psi)_*$ is compatible with all operations
  that are natural for $C$-global $\sigma$-loop spaces.
  This includes in particular the binary operation $m^*$. So
  \begin{align*}
        (\Omega^\sigma\psi)_*(2x-y)\
    &= \ (\Omega^\sigma\psi)_*(m^*(x,y))\\
    &= \ m^*((\Omega^\sigma\psi)_*(x),(\Omega^\sigma\psi)_*(y))\
    =\ 2\cdot (\Omega^\sigma \psi)_*(x)-(\Omega^\sigma \psi)_*(y)
  \end{align*}
    for all $x,y\in [A,\Omega^\sigma\bU]^\alpha$.  

    (ii)
    We claim that  the map
    $[A,\epsilon^*]^\alpha\colon[A,\Omega^\sigma(\Omega^\bullet Y)]^\alpha\to [A,\Omega^\sigma(\Omega^\bullet Y)]^\alpha$
    is multiplication by the class $\alpha^*(\epsilon)\in \pi_0^G(\mS)$,
     for every orthogonal $C$-spectrum  $Y$.
 This is almost a tautology, and a special case of a more
  general fact: for every orthogonal representation $V$ of a compact Lie group $G$,
  and for every continuous based $G$-map $f\colon S^V\to S^V$,
  precomposition with $f$ and multiplication by $[f]\in\pi_0^G(\mS)$
  coincide on $\pi_V^G(X)$ for every orthogonal $G$-spectrum $X$.

  Also for every orthogonal $C$-spectrum $Y$, the map
  \[ m^* \ = \ [A,m^*]^\alpha\ : \
    [A,\Omega^\sigma(\Omega^\bullet Y)]^\alpha\times [A,\Omega^\sigma(\Omega^\bullet Y)]^\alpha \ \to \ [A,\Omega^\sigma(\Omega^\bullet Y)]^\alpha \]
  is additive for the abelian group structure arising from stability.
  Relation \eqref{eq:pseudo_in_Ab} and the claim above thus yield the relation
  \[ m^*(x,y)\ =\ x -m^*(0,x) + m^*(0,y)\ = \
    x -[A,\epsilon^*]^\alpha(x) + [A,\epsilon^*]^\alpha(y)\ = \
    (1-\alpha^*(\epsilon))\cdot x +\alpha^*(\epsilon)\cdot y \]
  for all $x,y\in [A,\Omega^\sigma(\Omega^\bullet Y)]^\alpha$.
  The map
  \[  (\Omega^\sigma d)_* = [A, \Omega^\sigma d]^\alpha\ : \  [A,\Omega^\sigma\bU]^\alpha\ \to \ 
    [A,\Omega^\sigma(\Omega^\bullet(\sh^\sigma(\Sigma^\infty_+\bP)))]^\alpha \]
  is compatible with all operations that are natural for $C$-global $\sigma$-loop spaces.
  This includes in particular the binary operation $m^*$. So
  \begin{align}\label{eq:Omega_d_m}
    (\Omega^\sigma d )_*(2x-y)\
    &= \ (\Omega^\sigma d )_*(m^*(x,y))\\
    &= \ m^*((\Omega^\sigma d )_*(x),(\Omega^\sigma d )_*(y))\nonumber\\
    &=\ (1 - \alpha^*(\epsilon))\cdot (\Omega^\sigma d)_*(x)\ +\
       \alpha^*(\epsilon)\cdot (\Omega^\sigma d)_*(y)\nonumber
  \end{align}
    for all $x,y\in [A,\Omega^\sigma\bU]^\alpha$.
    
  The global Segal--Becker splitting $c$ was defined by the
  commutative square \eqref{eq:define_c}. It induces a commutative diagram
  \[ \xymatrix@C=30mm{
[A,\bBUP]^\alpha \ar[r]^-{[A,c]^\alpha} \ar[d]_{[A,\gamma]^\alpha}^-\iso & [A,\Omega^\bullet(\Sigma^\infty_+\bP)]^\alpha\ar[d]^{[A,\Omega^\bullet(\tilde\lambda^\sigma_{\Sigma^\infty_+\bP})]^\alpha}_\iso\\
    [A,\Omega^\sigma\bU]^\alpha\ar[r]_-{[A, \Omega^\sigma d]^\alpha} &   [A,\Omega^\bullet(\Omega^\sigma(\sh^\sigma(\Sigma^\infty_+\bP)))]^\alpha
    } \]
  in which all objects are abelian groups, but the horizontal maps are
  {\em not} generally additive.
  The morphism $\gamma$ is a Real-global equivalence of ultra-commutative $C$-monoids,
  so the induced bijection is additive.
  The $C$-global equivalence $\Omega^\bullet(\tilde\lambda^\sigma_{\Sigma^\infty_+\bP})$
  arises from a $C$-global stable map,
  so the induced bijection is additive and compatible with multiplication
  by the class $\alpha^*(\epsilon)$.
  So relation \eqref{eq:Omega_d_m} for the lower horizontal map
  implies the analogous relation for the upper horizontal map $c_*=[A,c]^\alpha$, i.e.,
  \[  c_*(2x-y)\ =\ (1 - \alpha^*(\epsilon))\cdot c_*(x)\ +\ \alpha^*(\epsilon)\cdot c_*(y) \]
  for all $x,y\in[A,\bBUP]^\alpha$.
  Setting $x=0$ yields $c_*(-y)=\alpha^*(\epsilon)\cdot c_*(y)$, and thus
  \begin{align*}
    c_*(2x+y)\
    &=\  c_*(2x-(-y))\\
    &=\  (1-\alpha^*(\epsilon))\cdot c_*(x)+\alpha^*(\epsilon)\cdot c_*(-y)\
    = \  (1-\alpha^*(\epsilon))\cdot c_*(x) + c_*(y)\    . \qedhere 
  \end{align*}
\end{proof}

  \begin{eg}[$c_*$ is not additive]\label{eg:c_not_additive}
    In the special case of the group $C$ augmented by the identity, and for $A=\ast$,
    the global Segal--Becker splitting becomes a map
    \[ c_*\ : \  \pi_0^C(\bBUP)\ \to \ \pi_0^C(\Sigma^\infty_+\bP) \ .\]
    In this case the group $\pi_0^C(\bBUP)\iso RR(C)$ is infinite cyclic,
    generated by the class $1=h_*(\res^{\tilde T}_C(u_{\tilde T}))$,
    where  $u_{\tilde T}\in\pi^{\tilde T}_0(\bP)$
    is the unstable tautological class \eqref{eq:define_u_T},
    and $h\colon\bP\to \bBUP$ is the `inclusion of line bundles'  defined in \eqref{eq:define_h}.
    Moreover, $\pi_0^C(\Sigma^\infty_+\bP)$ is a free module of rank~1
    over the ring $\pi_0^C(\mS)=\mZ[\epsilon]/(\epsilon^2-1)$,
    also known as the Burnside ring of the group $C$.
    We have $c_*(0)=0$ and 
    \[ c_*(1)\ =\ c_*(h_*(\res^{\tilde T}_C(u_{\tilde T})))
      \ =\ \res^{\tilde T}_C(c_*(h_*(u_{\tilde T})))
      \ =\ \res^{\tilde T}_C(e_{\tilde T})\ = \ 1 \ . \]
    The third equation uses that $c\circ h\colon\bP\to\Omega^\bullet(\Sigma^\infty_+\bP)$
    is the adjunction unit, see Corollary \ref{cor:c_on_P}.
    Theorem \ref{thm:formula_c_*} (ii) provides the relation
    \[  c_*(2)\ = \  (1-\epsilon)\cdot c_*(1)\ = \ 1- \epsilon \ \ne \ 2\ = \ 2\cdot c_*(1)\ .\]
    In particular, the map $c_*$ is {\em not} additive.
\end{eg}

\section{The \texorpdfstring{$G$}{G}-equivariant Segal--Becker splitting and explicit Brauer induction}

In this section we show that our global Segal--Becker splitting
$c\colon \bBUP\to \Omega^\bullet(\Sigma^\infty_+\bP)$
rigidifies and globalizes the classical Segal--Becker
splittings at the level of equivariant cohomology theories,
and that it induces the Boltje--Symonds `explicit Brauer induction' on equivariant homotopy groups.
The first fact is Theorem C of the introduction, and Theorem \ref{thm:rigidifies} below;
the second fact is Theorem D of the introduction, and Corollary \ref{cor:c_on_RG} below.

\begin{construction}[The equivariant Segal--Becker splitting]\label{con:classical SB}
  We let $\alpha\colon G\to C$ be an augmented Lie group.
  We recall the $\alpha$-equivariant Segal--Becker splitting
  via equivariant transfers due to Iriye--Kono \cite[\S 3]{iriye-kono:Segal-Becker KG},
  following Crabb's presentation \cite{crabb:brauer}.
  Crabb only discusses the construction
  for complex vector bundles, which corresponds to the special case where
  the augmentation $\alpha$ is trivial, and hence need not be mentioned.
  The construction generalizes to general augmented Lie groups;
  however, a subtlety arises when passing from vector bundles to virtual vector bundles,
  due to the non-additivity of the Segal--Becker splitting for surjectively augmented Lie groups.

  We let $\xi\colon E\to A$ be a Real $\alpha$-vector bundle
  over a finite $G$-CW-complex.  We denote by
  \begin{equation}\label{eq:Pxi}
    P\xi \ : \ P E\ \to \ A      
  \end{equation}
  the projectivized bundle; its fiber over $a\in A$ is the projective space
  of the complex vector space $E_a=\xi^{-1}(\{a\})$.
  The total space $P E$ inherits a continuous $G$-action, and the projection
  to $A$ is $G$-equivariant.
  The projection \eqref{eq:Pxi} thus has an associated transfer,
  a morphism in the homotopy category of
  genuine $G$-spectra from $\Sigma^\infty_+ A$ to $\Sigma^\infty_+ P E$.
  Since $A$ is a finite $G$-CW-complex, this transfer is represented by
  a continuous based $G$-map
  \[ \tau(P\xi)\ : \ S^V\sm A_+ \ \to \ S^V \sm (P E)_+\ ,\]
  for some orthogonal $G$-representation $V$.
  The tautological complex line bundle over $P E$
  is a Real $\alpha$-line bundle, and thus classified by a continuous $G$-map 
  \[ \kappa\ :\ P E\ \to  \ \alpha^\flat(\bP(W)) \]
  for some sufficiently large orthogonal $G$-subrepresentation $W$ of $\Uc_G$.
  By enlarging, if necessary, we may assume that $V=W$.
  We write
  \[ \vartheta(\xi)\ \in \ [A,\Omega^\bullet(\Sigma^\infty_+ \bP)]^\alpha  \]
  for the class of the adjoint to the composite
  \[  S^V\sm A_+\ \xra{\tau(P\xi)} \ S^V\sm (P E)_+\
    \xra{S^V\sm\kappa_+} \ S^V\sm \alpha^\flat(\bP(V))_+ \ .\]
  The class $\vartheta(\xi)$ only depends on the isomorphism class of
  the $\alpha$-vector bundle $\xi$. So the construction provides a well-defined map
  \begin{equation} \label{eq:define_vartheta}
    \vartheta_{\alpha,A}\ : \  \Vect_\alpha^R(A)\ \to \ [A,\Omega^\bullet(\Sigma^\infty_+ \bP)]^\alpha
  \end{equation}
  that is natural for continuous $G$-maps in $A$, and for restriction along
  morphisms of augmented Lie groups.

  At this point a possibly unexpected feature appears for surjectively augmented Lie groups.
  If the augmentation is trivial, then Real $\alpha$-vector bundles
  are just complex $G$-vector bundles, and we are in the context discussed by
  Crabb \cite{crabb:brauer}.
  In this case the map \eqref{eq:define_vartheta} is additive
  for the Whitney sum of vector bundles, see \cite[Lemma 2.6]{crabb:brauer}.
  It thus extends uniquely to an additive map
  on the group completion $K U_G(A)$, the complex $G$-equivariant K-group of $A$.

  \Danger
  As we show in Remark \ref{rk:nonadditive} below, the equivariant Segal--Becker
  splitting \eqref{eq:define_vartheta} is {\em not} additive
  whenever the augmentation $\alpha\colon G\to C$ is surjective.
  Consequently, one cannot appeal to the group completion property
  to extend \eqref{eq:define_vartheta} from $\Vect_\alpha^R(A)$ to $KR_\alpha(A)$.
  We will show in Corollary \ref{cor:vartheta_nonadd} that the map
  \eqref{eq:define_vartheta} enjoys a weakened form of additivity that still allows
  to extend it meaningfully to the Grothendieck group, see \eqref{eq:K_alpha2P}.
\end{construction}

The universal case
of a Real-equivariant vector bundle of dimension $k$ is
the tautological $\tilde U(k)$-representation $\nu_k$,
considered as a Real $\tilde U(k)$-equivariant vector bundle over a point.
We will now identify the equivariant Segal--Becker splitting in this case.

\begin{prop}\label{prop:SB_on_nu_n}
  Let $[\nu_k]\in \Vect_{\tilde U(k)}^R(\ast)$ denote the class of the   
  tautological Real $\tilde U(k)$-representation on $\mC^k$,
  considered as a Real $\tilde U(k)$-equivariant vector bundle over a point.
  Then
  \[ \vartheta_{\tilde U(k),\ast}[\nu_k]\ = \
    \tr^{\tilde U(k)}_{\tilde U(k-1,1)}(1\times e_{\tilde T}) \]
  in $\pi_0^{\tilde U(k)}(\Sigma^\infty_+\bP)$,
  where $e_{\tilde T}\in \pi_0^{\tilde T}(\Sigma^\infty_+ \bP)$
  is the stable tautological class \eqref{eq:define_u_T}.
\end{prop}
\begin{proof}
  The group $\tilde U(k)$ acts transitively on the projective space $P(\nu_k)$,
  and the complex line
  \[ l \ = \ \mC\cdot(0,\dots,0,1) \]
  has stabilizer group $\tilde U(k-1,1)$.
  The action on $l$ thus identifies $P(\nu_k)$ with the coset space $\tilde U(k)/\tilde U(k-1,1)$.
  So the equivariant transfer
  \[ \tau(P(\nu_k))\ : \ \mS \ \to \ \Sigma^\infty_+ P(\nu_k) \]
  associated with the unique $\tilde U(k)$-map $P(\nu_k)\to\ast$
  sends $1\in \pi_0^{\tilde U(k)}(\mS)$ to the class
  \[  \tr^{\tilde U(k)}_{\tilde U(k-1,1)}(\sigma^{\tilde U(k-1,1)}[l])\ \in \pi_0^{\tilde U(k)}(\Sigma^\infty_+ P(\nu_k)) \ , \]
  where $[l]\in \pi_0(P(\nu_k)^{\tilde U(k-1,1)})$
  is the class of the $\tilde U(k-1,1)$-fixed point $l$, and 
  \[ \sigma^{\tilde U(k-1,1)}\ :\ \pi_0(P(\nu_k)^{\tilde U(k-1,1)})\ \to\
    \pi_0^{\tilde U(k-1,1)}(\Sigma^\infty_+ P(\nu_k)) \]
  is the stabilization map \cite[(3.3.12)]{schwede:global}.

  The stabilizer group $\tilde U(k-1,1)$ acts on the invariant line $l$ through
  the homomorphism $p_2\colon \tilde U(k-1,1)\to \tilde T$
  that projects to the second block;
  so the classifying $\tilde U(k)$-map 
  \[ \kappa\ : \  P(\nu_k)\ \to\  \bP(\Uc_{U(k)}) \]
  for the tautological line bundle satisfies
  $\kappa_*[l]=     p_2^*(u_{\tilde T})$ in $\pi_0^{\tilde U(k-1,1)}(\bP)$,
  where $u_{\tilde T}\in \pi_0^{\tilde T}(\bP)$
  is the unstable tautological class \eqref{eq:define_u_T}. Thus
  \[ (\Sigma^\infty_+\kappa)_*(\sigma^{\tilde U(k-1,1)}[l])\
    = \  \sigma^{\tilde U(k-1,1)}(\kappa_*[l])\
    = \  \sigma^{\tilde U(k-1,1)}(p_2^*(u_{\tilde T}))\ = \
    p_2^*(\sigma^{\tilde T}(u_{\tilde T}))\ = \  1\times e_{\tilde T}\ . \]
  Combining these observations yields
  \begin{align*}
    \vartheta_{\tilde U(k),\ast}[\nu_k]\
    &= \  ( (\Sigma^\infty_+\kappa)\circ \tau(P(\nu_k)))_*( 1) \
    = \   (\Sigma^\infty_+\kappa)_*(\tr^{\tilde U(k)}_{\tilde U(k-1,1)}(\sigma^{\tilde U(k-1,1)}[l])) \\
    &= \  \tr^{\tilde U(k)}_{\tilde U(k-1,1)}((\Sigma^\infty_+ \kappa)_*(\sigma^{\tilde U(k-1,1)}[l]))\
    = \  \tr^{\tilde U(k)}_{\tilde U(k-1,1)}(1\times e_{\tilde T})\ . \qedhere
    \end{align*}
\end{proof}

\begin{rk}[Beware non-additivity]\label{rk:nonadditive}
  We show that the map $\vartheta_{\alpha,A}$ from \eqref{eq:define_vartheta}
  is {\em not} additive for augmented Lie groups with non-trivial augmentation.
  This is in contrast to trivially augmented Lie groups,
  where the map is additive by \cite[Lemma 2.6]{crabb:brauer}.
  We use the double coset formula
  for $\res^{\tilde U(2)}_C\circ\tr_{\tilde U(1,1)}^{\tilde U(2)}$,
  where $C$ sits inside $\tilde U(2)$ as complex conjugation.
  The general double coset formula for $\res^G_K\circ \tr_H^G$
  can be found in \cite[IV Section 6]{lms} or \cite[Theorem 3.4.9]{schwede:global},
  and we need to specialize it.
  The coset space $\tilde U(2)/\tilde U(1,1)$ is homeomorphic
  to the projective space $P(\mC^2)$, and the $C$-action on 
  $\tilde U(2)/\tilde U(1,1)$ by left translation corresponds to
  the action on $P(\mC^2)$ by complex conjugation.
  The double coset space $C\backslash\tilde U(2)/\tilde U(1,1)$ is
  thus homeomorphic to a 2-disc, stratified by its boundary
  (coming from the $C$-fixed points) and the interior (corresponding to the points of
  $\tilde U(2)/\tilde U(1,1)$ on which $C$ acts freely).
  The $C$-fixed points $(\tilde U(2)/\tilde U(1,1))^C\iso P(\mC^2)^C$ are homeomorphic to $S^1$,
  thus have Euler characteristic 0, and do not contribute to the double coset formula.
  The non-singular part is an open 2-disc, with internal Euler characteristic
  \[ \chi^\sharp( C\backslash P(\mC^2)_{\text{free}}) \ = \
     \chi(D^2) - \chi(\partial D^2) \ = \ 1\ .  \]
   So the double coset formula becomes
   \[ \res^{\tilde U(2)}_C\circ\tr_{\tilde U(1,1)}^{\tilde U(2)}\ = \
     \tr_{\{1\}}^C\circ     \res^{\tilde U(1,1)}_{\{1\}}\ .\] 

  Now we consider the group $C$, augmented over itself by the identity.
We show that the map
$\vartheta_{C,\ast}\colon\Vect^R_C(\ast) \to \pi_0^C(\Sigma^\infty_+\bP)$ 
is not additive.
The ring $\pi_0^C(\Sigma^\infty_+ \bP)$ is isomorphic to $\mZ[\epsilon]/(\epsilon^2-1)$,
for $\epsilon$ the class of the sign involution of $S^\sigma$.
The monoid $\Vect^R_C(\ast)$ of isomorphism classes of Real $C$-representations
is isomorphic to $\mN$, generated by the class of $\mC$ with $C$-action by complex conjugation;
another name for this generator is $\res^{\tilde T}_C[\nu_1]$,
with $C$ embedded in $\tilde T$ as complex conjugation.
Proposition \ref{prop:SB_on_nu_n} thus yields
\[ \vartheta_{C,\ast}[\mC]\ = \
  \vartheta_{C,\ast}(\res^{\tilde T}_C[\nu_1])\
  =\   \res^{\tilde T}_C( \vartheta_{\tilde T,\ast}[\nu_1])\
  =\   \res^{\tilde T}_C(e_{\tilde T})\  =\ 1\ . \]
  We have $[\mC]\oplus [\mC]\iso\res^{\tilde U(2)}_C(\nu_2)$
  as Real $C$-representations. So by Proposition \ref{prop:SB_on_nu_n}
  and the double coset formula,
\begin{align*}
 \vartheta_{C,\ast}(2\cdot [\mC])\
  &=\   \vartheta_{C,\ast}(\res^{\tilde U(2)}_C[\nu_2])\
    =\ \res^{\tilde U(2)}_C(\vartheta_{\tilde U(2),\ast}[\nu_2])\
  =\ \res^{\tilde U(2)}_C(\tr^{\tilde U(2)}_{\tilde U(1,1)}(1\times e_{\tilde T}))\\
  &= \ \tr_{\{1\}}^C(\res^{\tilde U(1,1)}_{\{1\}} (1\times e_{\tilde T}))\
  = \ \tr_{\{1\}}^C(1) \ = \ 1 -\epsilon\ \ne \ 2 \ = \ 2\cdot \vartheta_{C,\ast}[\mC]\ .
\end{align*}
In Corollary \ref{cor:vartheta_nonadd} we show more generally
(but by a less direct argument)
that $\vartheta_{\alpha,A}(2\cdot[\xi])= (1-\alpha^*(\epsilon))\cdot \vartheta_{\alpha,A}[\xi]$.
\end{rk}

In Theorem \ref{thm:BUP_represents} we exhibit a natural isomorphism
of abelian monoids
\[ \td{-}\ : \  [A,\bGr]^\alpha\ \xra{\ \iso \ } \ \Vect^R_\alpha(A) \]
for every augmented Lie group $\alpha\colon G\to C$ and every finite $G$-CW-complex $A$.
The morphism of ultra-commutative $C$-monoids $i\colon\bGr\to\bBUP$ is
defined in \eqref{eq:Gr2BUP}.

\begin{theorem}\label{thm:rigidifies_Gr}
 For every  augmented Lie group $\alpha\colon G\to C$ and every finite $G$-CW-complex $A$,
 the following diagram commutes:
    \[ \xymatrix@C=20mm{
        [A,\bGr]^\alpha \ar[r]^-{\td{-}}_-{\iso\ \eqref{eq:Gr_to_Vect}} \ar[d]_{[A,i]^\alpha} & \Vect^R_\alpha(A)\ar[d]^{\vartheta_{\alpha,A}} \\
 [A,\bBUP]^\alpha\ar[r]_-{[A,c]^\alpha}  &    [A,\Omega^\bullet(\Sigma^\infty_+\bP)]^\alpha    } \]
\end{theorem}
\begin{proof}
  Theorem \ref{thm:special_vartheta}
  and Proposition \ref{prop:SB_on_nu_n} show that
  \begin{equation}\label{eq:special_vartheta}
 (c\circ i)_*\{\nu_k\}\ =\ \tr_{\tilde U(k-1,1)}^{\tilde U(k)}(1\times e_{\tilde T})
    \ = \   \vartheta_{\tilde U(k),\ast}[\nu_k]\ .
  \end{equation}
  The map $\td{-}$ from \eqref{eq:Gr_to_Vect} takes $\{\nu_k\}$ to the isomorphism
  class of $\nu_k$, considered as a Real $\tilde U(k)$-vector bundle over a point.
  So equation \eqref{eq:special_vartheta} shows that
  the square commutes for $\alpha=\tilde U(k)$ and $A=\ast$
  on the class $\{\nu_k\}$.
  
  Now we compose both composites in the diagram with the map
  $[A,\bGr_k]^\alpha\to[A,\bGr]^\alpha$ induced by the inclusion $\bGr_k\to\bGr$.
  Varying $\alpha$ and $A$ yields two $C$-global transformations
  from $\bGr_k$ to $\Omega^\bullet(\Sigma^\infty_+\bP)$
  in the sense of Definition \ref{def:global_trans}.
  The orthogonal $C$-space $\bGr_k$ is a Real-global classifying space
  for the augmented Lie group $\tilde U(k)$, and the class $\{\nu_k\}$
  is the tautological class $u_{\tilde U(k),\ast}$ in the sense of 
  \eqref{eq:u_beta}.
  Since the $C$-global transformations coincide for $\alpha=\tilde U(k)$ and $A=\ast$
  on the class of $\nu_k$, the two $C$-global transformations agree altogether
  by the uniqueness part of
  Theorem \ref{thm:global2[B,E]^K}. In other words: the diagram commutes for
  all classes in the image of $[A,\bGr_k]^\alpha$ for some $k\geq 0$.

  Now we suppose that $A$ is `$G$-connected' in the sense
  that the group $\pi_0(G)$ acts transitively on $\pi_0(A)$.
  This ensures that every Real $\alpha$-vector bundle over $A$ has constant rank,
  and every class in $[A,\bGr]^\alpha$ is in the image of $[A,\bGr_k]^\alpha$ for some $k\geq 0$.
  Hence the diagram commutes for such $A$.
  
  A general finite $G$-CW-complex decomposes as  a disjoint union
  $A=A_1\amalg\ldots\amalg A_m$ of $G$-connected
  finite $G$-CW-complexes, indexed by the $\pi_0(G)$-orbits of $\pi_0(A)$.
  We write $\iota_j\colon A_j\to A$ for the inclusion of the $j$-th summand.
  Then the map
  \[ (\iota_1^*,\dots,\iota_m^*)\ : \ [A,\Omega^\bullet(\Sigma^\infty_+\bP)]^\alpha \ \to \ [A_1,\Omega^\bullet(\Sigma^\infty_+\bP)]^\alpha \times\dots\times [A_m,\Omega^\bullet(\Sigma^\infty_+\bP)]^\alpha  \]
  is bijective. Hence it suffices to show that the diagram commutes after postcomposition
  with $\iota_j^*$ for each $1\leq j\leq m$. But this is the case by naturality
  for  $\iota_j\colon A_j\to A$, and because the diagram commutes for $A_j$
  by the previous case.
\end{proof}

Theorem \ref{thm:formula_c_*} (ii) and the additivity of
the map $i_*\colon [A,\bGr]^\alpha\to[A,\bBUP]^\alpha$
provide the relation
\[  (c\circ i)_*(2x+y)\ = \ (1-\alpha^*(\epsilon))\cdot (c\circ i)_*(x)\ +\ (c\circ i)_*(y) \]
in the group $[A,\Omega^\bullet(\Sigma^\infty_+\bP)]^\alpha$, for all
classes $x,y\in [A,\bGr]^\alpha$.
Here $\epsilon\in \pi_0^C(\mS)$ is the class of the sign involution
$\epsilon\colon S^\sigma\to S^\sigma$, so that
\[ (1-\alpha^*(\epsilon))\ = \
  \begin{cases}
   \qquad  2 & \text{ if $\alpha$ is trivial, and}\\
    \tr_{\ker(\alpha)}^G(1) & \text{ if $\alpha$ is surjective.}
  \end{cases}
\]
The map $\td{-} \colon[A,\bGr]^\alpha \to\Vect_\alpha^R(A)$
is an isomorphism of abelian monoids by Theorem \ref{thm:BUP_represents} (ii).
So the commutative diagram just established in Theorem \ref{thm:rigidifies_Gr}
implies:

\begin{cor}\label{cor:vartheta_nonadd}
  For every augmented Lie group $\alpha\colon G\to C$
  and all Real $\alpha$-equivariant vector bundles $\xi$ and $\zeta$ over a
  finite $G$-CW-complex $A$,  the $\alpha$-equivariant
  Segal--Becker splitting \eqref{eq:define_vartheta} satisfies the relation
  \[  \vartheta_{\alpha,A}(2\cdot[\xi]+ [\zeta])\ = \
    (1-\alpha^*(\epsilon))\cdot \vartheta_{\alpha,A}[\xi]\ +\ \vartheta_{\alpha,A}[\zeta] \ .\]
\end{cor}

Despite the fact that the Segal--Becker splitting \eqref{eq:define_vartheta}
it is not generally additive, we can now extend it meaningfully
and naturally 
from the monoid $\Vect_\alpha^R(A)$ to its group completion $K R_\alpha(A)$.

\begin{construction}
  We let $\alpha\colon G\to C$ be an augmented Lie group,
  and we let $A$ be a finite $G$-CW-complex.
  Corollary \ref{cor:vartheta_nonadd} in particular shows that the relation
   \begin{equation} \label{eq:2-relation}
     \vartheta_{\alpha,A}(2 x +y)\ = \
     \vartheta_{\alpha,A}(2 x)\ +\ \vartheta_{\alpha,A}(y)      
   \end{equation}
   holds for all classes $x,y\in\Vect_\alpha^R(A)$.
   We define a map 
   \begin{equation}\label{eq:K_alpha2P}
     \vartheta_{\alpha,A} \ : \ K R_\alpha(A)\ \to \ [A,\Omega^\bullet(\Sigma^\infty_+ \bP)]^\alpha
   \end{equation}
   by
   \[   \vartheta_{\alpha,A}(x-y) \ = \ \vartheta_{\alpha,A}(x+y)- \vartheta_{\alpha,A}(2 y)\ , \]
   for $x,y\in\Vect_\alpha^R(A)$.
   For trivial augmentations, the Segal--Becker splitting \eqref{eq:define_vartheta}
   is additive, and this formula is a slightly over-complicated way to define
   the unique additive extension to the group completion.

   We omit the straightforward verification, using the relation \eqref{eq:2-relation},
   that this assignment is well-defined,
   i.e., independent of how a class in $K R_\alpha(A)$ is presented as the formal difference
   of two Real vector bundles.
   Given well-definedness, the map \eqref{eq:K_alpha2P} is clearly an
   extension of the earlier geometric construction \eqref{eq:define_vartheta},
   which justifies that we give it the same name.
   Like the map \eqref{eq:define_vartheta}, the extension \eqref{eq:K_alpha2P}
   is additive whenever the augmentation is trivial, but it is {\em not} additive
   whenever the augmentation is surjective. 
   In any case,
   the extension \eqref{eq:K_alpha2P} again satisfies the relation \eqref{eq:2-relation},
   but now for all K-theory classes $x,y\in KR_\alpha(A)$.
\Danger One might be tempted to define the extension \eqref{eq:K_alpha2P}
by sending a formal difference $x-y$ to
$\vartheta_{\alpha,A}(x)- \vartheta_{\alpha,A}(y)$.
However, this assignment would not be well-defined when the augmentation is surjective.
\end{construction}

Now we proceed to prove Theorem C of the introduction,
saying that our global Segal--Becker splitting
induces the classical equivariant Segal--Becker splitting
at the level of equivariant cohomology theories.
In Theorem \ref{thm:BUP_represents} we exhibit a natural isomorphism
of abelian groups
\[ \td{-}\ : \  [A,\bBUP]^\alpha\ \xra{\ \iso \ } \ KR_\alpha(A) \]
for every augmented Lie group and every finite equivariant CW-complex.

\begin{theorem}\label{thm:rigidifies}
  For every augmented Lie group  $\alpha\colon G\to C$ and every finite $G$-CW-complex $A$,
  the map $[A,c]^\alpha\colon[A,\bBUP]^\alpha \to[A,\Omega^\bullet(\Sigma^\infty_+ \bP)]^\alpha$
  coincides with the composite
  \[  [A,\bBUP]^\alpha \ \xra[\iso]{\ \td{-}\ } \ KR_\alpha(A)
    \ \xra{\vartheta_{\alpha,A}}\
    \ [A,\Omega^\bullet(\Sigma^\infty_+ \bP)]^\alpha \ .\]
  \end{theorem}
\begin{proof}
  The map $i_*\colon[A,\bGr]^\alpha\to[A,\bBUP]^\alpha$ is a group completion of abelian monoids by
  Proposition \ref{prop:pointwise_groupcomplete}.
  So every class in $[A,\bBUP]^\alpha$ is of the form $i_*(x)-i_*(y)$
  for some $x,y\in[A,\bGr]^\alpha$.
  The commutative diagram established in Theorem \ref{thm:rigidifies_Gr}
  and the commutative diagram \eqref{eq:bGr to bBRP square}
  show that $c_*=[A,c]^\alpha$ and $\vartheta_{\alpha,A}\circ\td{-}$
  coincide on all classes in the image of $i_*$,
  in particular on the classes $i_*(x+y)$ and $i_*(2 y)$. Hence
  \begin{align*}
    c_*(i_*(x)-i_*(y))+  c_*(i_*(2 y))  \
    &= \ c_*(i_*(x+y))\
    = \ \vartheta_{\alpha,A}\td{i_*(x+ y)} \\
    _{\eqref{eq:K_alpha2P}}    &= \  \vartheta_{\alpha,A}\td{i_*(x)-i_*(y)}+ \vartheta_{\alpha,A}(2 \td{i_*(y)})\\
  &= \  \vartheta_{\alpha,A}\td{i_*(x)-i_*(y)}+  c_*(i_*(2 y))\ .
  \end{align*}
  The first equality is Theorem \ref{thm:formula_c_*} (ii).
  Subtracting $c_*(i_*(2 y))$ shows that the two maps in question coincide
  on the generic element $i_*(x)-i_*(y)$ of $[A,\bBUP]^\alpha$. 
\end{proof}

\begin{rk}\label{rk:KonoIriye_inconsistency}

  We augment the group $C=\Gal(\mC/\mR)$ by the identity, so that 
  $KR_C(X)$ becomes Atiyah's Real K-theory of a $C$-space $X$.
  Kono asserts in \cite[Theorem 1.1]{kono:Segal-becker KR}
  (in different notation) that the map 
  \[  \eta_*\ : \ [\Sigma^\infty_+ X,(\Sigma^\infty_+\bP)]^C\ \to\
    [\Sigma^\infty_+ X,\bKR]^C\ \iso \ KR_C(X) \]
  is a split epimorphism whenever $X$ is a compact $C$-space.
  Shortly thereafter, Iriye and Kono considered the more general situation
  of the semi-direct product  $\tilde G=G\rtimes_\tau C$ of
  a finite group $G$ by a multiplicative involution $\tau$,
  augmented to $C$ by the projection; they assert in
  \cite[Theorem 1']{iriye-kono:Segal-Becker KG}
  (again translated into our notation)
  that for compact $\tilde G$-spaces $X$ `there exists a split epimorphism'
  \begin{equation}\label{eq:2be_split}
  \eta_*\ : \ [\Sigma^\infty_+ X,(\Sigma^\infty_+\bP)]^{\tilde G}\ \to\
    [\Sigma^\infty_+ X,\bKR]^{\tilde G}\ \iso \ KR_{\tilde G}(X) \ .    
  \end{equation}
  \Danger
  The usual interpretation of `split epimorphism' is that of a homomorphism
  that admits an {\em additive} section.
  With that interpretation, the proofs of \cite[Theorem 1.1]{kono:Segal-becker KR}
  and \cite[Theorem 1']{iriye-kono:Segal-Becker KG} are flawed.
  Indeed, the transfer techniques employed by Kono in \cite{kono:Segal-becker KR}
  and by Iriye and Kono in \S 2 of their paper are just different expositions
  of special cases of the transfer construction \eqref{eq:define_vartheta};
  but we show in Remark \ref{rk:nonadditive} that this construction is {\em not}
  additive whenever there is a Real direction, i.e., for surjectively
  augmented Lie groups. This issue is already present in the simplest Real situation,
  i.e., for $C$ augmented by the identity.

  The proofs of \cite[Theorem 1.1]{kono:Segal-becker KR}
  and of \cite[Theorem 1]{iriye-kono:Segal-Becker KG}
  both start with an illegitimate reduction step:
  we find the statements ``we may assume $\xi$ is an $n$-dimensional Real vector bundle over $X$''
  and 
  ``So we may assume that $x\in K_G(X)$ is an $n$-dimensional $G$-vector bundle over $X$'',
  respectively.
  To be able to pass from vector bundles of fixed rank
  to general elements in the Grothendieck group, {\em some} argument is needed.
  When dealing with additive constructions, the universal property of the
  Grothendieck construction kicks in, often even without explicit mentioning.
  But for non-additive constructions such as the present one \eqref{eq:define_vartheta},
  some other justification is needed, for example
  the weak form of additivity from Corollary \ref{cor:vartheta_nonadd};
  and clearly, if one extends a non-additive map, the extension will not be additive, either.
  
  Our Real-global splitting in Theorem \ref{thm:split}
  in particular shows that the map \eqref{eq:2be_split} is surjective.
  I cannot rule out the possibility that there might exist sections to this map,
  different from ours, that are actually additive,
  possibly even with some naturality properties in $(G,X)$,
  so that the {\em statement} of \cite[Theorem 1']{iriye-kono:Segal-Becker KG}
  might be correct. However, I have no evidence supporting such an expectation.
\end{rk}

We conclude this section with the proof of Theorem D from the introduction,
saying that the effect of the global Segal--Becker splitting
$c\colon\bBUP\to\Omega^\bullet(\Sigma^\infty_*\bP)$ on equivariant homotopy
groups is the {\em explicit Brauer induction} of Boltje and Symonds.

\begin{rk}[Explicit Brauer induction]\label{rk:Brauer induction}
  Brauer showed in \cite[Theorem I] {brauer} 
  that the complex representation ring of a finite group
  is generated, as an abelian group, by representations that are induced from
  1-dimensional representations of subgroups.
  Segal generalized this result 
  to compact Lie groups in \cite[Proposition 3.11 (ii)]{segal:representation}, 
  where `induction' refers to smooth induction. 
  We write $\bA(T,G)$ for the free abelian group with a basis the symbols $[H,\chi]$,
  where $H$ runs over all conjugacy classes
  of closed subgroups of $G$ with finite Weyl group,
  and $\chi\colon H\to T=U(1)$ runs over all characters of $H$.
  The Brauer--Segal theorem can then be paraphrased as the fact that
  for every compact Lie group $G$, the map
  \begin{equation}\label{eq:ev(G)}
   \bA(T,G) \ \to \ R(G)    
  \end{equation}
  that sends $[H,\chi]$ to $\tr_H^G(\chi^*[\nu_1])$ is surjective,
  where $[\nu_1]\in R(T)$
  is the class of the tautological $T$-representation on $\mC$.  
  Informally speaking, an `explicit Brauer induction' is a
  collection of sections to the maps \eqref{eq:ev(G)}
  that are specified by a direct recipe, for example an explicit formula, 
  and with naturality properties as the group $G$ varies.
  So such maps give an `explicit and natural' way to write virtual representations
  as linear combinations of induced representations of 1-dimensional representations.
  What qualifies as `explicit' is, of course, in the eye of the beholder.

  The first explicit Brauer induction 
  was Snaith's formula \cite[Theorem 2.16]{Snaith-Brauer};
  however, Snaith's maps are not additive and not compatible with restriction to subgroups.
  Boltje \cite{boltje:canonical brauer} specified a different explicit Brauer induction
  formula for finite groups by purely algebraic means;
  Symonds \cite{symonds-splitting} gave a topological interpretation of
  the same section in the context of compact Lie groups. 
  Symonds' construction \cite[\S 4]{symonds-splitting} of the sections
  \begin{equation}\label{eq:define_b_G}
    b_G \ : \ R(G) \ \to \ \bA(T,G)     
  \end{equation}
  is designed so that
  \begin{equation} \label{eq:symonds}
    b_{U(k)}[\nu_k]\ = \ [U(k-1,1), p_2] \text{\qquad in\quad} \bA(T,U(k))  
  \end{equation}
  for all $k\geq 1$, where $p_2\colon U(k-1,1)\to U(1)=T$ is the
  projection to the second block.
  The Boltje--Symonds maps are additive and natural for
  restriction along continuous group homomorphisms;
  and the value of $b_G$ at a 1-dimensional representation with
  character $\chi\colon G\to T$ is given by
  \[ b_G[\chi]\ = \ [G,\chi] \ \in \ \bA(T,G)\ .\]
  The Boltje--Symonds maps \eqref{eq:define_b_G} are not (and in fact cannot be)
  in general compatible with transfers. 
\end{rk}

Since the orthogonal space $\bP$ is a global classifying space
for the circle group $T$, the global functor $\upi_0(\Sigma^\infty_+\bP)$
is represented by $T$, see \cite[Proposition 4.2.5]{schwede:global}.
In more down-to-earth terms, this means that the abelian group $\pi_0^G(\Sigma^\infty_+\bP)$
is free, with a basis given by the classes $\tr_H^G(\chi^*(e_T))$,
for $(H,\chi)$ ranging over all conjugacy classes of closed subgroups $H$ of $G$
with finite Weyl group, and all continuous homomorphisms $\chi\colon H\to T$;
see \cite[Corollary 4.1.13]{schwede:global}.
The group $\bA(T,G)$ was defined as a free abelian group with a corresponding basis,
so the map
\begin{equation}\label{eq:A(T)_rep}
\bA(T,G)\ \to \ \pi_0^G(\Sigma^\infty_+\bP)\ , \quad
[H,\chi]\ \longmapsto \ \tr_H^G(\chi^*(e_T))   
\end{equation}
is an isomorphism of abelian groups.

The Boltje--Symonds map $b_G\colon R(G)\to\bA(T,G)$ is the special case
of the equivariant Segal--Becker splitting \eqref{eq:K_alpha2P} for $A=\ast$,
in the sense that the composite
\[ R(G) = K U_G(\ast)\ \xra{\vartheta_{G,\ast}}\  [*,\Omega^\bullet(\Sigma^\infty_+\bP)]^G \ = \
\pi_0^G(\Sigma^\infty_+\bP)\ \iso_{\eqref{eq:A(T)_rep}}\ \bA(T,G)\]
agrees with \eqref{eq:define_b_G}.
Indeed, $b_G$ and $\vartheta_{G,\ast}$ coincide for $G=U(k)$
on the class of the tautological representation $\nu_k$,
by \eqref{eq:symonds} and Proposition \ref{prop:SB_on_nu_n}.
So they agree on arbitrary unitary representations by naturality,
and on virtual representations by additivity.
So the next theorem becomes a special case of Theorem \ref{thm:rigidifies}
for $A=\ast$.

\begin{cor}\label{cor:c_on_RG}
  For every compact Lie group $G$,
  the map $\pi_0^G(c)\colon \pi^G_0(\bBUP)\to\pi_0^G(\Sigma^\infty_+ \bP)$
  equals the composite
  \[
    \pi^G_0(\bBUP)\ \iso _{\eqref{eq:BUP_to_KR_alpha}}\ R(G)
    \ \xra[\eqref{eq:define_b_G}]{\ b_G\ } \ 
    \bA(T,G)    \ \iso_{\eqref{eq:A(T)_rep}}\ \pi_0^G(\Sigma^\infty_+ \bP)\ .
  \]
\end{cor}

\section{Global Adams operations}

In this section we use the global Segal--Becker splitting
to construct Real-global rigidifications \eqref{eq:define_Adams} of
the unstable Adams operations in equivariant K-theory.

\begin{construction}[Adams operations in equivariant K-theory]
  We let $\alpha\colon G\to C$ be an augmented Lie group.
  We recall the construction of the $\lambda$-operations and Adams operations
  on the Grothendieck ring $KR_\alpha(A)$ of Real $\alpha$-equivariant vector bundles
  over a compact $G$-space $A$.
  For all Real $\alpha$-vector bundles $\xi$ and $\zeta$ over $A$,
  the Real $\alpha$-vector bundle $\Lambda^n(\xi\oplus \zeta)$ is isomorphic
  to $\bigoplus_{i=0}^n\Lambda^i(\xi)\tensor_\mC\Lambda^{n-i}(\zeta)$.
  So the map
  \[ \Lambda\ : \ \Vect_\alpha^R(A)\ \to \ KR_\alpha(A)\gh{t}\ , \quad
    [\xi]\ \longmapsto \ {\sum}_{n\geq 0}\, [\Lambda^n(\xi)]\cdot t^n \]
  takes addition in the abelian monoid of isomorphism classes of $\alpha$-vector bundles
  to multiplication in the power series ring $KR_\alpha(A)\gh{t}$. All these power series
  moreover have constant term $\Lambda^0(\xi)=1$, and are thus invertible.
  So $\Lambda$ defines a monoid homomorphism from
  $\Vect_\alpha^R(A)$ to the multiplicative group of the ring $KR_\alpha(A)\gh{t}$.
  The universal property of the Grothendieck construction thus yields
  an extension to a group homomorphism
  \[ \Lambda\ : \ KR_\alpha(A)\ \to \ (KR_\alpha(A)\gh{t})^\times\ .\]
  The $\lambda$-operations
  $\lambda^i\colon KR_\alpha(A)\to KR_\alpha(A)$
  are then defined by
  \[ \Lambda(x)\ = \ {\sum}_{i\geq 0}\, \lambda^i(x)\cdot t^i\ . \]
  By design, these operations extend the exterior powers
  on classes of actual Real vector bundles.
  The $\lambda$-operations
  make the ring $KR_\alpha(A)$ into a special $\lambda$-ring,
  see \cite[Theorem 1.5 (i)]{atiyah-tall}.
  Every special $\lambda$-ring supports {\em Adams operations},
  i.e., ring endomorphisms $\psi^n$ for $n\geq 1$
  that satisfy $\psi^m\circ\psi^n=\psi^{m n}$ for all $m,n\geq 1$,
  as well as the congruence $\psi^p(x)\equiv x^p$ modulo $(p)$ for every prime $p$,
  see \cite[\S 5]{atiyah-tall}.
  We are particularly interested in these Adams operation
  \[ \psi^n\ : \ KR_\alpha(A)\ \to \ KR_\alpha(A) \]
  in the case of Real-equivariant K-theory. One key property
  is that on the class of a Real line bundle $\xi$, the Adams operation is given by
  \[ \psi^n[\xi]\ = \ [\xi^{\tensor n}]\ . \]
  Taking exterior power of vector bundles is natural both for $G$-maps in $A$,
  and for restriction along continuous homomorphisms of augmented Lie groups.
  Hence the $\lambda$-operations and the Adams operations inherit both
  kinds of naturalities.
\end{construction}

\begin{eg}
  The second Adams operation $\psi^2\colon KR_\alpha(A)\to KR_\alpha(A)$ is given
  on the class of a Real $\alpha$-vector bundle $\xi$ by the formula
  \[ \psi^2[\xi] \ = \ [\Sym^2(\xi)] \ - \ [\Lambda^2(\xi)]\ ,\]
  the formal difference of
  the second symmetric and exterior power of $\xi$.
  Indeed, this formula has the correct behavior on line bundles, is additive
  for Whitney sum in $\xi$, and natural in $(\alpha,A)$. So the naturality
  properties force $\psi^2$ to be given by this formula.
  The formula shows that the Adams operations do not generally send vector bundles
  (other than line bundles) to vector bundles, but rather to virtual vector bundles.
\end{eg}

We shall now use the power endomorphisms
of the ultra-commutative monoid $\bP$ and our global Segal--Becker splitting
to define the global Adams operations on $\bBUP$.

\begin{construction}[Global Adams operations]
  For $n\geq 1$, we write 
  \[ \mu_n\ :\  T \ \to\  T\ , \quad \mu_n(\lambda)=\lambda^n\text{\qquad and\qquad}
   \tilde \mu_n= \mu_n\rtimes C\ :\  \tilde T \ \to\  \tilde T\]
 for the $n$-th power homomorphism, and for
 its extension to the extended circle group $\tilde T=T\rtimes C$, respectively.
  Since the pair $(\bP,u_{\tilde T})$ represents the functor $\pi_0^{\tilde T}$,
  we can define a morphism $\phi^n\colon\bP\to\bP$ in the unstable
  $C$-global homotopy category by the requirement that
  \[ \phi^n_*(u_{\tilde T})\ = \ \tilde\mu_n^*(u_{\tilde T}) \]
  in $\pi_0^{\tilde T}(\bP)$. The morphism $\phi^n$ then
  represents raising a line bundle to its $n$-th power.
  We define
  \begin{equation}\label{eq:define_deloop_Adams}
    \kappa^n \ : \ \bU \ \to \ \bU
  \end{equation}
  as the unique morphism in the based unstable $C$-global homotopy category making the
  following diagram commute:
\[ 
 \xymatrix@C=20mm{  \bU \ar[d]_d \ar[rr]^-{\kappa^n} && \bU\ar[d]_\sim^{\Omega^\bullet(\sh^\sigma\eta)\circ d}\\
      \Omega^\bullet(\sh^\sigma(\Sigma^\infty_+\bP))\ar[r]_-{ \Omega^\bullet(\sh^\sigma(\Sigma^\infty_+\phi^n))} &
    \Omega^\bullet(\sh^\sigma(\Sigma^\infty_+\bP))\ar[r]_-{\Omega^\bullet(\sh^\sigma\eta)} &   \Omega^\bullet(\sh^\sigma\bKR)
    } \]
  The $C$-global equivalence $\gamma\colon\bBUP\xra{\sim}\Omega^\sigma\bU$
  is defined in \eqref{eq:define_gamma}.
  The {\em $n$-th global Adams operation}
  \begin{equation}\label{eq:define_Adams}
    \Upsilon^n \ : \ \bBUP \ \to \ \bBUP
  \end{equation}
  is the unique morphism  in the based unstable $C$-global homotopy category making the
  following diagram commute:
  \begin{equation} \begin{aligned}\label{eq:def_upsilon}
    \xymatrix@C=15mm{  \bBUP \ar[d]_\gamma^\sim \ar[r]^-{\Upsilon^n} & \bBUP\ar[d]_\sim^\gamma\\
      \Omega^\sigma\bU\ar[r]_-{\Omega^\sigma(\kappa^n)} & \Omega^\sigma\bU    }      \end{aligned}
  \end{equation}
  Expanding the definition \eqref{eq:define_deloop_Adams} of $\kappa^n$, 
  the definition \eqref{eq:define_gSB} of the Segal--Becker splitting $c$,
  and exploiting that $\Omega^\bullet(\eta)\circ c=\Theta$ by Theorem \ref{thm:split}
  shows that the
  following diagram commutes in the unstable $C$-global homotopy category:
  \begin{equation} \begin{aligned}\label{eq:Upsilon_and_c}
    \xymatrix@C=18mm{
\bBUP \ar[rr]^-{\Upsilon^n} \ar[d]_c &&\bBUP \ar[d]^{\Theta}_\sim\\
      \Omega^\bullet(\Sigma^\infty_+\bP)\ar[r]_-{ \Omega^\bullet(\Sigma^\infty_+\phi^n)} &
    \Omega^\bullet(\Sigma^\infty_+\bP)\ar[r]_-{\Omega^\bullet(\eta)} &   \Omega^\bullet(\bKR)
                                                                       }
  \end{aligned}   \end{equation}
  Clearly, the morphism $\phi^1$ is the identity of $\bP$,
  $\kappa^1$ is the identity of $\bU$, and $\Upsilon^1$ is the identity of $\bBUP$.
\end{construction}

The next proposition verifies a globally-coherent version
of the design criterion for Adams operations, namely that  on line bundles,
$\psi^n$ is the $n$-th tensor power.
The morphism of $C$-global spaces $h\colon\bP\to\bBUP$ was defined in \eqref{eq:define_h};
it represents the inclusion of line bundles into virtual vector bundles.

\begin{prop}\label{prop:psi_on_P}
  For every $n\geq 1$, the following diagram commutes
  in the unstable $C$-global homotopy category:
  \[ \xymatrix@C=12mm{ \bP \ar[r]^-{\phi^n}\ar[d]_h & \bP\ar[d]^h \\
       \bBUP  \ar[r]_-{\Upsilon^n} & \bBUP  } \]
\end{prop}
\begin{proof}
  Corollary \ref{cor:c_on_P} shows that $c\circ h\colon \bP\to \Omega^\bullet(\Sigma^\infty_+\bP)$
  is the unit of the adjunction $(\Sigma^\infty_+,\Omega^\bullet)$;
  naturality of the adjunction unit yields
  $(\Sigma^\infty_+\phi^n)\circ c\circ h=c\circ h\circ \phi^n$.
  Hence
  \[  \Theta\circ \Upsilon^n\circ h \  =_{\eqref{eq:Upsilon_and_c}}
    \ \Omega^\bullet(\eta)\circ (\Sigma^\infty_+\phi^n)\circ c\circ h
    \ = \ \Omega^\bullet(\eta)\circ c\circ h\circ \phi^n
     \ = \ \Theta\circ h\circ \phi^n \ .\]
   The final equation is Theorem \ref{thm:split}.
   Since $\Theta\colon\bBUP\to \Omega^\bullet(\bKR)$
  is a $C$-global equivalence by Theorem \ref{thm:KU_deloops_BUP} (i), we can cancel it,
  and deduce the desired relation.
\end{proof}

The next theorem justifies the name `global Adams operation'
for the morphism $\Upsilon^n\colon\bBUP\to\bBUP$.

\begin{theorem}\label{thm:Adams}
  For every augmented Lie group $\alpha\colon G\to C$, every finite $G$-CW-complex $A$
  and every $n\geq 1$, the following square commutes:
  \[ \xymatrix@C=15mm{
      [A,\bBUP]^\alpha \ar[r]^-{\Upsilon^n_*} \ar[d]_{\eqref{eq:BUP_to_KR_alpha}}^\iso 
      & [A,\bBUP]^\alpha\ar[d]^{\eqref{eq:BUP_to_KR_alpha}}_\iso \\
     KR_\alpha(A)\ar[r]_-{\psi^n} &  KR_\alpha(A)
        } \]
\end{theorem}
\begin{proof}
  We define a map
  \[ \delta_{\alpha,A}\ : \ [A,\bBUP]^\alpha\ \to \ KR_\alpha(A) \text{\quad by\quad}
 \delta_{\alpha,A}(x)\ =\ \psi^n\td{x}-\td{\Upsilon^n_*(x)}\ , \]
  the difference of the two composites around the diagram in question.
  We need to show that $\delta_{\alpha,A}=0$ for all $(\alpha,A)$.
  
  We claim that these maps satisfy:
  \begin{enumerate}[(a)]
  \item The maps $\delta_{\alpha,A}$ are natural for continuous $G$-maps in $A$,
    and for morphisms of augmented Lie groups.
  \item The map $\delta_{\alpha,A}$ satisfies $\delta_{\alpha,A}(2x-y)=2\cdot\delta_{\alpha,A}(x)-\delta_{\alpha,A}(y)$
    for all $x,y\in [A,\bBUP]^\alpha$.
  \item The map $\delta_{\alpha,A}$ is additive whenever the augmentation $\alpha$
    is trivial.
  \end{enumerate}
  Property (a) is clear because the Adams operations
  and the isomorphisms \eqref{eq:BUP_to_KR_alpha} are natural in $A$ and $\alpha$,
  and so are the maps $\Upsilon^n_*$, since they arise from a morphism
  of Real-global spaces.
  
  (b)
  Applying $[A,-]^\alpha$ to the defining diagram
  for $\Upsilon^n$  yields a commutative diagram
  \[ \xymatrix@C=30mm{
[A,\bBUP]^\alpha \ar[r]^-{\Upsilon^n_*} \ar[d]_{\gamma_*}^-\iso &[A,\bBUP]^\alpha \ar[d]^{\gamma_*}_\iso\\
    [A,\Omega^\sigma\bU]^\alpha\ar[r]_-{(\Omega^\sigma(\kappa^n))_*} & [A,\Omega^\sigma\bU]^\alpha
    } \]
  in which all objects are abelian groups, but the horizontal maps are not a priori additive.
  The vertical maps are isomorphisms of abelian groups because $\gamma$ is a $C$-global
  equivalence of ultra-commutative $C$-monoids.
  As induced by a $\sigma$-loop map,
  the map $(\Omega^\sigma(\kappa^n))_*$  satisfies
  \[ (\Omega^\sigma(\kappa^n))_*(2x -y)\ = \ 2\cdot (\Omega^\sigma(\kappa^n))_*(x)\ -\
    (\Omega^\sigma(\kappa^n))_*(y) \]
  for all $x,y\in  [A,\Omega^\sigma\bU]^\alpha$, see Theorem \ref{thm:formula_c_*} (i).
  So the upper horizontal map $\Upsilon_*^n$ satisfies the analogous relation
  \[ \Upsilon^n_*(2x -y)\ = \ 2\cdot\Upsilon^n_*(x) -\Upsilon^n_*(y) \]
  for all $x,y\in  [A,\bBUP]^\alpha$.
  The Adams operation $\psi^n$ is additive, so the difference
  $\delta_{\alpha,A}$ still has the weak additivity property (b).
  
  (c) The  `non-Real' global morphism underlying $\Upsilon^n$ is a global loop
  map, by construction. So for all compact Lie groups $G$ with trivial augmentation,
  the map $\Upsilon^n_*\colon[A,\bBUP]^G\to[A,\bBUP]^G$ is additive.
  The Adams operation $\psi^n$ is additive, too, hence so is the difference $\delta_{\alpha,A}$.
  \medskip
  
  Now we show in five steps that the maps $\delta_{\alpha,A}$ vanish.
  The first step is $(\alpha,A)=(\tilde T,\ast)$ and the class of the tautological
  Real $\tilde T$-representation.
  By definition, the morphism $h\colon\bP\to\bBUP$ satisfies
  $i_*\{\nu_1\} = h_*(u_{\tilde T})$ in $\pi_0^{\tilde T}(\bBUP)$. So
  \begin{align*}
    \td{\Upsilon^n_*(i_*\{\nu_1\})}
    &= \  \td{\Upsilon^n_*(h_*(u_{\tilde T}))} \ = \ \td{h_*(\phi^n_*(u_{\tilde T}))}
      \ = \ \td{h_*(\tilde\mu_n^*(u_{\tilde T}))}   \\
    &= \ \td{\tilde\mu_n^*(h_*(u_{\tilde T}))} \ = \  \td{\tilde\mu_n^*(i_*\{\nu_1\})}
      \ = \  \tilde\mu_n^*\td{i_*\{\nu_1\}}
      \ = \ \tilde\mu_n^*[\nu_1]
      \ = \  \psi^n[\nu_1]  \ = \  \psi^n\td{i_*\{\nu_1\}}\ .
  \end{align*}
  The second equation is Proposition \ref{prop:psi_on_P}.
  This relation precisely means that $\delta_{\tilde T,\ast}(i_*\{\nu_1\})=0$.

  The second step deals with 
  the tautological Real $\tilde U(k)$-representation $\nu_k$.
  For $1\leq i\leq k$, we let $p_i\colon T^k\to \tilde T$
  denote the projection to the $i$-th factor,
  followed by the inclusion $T\to\tilde T$ into the extended circle group. Then
 \begin{align*}
    \res^{\tilde U(k)}_{T^k}(\delta_{\tilde U(k),\ast}(i_*\{\nu_k\}))\
   &= \ \delta_{T^k,\ast}(\res^{\tilde U(k)}_{T^k}(i_*\{\nu_k\}))\
   = \ \delta_{T^k,\ast}(p_1^*(i_*\{\nu_1\})+\dots+ p_k^*(i_*\{\nu_1\}))\\
   &= \ \sum_{1\leq i\leq k} \delta_{T^k,\ast}(p_i^*(i_*\{\nu_1\}))\
   = \ \sum_{1\leq i\leq k} p_i^*(\res^{\tilde T}_T(\delta_{\tilde T,\ast}(i_*\{\nu_1\})))\ =\ 0 \ .
 \end{align*}
 The first and forth equation are the naturality property (a)
 in the group.
 The third equation uses that the map $\delta_{T^k,\ast}$ is additive by (c),
 because $T^k$ is trivially augmented.
The restriction homomorphism $\res^{\tilde U(k)}_{U(k)}\colon RR(\tilde U(k))\to R(U(k))$
is injective, see for example \cite[page 13]{atiyah-segal:completion}.
And the restriction homomorphism $\res^{U(k)}_{T^k}\colon R(U(k))\to R(T^k)$
is injective, too.
So this proves that $\delta_{\tilde U(k),\ast}(i_*\{\nu_k\})=0$
in $RR(\tilde U(k))$, the Real representation ring of $\tilde U(k)$.

In the third step we fix $k\geq 0$, and we write $j_{\alpha,A}^k$ for the composite 
\[ [A,\bGr_k]^\alpha\ \xra{\ i^k_*\ } \ [A,\bBUP]^\alpha \ \xra{\delta_{\alpha,A}} \
  KR_\alpha(A) \ \xra{\eqref{eq:BUP_to_KR_alpha}^{-1}}\ [A,\bBUP]^\alpha\ . \]
Here $i^k\colon\bGr_k \to\bBUP$ is the restriction of
the morphism $i\colon\bGr\to\bBUP$ to the $k$-th summand.
Varying $\alpha$ and $A$ yields a $C$-global transformation
$j^k$ from $\bGr_k$ to $\bBUP$ in the sense of Definition \ref{def:global_trans}.
The orthogonal $C$-space $\bGr_k$ is a $C$-global classifying space
for the augmented Lie group $\tilde U(k)$, and the class $\{\nu_k\}$
is the tautological class $u_{\tilde U(k),\ast}$ in the sense of 
\eqref{eq:u_beta}.
We showed in the previous step that this $C$-global transformation vanishes
on the tautological class;
so $j_{\alpha,A}^k$ vanishes on all elements of $[A,\bGr_k]^\alpha$, for all augmented Lie groups
and all finite equivariant CW-complexes, by
the uniqueness part of Theorem \ref{thm:global2[B,E]^K}.

In the fourth step we assume that $A$ is $G$-connected, i.e.,
the group $\pi_0(G)$ acts transitively on $\pi_0(A)$.
This ensures that every Real $\alpha$-vector bundle over $A$ has constant rank.
By the isomorphism of Theorem \ref{thm:BUP_represents} (iii),
every element of $[A,\bBUP]^\alpha$ is thus of the form $j_{\alpha,A}^k(y)-j_{\alpha,A}^l(z)$
for some $k,l\geq 0$, some $y\in [A,\bGr_k]^\alpha$ and some
$z\in [A,\bGr_l]^\alpha$.
Then 
\begin{align*}
    \delta_{\alpha,A}(j_{\alpha,A}^k(y)-j_{\alpha,A}^l(z)) \
    &=\ \delta_{\alpha,A}(2\cdot j_{\alpha,A}^k(y)- j_{\alpha,A}^{k+l}(y+z)) \\
_{\text{(b)}}    &=\  2\cdot \delta_{\alpha,A}(j_{\alpha,A}^k(y)) \  -\ \delta_{\alpha,A}(j_{\alpha,A}^{k+l}(y+z)) \ =\ 0\ .
\end{align*}
The third equality is step three.

  In the fifth and final step we let $A$ be any finite $G$-CW-complex.
  Then $A=A_1\amalg\ldots\amalg A_m$ is a disjoint union of $G$-connected
  finite $G$-CW-complexes, indexed by the $\pi_0(G)$-orbits of $\pi_0(A)$.
  We write $\iota_j\colon A_j\to A$ for the inclusion of the $j$-th summand.
  Then by step four,
  \[ \iota_j^*(\delta_{\alpha,A}(x)) \ = \ \delta_{\alpha,A_j}(\iota_j^*(x)) \ = \ 0 \]
  for all $1\leq j\leq m$ and every $x\in [A,\bBUP]^\alpha$.
  Real-equivariant K-theory takes disjoint unions to products, so the map
  \[ (\iota_1^*,\dots,\iota_m^*)\ : \ KR_\alpha(A)\ \to \ KR_\alpha(A_1)\times\dots\times KR_\alpha(A_m) \]
  is an isomorphism of rings. Hence $\delta_{\alpha,A}(x)=0$, which concludes the proof.
\end{proof}

The Adams operations satisfy $\psi^m\circ\psi^n=\psi^{m n}$ for all $m,n\geq 1$.
So by Theorem \ref{thm:Adams}, the morphisms of Real-global spaces
\[ \Upsilon^m\circ\Upsilon^n,\ \Upsilon^{m n}\ : \ \bBUP \ \to \ \bBUP  \]
induce the same map on $[A,\bBUP]^\alpha$ for all augmented Lie groups
and all finite equivariant CW-complexes.
I do not know if in fact $\Upsilon^m\circ\Upsilon^n$ equals $\Upsilon^{m n}$ as endomorphisms
of $\bBUP$ in the unstable $C$-global homotopy category.
Or even better, if the $\sigma$-deloopings \eqref{eq:define_deloop_Adams}
of the global Adams operations satisfy $\kappa^m\circ\kappa^n=\kappa^{m n}\colon \bU\to\bU$.
Also,  $\Upsilon^n$ induces additive maps upon applying $[A,-]^\alpha$
for all augmented Lie groups and all finite equivariant CW-complexes;
a natural question is thus whether $\Upsilon^n$ is a $C$-global H-map.
I expect this to be the case, but our techniques do not suffice to show it.

\begin{appendix}

  \section{Some \texorpdfstring{$C$}{C}-global homotopy theory}\label{app:A}
  
  In this appendix we develop some basics about $C$-global homotopy theory that we need in
  this paper.
  In the bulk of the article, the group $C$ is the Galois group of $\mC$ over $\mR$; 
  the arguments in this appendix work more generally,
  and here we let $C$ be any compact Lie group.
  The main results of this appendix are the classification of $C$-global transformations
  with source a $C$-global classifying space in Theorem \ref{thm:global2[B,E]^K},
  and the criterion of Proposition \ref{prop:closed_criterion}
  to recognize coinduced $C$-spaces.
  Two other references that develop general $C$-global homotopy theory are
  \cite[Appendix A]{schwede:stiefel} and \cite{barrero}.\medskip

An {\em inner product space} is a finite-dimensional real vector space equipped
with a scalar product, i.e., a positive-definite symmetric bilinear form.
We denote by $\bL$ the topological category with objects the inner product spaces and morphisms the
linear isometric embeddings, topologized as Stiefel manifolds.

\begin{defn}\label{def:orthogonal_C-space}
  Let $C$ be a compact Lie group. An {\em orthogonal $C$-space} is a continuous functor
  from the linear isometries category $\bL$ to the category of $C$-spaces.    
\end{defn}

  The notion of {\em $C$-global equivalence} for morphisms of
  orthogonal $C$-spaces is defined in \cite[Definition A.2]{schwede:stiefel}
  and \cite[Definition 3.2]{barrero}.
  It generalizes that of global equivalences of orthogonal spaces
  from \cite[Definition 1.1.2]{schwede:global}, to which it reduces when $C$ is the trivial group.
  The $C$-global equivalences are part of the {\em $C$-global model structure}
  on the category of orthogonal $C$-spaces established in \cite[Theorem A.20]{barrero}.
  When $C$ is a trivial group, this specializes to the global model
  structure on orthogonal spaces from \cite[Theorem 1.2.21]{schwede:global}.

\begin{construction}[Equivariant homotopy sets]
  We introduce the equivariant homotopy sets defined by orthogonal $C$-spaces,
  for a compact Lie group $C$. When $C$ is the trivial group,
  these are discussed in more detail in \cite[Section 1.5]{schwede:global}.

  For every compact Lie group $G$, we choose a complete $G$-universe, i.e.,
  an orthogonal $G$-representation $\Uc_G$ of countably infinite dimension
  such that every finite-dimensional $G$-representation embeds into $\Uc_G$
  by an $\mR$-linear $G$-equivariant isometric embedding.
  We write $s(\Uc_G)$ for the poset, under inclusion, of finite-dimensional
  $G$-subrepresentations of $\Uc_G$.
  We let $E$ be an orthogonal $C$-space, we let $\alpha\colon G\to C$ be a
  continuous homomorphism of compact Lie groups.
  For every $V\in s(\Uc_G)$, the space
  $E(V)$ becomes a $(C\times G)$-space via the $C$-action on $E$,
  and the $G$-action on $V$ through the functoriality of $E$.
  We write $\alpha^\flat(E(V))$ for the $G$-space obtained by
  restriction of scalars along $(\alpha,\Id)\colon G\to C\times G$.
  Said differently, $G$ acts diagonally on $\alpha^\flat(E(V))$,
  via the $G$-action on $E$ through $\alpha$, and on $V$.
  For a $G$-space $A$, we set
  \begin{equation}\label{eq:[-,-]}
 [A,E]^\alpha \ = \ \colim_{V\in s(\Uc_G)}\, [A, \alpha^\flat(E(V))]^G\ .
        \end{equation}
  Here $[-,-]^G$ is the set of $G$-equivariant
  homotopy classes of equivariant maps.
  The colimit is taken along the maps $E(V)\to E(W)$ induced by the
  inclusions for $V\subseteq W$ in $s(\Uc_G)$.

  We shall also use the notation
  \[ \pi_0^\alpha(E)\ = \ \colim_{V\in s(\Uc_G)}\, \pi_0(\alpha^\flat(E(V))^G)\ . \]
  Evaluation at the unique point is a natural bijection
  $[\ast,\alpha^\flat(E(V))]^G\iso \pi_0(\alpha^\flat(E(V))^G)$; in the colimit over
  the poset $s(\Uc_G)$, this becomes a natural bijection
  \[ [\ast,E]^\alpha \ \iso \ \pi_0^\alpha(E)\ . \]
  In the following we shall routinely identify 
  $[\ast,E]^\alpha$ and $\pi_0^\alpha(E)$ in this way without further notice.
\end{construction}

The sets $[A,E]^\alpha$ are contravariantly functorial
for continuous $G$-maps in $A$ by precomposition, and covariantly functorial
for morphisms of orthogonal $C$-spaces in $E$. Another functoriality
in $\alpha\colon G\to C$ will be discussed in Construction \ref{con:funct_in_G} below.

When $C$ is the trivial group, the following proposition specializes to 
\cite[Proposition 1.5.3]{schwede:global}.
The proof in the present $C$-global context is almost literally the same,
so we omit it.

\begin{prop}
  Let $\alpha\colon G\to C$ be a continuous homomorphism of compact Lie groups,
  let $E$ be an orthogonal $C$-space, and let $A$ be a $G$-space.
  \begin{enumerate}[\em (i)]
  \item Suppose that the $G$-space $A$ is compact and the orthogonal space underlying $E$
    is {\em closed}, i.e., all structure maps $E(\psi)\colon E(V)\to E(W)$ are closed embeddings.
    Then the canonical map
    \[ [A,E]^\alpha \ \to \ [A, \alpha^\flat(E(\Uc_G))]^G \]
    is bijective, where $E(\Uc_G)=\colim_{V\in s(\Uc_G)} E(V)$.
  \item If $A$ is a finite $G$-CW-complex, then the assignment $E\mapsto [A,E]^\alpha$
  sends $C$-global equivalences of orthogonal $C$-spaces to bijections.
  \item If $F$ is another orthogonal $C$-space, then the map
    \[ ([A,p_E],[A,p_F])\ : \ [A,E\times F]^\alpha \ \to \ [A,E]^\alpha \times [A,F]^\alpha  \]
    is bijective, where $p_E$ and $p_F$ are the projections.
\end{enumerate}
\end{prop}

\begin{construction}[Functoriality in the group]\label{con:funct_in_G}
  Let $\alpha\colon G\to C$ be a continuous homomorphism of compact Lie groups,
  and let $E$ be an orthogonal $C$-space.
  A continuous homomorphism $\beta\colon K\to G$ of compact
  Lie groups induces a restriction map
  \begin{equation}\label{eq:define_restriction}
    \beta^* \ : \ [A,E]^\alpha\ \to \ [\beta^*(A),E]^{\alpha\beta}\ ,
  \end{equation}
  as follows.
  Given $V\in s(\Uc_G)$ and a continuous $G$-equivariant map
  $f\colon A\to \alpha^\flat(E(V))$, restriction of actions along $\beta$
  yields a $K$-equivariant map
  \[ \beta^*(f)\ : \ \beta^*(A)\ \to \ \beta^*(\alpha^\flat(E(V))) \ = \
    (\alpha\beta)^\flat(E(\beta^*(V)))\ .
  \]
  We choose a $K$-equivariant linear isometric embedding
  $j\colon \beta^*(V)\to \Uc_K$ into the chosen $K$-universe.
  Then $W=j(\beta^*(V))$ is a finite-dimensional $K$-subrepresentation,
  and thus an element of the poset $s(\Uc_K)$; and the embedding $j$ restricts
  to an isomorphism of $K$-representations $j\colon \beta^*(V)\iso W$. 
  The restriction map \eqref{eq:define_restriction} sends the class represented by $f$ in 
  $[A,E]^\alpha$ to the class in $[\beta^*(A),E]^{\alpha\beta}$ represented by the composite
   $K$-equivariant map
   \[ \beta^*(A)\ \xra{\beta^*(f)}\ (\alpha\beta)^\flat(E(\beta^*(V))) \ \xra[\iso]{\ E(j)\ } \
     (\alpha\beta)^\flat(E(W)) \ .\]
  The analogous argument as in the special case $C=\{1\}$ and $A=\ast$
  in \cite[Proposition 1.5.8]{schwede:global}
  shows that the resulting class in $[\beta^*(A),E]^{\alpha\beta}$
  does not depend on the choice of embedding $j\colon\beta^*(V)\to\Uc_K$,
  and so the construction is well-defined.
  Given well-definedness, the construction is clearly natural in $A$ and $E$,
  and contravariantly functorial in groups augmented to $C$:
  given another continuous homomorphism $\gamma\colon L\to K$, we have
  \[ \gamma^*\circ\beta^*\ = \ (\beta\circ\gamma)^* \ : \
    [A,E]^\alpha\ \to \ [(\beta\gamma)^*(A),E]^{\alpha\beta\gamma}\ . \]
\end{construction}

\begin{construction}[Induction isomorphisms]
  We let $\alpha\colon G\to C$ be a continuous homomorphism of compact Lie groups,
  we let $E$ be an orthogonal $C$-space, and we let $H$ be a closed subgroup of $G$.
  For an $H$-space $B$, we write
  \[ [1,-] \ : \ B\ \to \ G\times_H B \ , \quad y \ \longmapsto \ [1,y] \]
  for the  $H$-equivariant unit of the adjunction $(G\times_H-,\res^G_H)$.
  The adjunction bijections
  \[ [G\times_H B, \alpha^\flat(E(V))]^G\ \iso \ [B,(\alpha|_H)^\flat(E(V))]^H  \ , \quad
  [f]\ \longmapsto \ [\res^G_H(f)\circ [1,-]]\]
  for $V\in s(\Uc_G)$, and the fact that the underlying $H$-universe of $\Uc_G$ is a
  complete $H$-universe provide an {\em induction isomorphism}:
  the composite
  \begin{equation}\label{eq:induction_iso}
    [G\times_H B,E]^{\alpha}\  \xra{\res^G_H}\  [G\times_H B,E]^{\alpha|_H} \
    \xra{[1,-]^*}\  [B,E]^{\alpha|_H}
  \end{equation}
  is bijective, where $\res^G_H$ is short hand for the restriction homomorphism
  \eqref{eq:define_restriction} associated with the inclusion $H\to G$.
\end{construction}

If $H$ is a closed subgroup of smaller dimension in a compact Lie group $G$,
then the underlying $H$-space of a $G$-CW-complex 
need not admit an $H$-CW-structure;
an example is given in \cite[Section 2]{illman:restricting equivariance}.
Nevertheless, the underlying $H$-space of a finite $G$-CW-complex 
is always $H$-homotopy equivalent to a finite $H$-CW-complex,
see \cite[Corollary B]{illman:restricting equivariance}.
Consequently, for every continuous homomorphism $\beta\colon K\to G$ of
compact Lie groups, the restriction functor $\beta^*$ takes
$G$-spaces of the $G$-homotopy type of a finite $G$-CW-complex to
$K$-spaces of the $K$-homotopy type of a finite $K$-CW-complex.

\begin{defn}[$C$-global transformations]\label{def:global_trans}
  Let $C$ be a compact Lie group, and let $F$ and $E$ be orthogonal $C$-spaces.
  A {\em $C$-global transformation} $\tau$ from $F$ to $E$ consists of maps 
  \[ \tau_{\alpha,A}\ : \ [A,F]^\alpha\ \to \ [A,E]^\alpha \]
  for all continuous homomorphism $\alpha\colon G\to C$ of compact Lie groups
  and all $G$-spaces $A$ of the $G$-homotopy type of a finite $G$-CW-complex that are
  natural in the following sense:
  \begin{itemize}
  \item For every continuous $G$-map $f\colon B\to A$, we have
    \[ \tau_{\alpha,B}\circ [f,F]^\alpha \ = \ [f,E]^\alpha\circ \tau_{\alpha,A}\ : \
      [A,F]^\alpha\ \to \ [B,E]^\alpha \ .\]
  \item For every continuous homomorphism $\beta\colon K\to G$ of compact Lie groups, we have
    \[ \tau_{\alpha\beta,\beta^*(A)}\circ \beta^* \ = \ \beta^*\circ \tau_{\alpha,A}\ : \
      [A,F]^\alpha\ \to \ [\beta^*(A),E]^{\alpha\beta} \ .\]
  \end{itemize}
\end{defn}

\begin{construction}[Global classifying spaces]\label{con:B_gl beta}
  We let $\beta\colon K\to C$ be a continuous homomorphism of compact Lie groups.
  We choose a faithful $K$-representation $V$ and define the
  {\em $C$-global classifying space}  as the orthogonal $C$-space 
  \[ B_{\gl}\beta \ = \ C\times_\beta\bL(V,-)\ . \]
  In more detail, $(B_{\gl}\beta)(W)$ is the quotient of the $C$-space
  $C\times\bL(V,W)$ by the equivalence relation
  \[ (c,\varphi)\ \sim \  (c\cdot \beta(k),\varphi\circ l_k) \]
  for all $(c,k)\in C\times K$ and all linear isometric embeddings
  $\varphi\colon V\to W$, where $l_k\colon V\to V$ is the action of $k\in K$.
  The notation $B_{\gl}\beta$ is slightly abusive in that we do not record the choice of
  faithful $K$-representation. This abuse is justified by
  \cite[Proposition A.5]{schwede:stiefel}, showing that
  $B_{\gl}\beta$ is independent of the choice of faithful representation
  up to a preferred zigzag of $C$-global equivalences.

  The equivalence class
  \[[1,\Id_V]\ \in \ C\times_\beta\bL(V,V)\ = \ (B_{\gl}\beta)(V) \]
  is a $K$-fixed point of $\beta^\flat((B_{\gl}\beta)(V))$, so it represents a class
  \begin{equation} \label{eq:u_beta}
 u_{\beta}\ \in \ \pi_0^\beta(B_{\gl}\beta) \ ,    
  \end{equation}
  the {\em tautological class}.
\end{construction}

\begin{theorem}\label{thm:global2[B,E]^K}
  Let $\beta\colon K\to C$ be a continuous homomorphism of compact Lie groups,
  and let $E$ be an orthogonal $C$-space.
  For every class $y$ in $\pi_0^\beta(E)$, there is a unique $C$-global
  transformation $\tau$ from $B_{\gl}\beta$ to $E$ such that the map
  \[ \tau_{\beta,\ast}\ : \ \pi_0^\beta(B_{\gl}\beta)\ \to \ \pi_0^\beta(E) \]
  sends the tautological class $u_{\beta}$ to $y$.
\end{theorem}
\begin{proof}
  To construct $\tau$  we represent $y$ by a $K$-fixed point
  $\tilde y\in (\beta^\flat(E(W))^K$,
    for some $K$-representation $W$.
    By enlarging $W$, if necessary, we can assume that there is a $K$-equivariant
    linear isometric embedding $\varphi\colon V\to W$ from the faithful $K$-representation
    that is implicit in the construction of $B_{\gl}\beta$.
    By \cite[Proposition A.5]{schwede:stiefel},
    the morphism of orthogonal $C$-spaces
    \[    \varphi^* = C\times_\beta\bL(\varphi,-)\ :\ Y = C\times_\beta\bL(W,-)\ \to\
      C\times_\beta\bL(V,-)\ =\  B_{\gl}\beta        \]
    is a $C$-global equivalence.
    Moreover, if we let $\gh{1,\Id_W}\in\pi_0^\beta(Y)$ denote the class
    represented by the $K$-fixed point
    $[1,\Id_W]$ of $\beta^\flat(C\times_\beta\bL(W,W))=\beta^\flat(Y(W))$,
    then $\pi_0^\beta(\varphi^*)\gh{1,\Id_W}=u_{\beta}$.
    
    The fixed point $\tilde y$ is represented by a unique morphism
    of orthogonal $C$-spaces $f\colon Y\to E$
    such that the map
    \[ f(W)\ : \  C\times_\beta\bL(W,W) = Y(W)\ \to\ E(W)  \]
    takes $[1,\Id_W]$ to $\tilde y$.
    The global equivalence $\varphi^*$ and the morphism $f$ induce a $C$-global
    transformation $\tau$ from $B_{\gl}\beta$ to $E$ with constituents
    \[ \tau_{\alpha,A}\ : \ [A,B_{\gl}\beta]^\alpha\ \xra[\iso]{([A,\varphi^*]^\alpha)^{-1}} \
      [A,Y]^\alpha\ \xra{[A,f]^\alpha} \ [A,E]^\alpha \ .\]
    This $C$-global transformation satisfies
    \[
      \tau_{\beta,\ast}(u_\beta)\ = \
      \pi_0^\beta(f)(\pi_0^\beta(\varphi^*)^{-1}(u_{\beta})) \ = \
      \pi_0^\beta(f)\gh{1,\Id_W} \ = \ [\tilde y]\ =  \ y\ . \]

  The proof of uniqueness is more involved.
  We let $V$ be the faithful $K$-representation implicit in the
  definition of $B_{\gl}\beta$.
  We let $W$ be an inner product space.
  The group $C\times O(W)$ acts on
  $(B_{\gl}\beta)(W)= C\times_\beta \bL(V,W)$  by
  \[ (c,A)\cdot [d,\varphi]\ = \  [c d, A\circ \varphi]\ .\]
  This action is transitive;
  and for a linear isometric embedding $\varphi\colon V\to W$,
  the stabilizer of the point $[1,\varphi]\in C\times_\beta\bL(V,W)$ is the subgroup
  \[ \Sc[\varphi]\ = \ \{(c,A)\in C\times O(W)\colon
    \text{there is $k\in K$ such that $c=\beta(k)$ and $A\varphi=\varphi l_k$}\}\ . \]
  Since the $K$-action on $V$ is faithful, for every $(c,A)\in \Sc[\varphi]$,
  the $k\in K$ such that $c=\beta(k)$ and $A\varphi=\varphi l_k$ is unique,
  and we can define a continuous
  homomorphism $\gamma\colon \Sc[\varphi]\to K$ by letting $\gamma(c,A)$ be this
  unique element $k$ of $K$.
  We let $\Pi\colon C\times O(W)\to C$ denote the projection to the first factor.
  Then the square of continuous group homomorphisms
  \[ \xymatrix{
      \Sc[\varphi]\ar[r]^-{\text{incl}}\ar[d]_\gamma & C\times O(W)\ar[d]^{\Pi}\\ 
      K\ar[r]_-\beta & C} \]
  commutes by design.
  Since the map
  \[ ( C\times O(W))/\Sc[\varphi]\ \to \ C\times_\beta \bL(V,W)  = (B_{\gl}\beta)(W)\ , \quad
    [c, A] \ \longmapsto\  [c,A \varphi] \]
  is a $(C\times O(W))$-equivariant homeomorphism,
  the induction isomorphism \eqref{eq:induction_iso} shows that the composite 
  \[  [C\times_\beta \bL(V,W),E]^{\Pi}\  \xra{\res^{C\times O(W)}_{\Sc[\varphi]}}\
    [C\times_\beta\bL(V,W),E]^{\beta\gamma} \ \xra{[1,\varphi]^*}\
   \pi_0^{\beta\gamma}(E) \]
  is bijective for every orthogonal $C$-space $E$.
  The identity of $C\times_\beta\bL(V,W)$ represents a tautological class
  \[ [\Id_{C\times_\beta\bL(V,W)}]\ \in \ [C\times_\beta \bL(V,W),B_{\gl}\beta]^{\Pi} \]
  that satisfies
  \[ [1,\varphi]^*(\res^{C\times O(W)}_{\Sc[\varphi]}[\Id_{C\times_\beta\bL(V,W)}])\ = \
    \gh{1,\varphi}\ = \ \gamma^*(u_\beta) \]
  in $\pi_0^{\beta\gamma}(B_{\gl}\beta)$.

  Now we let $\tau$ be any $C$-global transformation from $B_{\gl}\beta$ to $E$.
  Its naturality properties yield the relations
  \begin{align*}
    [1,\varphi]^*(\res^{C\times O(W)}_{\Sc[\varphi]}
    &(\tau_{\Pi,C\times_\beta \bL(V,W)}[\Id_{C\times_\beta\bL(V,W)}]))\\
    &= \  \tau_{\beta\gamma,\ast}([1,\varphi]^*(\res^{C\times O(W)}_{\Sc[\varphi]}[\Id_{C\times_\beta\bL(V,W)}]))\\
    &= \  \tau_{\beta\gamma,\ast}(\gamma^*(u_\beta))\ = \  \gamma^*(\tau_{\beta,\ast}(u_\beta))\ .
  \end{align*}
  Since the composite $[1,\varphi]^*\circ\res^{C\times O(W)}_{\Sc[\varphi]}$
  is bijective by the induction isomorphism \eqref{eq:induction_iso},
  this relation shows that---and how---the class 
  $\tau_{\Pi,C\times_\beta \bL(V,W)}[\Id_{C\times_\beta\bL(V,W)}]$
  is determined by the class $\tau_{\beta,\ast}(u_\beta)$.

  Now we consider another compact Lie group $G$, a continuous homomorphism $\alpha\colon G\to C$,
  a finite $G$-CW-complex $A$, and a class $x\in [A,B_{\gl}\beta]^\alpha$.
  We represent $x$ by a continuous $G$-map $f\colon A\to \alpha^\flat((B_{\gl}\beta)(W))$
  for some $G$-representation $W$.
  If $A=\emptyset$, then $[A,E]^\alpha$ has only one element, and there is nothing to show.
  If $A$ is nonempty, then also $(B_{\gl}\beta)(W)$ is nonempty, and there exists
  a linear isometric embedding $\varphi\colon V\to W$.
  
  We let $\rho\colon G\to O(W)$ classify the $G$-action on $W$.
  Then $(\alpha,\rho)\colon G\to C\times O(W)$ is a homomorphism over $C$,
  in the sense that $\Pi\circ (\alpha,\rho)=\alpha$.
  The $G$-action on $\alpha^\flat(C\times_\beta\bL(V,W))=\alpha^\flat((B_{\gl}\beta)(W))$
  is obtained from the $(C\times O(W))$-action
  by restriction along $(\alpha,\rho)\colon G\to (C\times O(W))$.
  So 
  \[ f\ : \ A\ \to \ (\alpha,\rho)^*(C\times_\beta \bL(V,W)) \]  
  is $G$-equivariant.
  Moreover, the tautological relation
  \[ x \ = \ [f]\ = \ f^*((\alpha,\rho)^*[\Id_{C\times_\beta\bL(V,W)}]) \]
  holds, i.e., $x$ is the image of the tautological class under the composite
  \[ [C\times_\beta\bL(V,W),B_{\gl}\beta]^{\Pi}\ \xra{(\alpha,\rho)^*}\ 
     [(\alpha,\rho)^*(C\times_\beta\bL(V,W)),B_{\gl}\beta]^\alpha\ \xra{\ f^* \ }\
     [A,B_{\gl}\beta]^\alpha\ .\]
   The naturality properties of the $C$-global transformation $\tau$ yield
   \begin{align*}
     \tau_{\alpha,A}(x)\
  &= \ \tau_{\alpha,A}(f^*((\alpha,\rho)^*[\Id_{C\times_\beta\bL(V,W)}])) \\
  &= \ f^*((\alpha,\rho)^*(\tau_{\Pi,C\times_\beta\bL(V,W)}[\Id_{C\times_\beta\bL(V,W)}])) \ .
  \end{align*}
  We argued above that the class 
  $\tau_{\Pi,C\times_\beta \bL(V,W)}[\Id_{C\times_\beta\bL(V,W)}]$
  is determined by the class $\tau_{\beta,\ast}(u_\beta)$.
  So also $\tau_{\alpha,A}(x)$ is determined by the class $\tau_{\beta,\ast}(u_\beta)$.
  This completes the proof of uniqueness.
\end{proof}

In the remainder of this appendix,
we investigate the class of {\em coinduced} $C$-global spaces.
For these objects, the $C$-global information is determined,
in the precise sense of Proposition \ref{prop:coinduced_coinduces},
by the global information of the underlying orthogonal space, i.e., after forgetting the $C$-action.
This class of objects is relevant for our purposes because
for $C=\Gal(\mC/\mR)$, the $C$-global spaces $\bBUP$, $\bU$ and $\Omega^\bullet(\bKR)$
are coinduced,
see Theorems \ref{thm:coinduced} and \ref{thm:Omega^bullet coinduced}.

\begin{prop}\label{prop:times_free}
  Let $C$ be a compact Lie group.
  Let $\phi\colon A\to B$ be a morphism of orthogonal $C$-spaces 
  that is a global equivalence of underlying orthogonal spaces after
  forgetting the $C$-action. Then for every free $C$-space $E$,
  the morphism $\phi\times E \colon A\times E\to B\times E$ is a $C$-global equivalence.
\end{prop}
\begin{proof}
  Product with any topological space preserves global equivalences
  of orthogonal spaces by \cite[Proposition 1.1.9 (vii)]{schwede:global}.
  Since the underlying morphism of $\phi$ is a global equivalence,
  so is the underlying morphism of $\phi\times E$.
  To show that $\phi\times E$ is a $C$-global equivalence, it thus remains
  to solve the lifting property from the defining property \cite[Definition A.2]{schwede:stiefel}
  of $C$-global equivalences for all non-trivial continuous homomorphisms
  $\alpha\colon G\to C$.
  But if $\alpha$ is non-trivial, then for every $G$-representation $V$, the space
  \[ \alpha^\flat((B\times E)(V))^G\ = \ \alpha^\flat(B(V))^G\times E^{\alpha(G)}  \]
  is empty, because because $C$ acts freely on $E$. 
  Hence there are no lifting problems to solve for $\alpha$, which completes the
  proof that $\phi\times E$ is a $C$-global equivalence.
\end{proof}

In the following we shall write $\gh{-,-}^C$ for the set of morphisms
in the $C$-global homotopy category, i.e., the localization of
the category of orthogonal $C$-spaces at the class of $C$-global equivalences.
We let $E C$ be a universal free $C$-space, i.e., a free $C$-CW-complex
whose underlying space is contractible.
And we shall write $p\colon EC\to *$ for the unique map.

\begin{defn}\label{def:coinduced}
  Let $C$ be a compact Lie group.
  An orthogonal $C$-space $X$ is {\em coinduced} if for every orthogonal $C$-space $A$ the map
  \[ \gh{A\times p,X}^C\ : \  \gh{A,X}^C \ \to\ \gh{A\times E C, X}^C \]
   induced by $A\times p\colon A\times E C \to A$ is bijective.
\end{defn}

\begin{prop}\label{prop:coinduced_coinduces}
  Let $X$ and $Y$ be coinduced orthogonal $C$-spaces.
  \begin{enumerate}[\em (i)]
  \item 
    For every  morphism of orthogonal $C$-spaces $\phi\colon A\to B$ 
    that is a global equivalence of underlying orthogonal spaces after
    forgetting the $C$-action, the map
    \[ \gh{\phi,X}^C \ : \ \gh{B,X}^C \ \to \ \gh{A, X}^C \]
    is bijective.
  \item
      Let $f\colon X\to Y$ be a morphism that is a global equivalence
  of underlying orthogonal spaces after forgetting the $C$-action.
  Then $f$ is a $C$-global equivalence.
  \end{enumerate}
\end{prop}
\begin{proof}
  (i) Because the underlying morphism of $\phi$ is a global equivalence
  and the $C$-action on $EC$ is free, the morphism
  $\phi\times E C\colon A\times E C\to B\times E C$ is a $C$-global equivalence by
  Proposition \ref{prop:times_free}.
  So the lower horizontal map in the following commutative diagram is bijective:
  \[ \xymatrix@C=20mm{ \gh{B,X}^C \ar[r]^-{\gh{\phi,X}^C}\ar[d]^\iso_{\gh{B\times p,X}^C} & \gh{A, X}^C\ar[d]^{\gh{A\times p,X}^C}_\iso\\
      \gh{B\times EC,X}^C \ar[r]^\iso_-{\gh{\phi\times EC,X}^C} & \gh{A\times EC, X}^C} \]
  The vertical maps are bijective because $X$ is coinduced, so
  also $\gh{\phi,X}^C$ is bijective, as claimed.

  (ii)
   Since $X$ is coinduced, the map 
   $\gh{f,X}^C \colon\gh{Y,X}^C \to  \gh{X,X}^C$
    is bijective by part (i).
    So there is a morphism $g\colon Y\to X$
    in the $C$-global homotopy category such that $g f=\Id_X$.
    Then
    \[ \gh{f,Y}(f g)\ = \ f g f\ = \ f \ = \  \gh{f,Y}(\Id_Y)\ .\]
    Since $Y$ is coinduced, the map 
   $\gh{f,Y}^C \colon \gh{Y,Y}^C  \to  \gh{X,Y}^C$
    is bijective, again by part (i), so $f g=\Id_Y$. Hence $f$ is an isomorphism in the 
    $C$-global homotopy category, and thus a $C$-global equivalence.
\end{proof}

An orthogonal $C$-space $X$ is fibrant in the $C$-global model structure
of \cite[Theorem A.20]{barrero}
if and only if the following condition holds, see \cite[Definition A.13]{barrero}:
for every continuous homomorphism of compact Lie groups $\alpha\colon G\to C$
and every linear isometric embedding of $G$-representations
$\varphi\colon V\to W$ such that $G$ acts faithfully on $V$,
the map
\[ \alpha^\flat (X(\varphi))^G\ :\  \alpha^\flat(X(V))^G\ \to \ \alpha^\flat(X(W))^G\]
is a weak equivalence.
We will call such orthogonal $C$-spaces  {\em $C$-fibrant}.

\begin{prop}\label{prop:fibrant_criterion}
  A $C$-fibrant orthogonal $C$-space $X$ is coinduced if and only if
  the morphism of orthogonal $C$-spaces
    \[ \map(p,X)\ : \ X\ \to \ \map(E C,X) \]
    is a $C$-global equivalence.
\end{prop}
\begin{proof}
  The adjoint functors $(-\times EC,\map(E C,-))$ are a Quillen functor
  pair for the $C$-global model structure, so they induce an adjoint pair
  of derived functors at the level of the $C$-global homotopy category.
  The left adjoint $-\times EC$ is fully homotopical, i.e., it preserves
  arbitrary $C$-global equivalences, for example by Proposition \ref{prop:times_free};
  so it descends to the $C$-global homotopy category, and the descended functor is the derived
  left adjoint. Since $X$ is $C$-fibrant, the orthogonal $C$-space $\map(E C,X)$
  models the total right derived functor of $\map(E C,-)$ on $X$.
  The following diagram commutes:
  \[ \xymatrix@C=20mm{ \gh{A,X}^C \ar[d]_{{\gh{A,\map(p,X)}^C}}\ar@/^1pc/[dr]^-{\gh{A\times p, X}^C}&\\ 
      \gh{A,\map(E C,X)}^C \ar[r]_-{\text{adjunction}}^-\iso& \ \gh{A\times E C, X}^C  } \]
  So $X$ is coinduced if and only if for every orthogonal $C$-space $A$,
  the left vertical map $\gh{A,\map(p,X)}^C$ is bijective. This happens if and
  only if $\map(p,X)$ is an isomorphism in the $C$-global homotopy category,
  which is equivalent to $\map(p,X)$ being a $C$-global equivalence.
\end{proof}

If $X$ is an orthogonal space and $G$ a compact Lie group,
the {\em underlying $G$-space} of $X$ is the $G$-space
\[ X(\Uc_G)\ =  \ \colim_{V\in s(\Uc_G)} X(V)\ ; \]
the colimit is formed over the poset $s(\Uc_G)$, under inclusion,
of finite-dimensional $G$-subrepresentations of the chosen complete $G$-universe $\Uc_G$.
If $X$ is an orthogonal $C$-space, then the $C$-action on $X$ induces a continuous
$C$-action on $X(\Uc_G)$ that commutes with the $G$-action, so $X(\Uc_G)$
becomes a $(C\times G)$-space.
We write $\alpha^\flat(X(\Uc_G))$ for the $G$-space obtained by
restriction of scalars along $(\alpha,\Id)\colon G\to C\times G$.

By \cite[Proposition 3.5]{barrero},
a morphism $f\colon X\to Y$ between closed orthogonal $C$-spaces
is a $C$-global equivalence if and only if for every
continuous homomorphism $\alpha\colon G\to C$ of compact Lie groups,
the map
\[ \alpha^\flat(f(\Uc_G))^G\ :\
 \alpha^\flat(X(\Uc_G))^G\ \to \ \alpha^\flat(Y(\Uc_G))^G \]
is a weak equivalence.

\begin{prop}\label{prop:closed_criterion}
  Let $X$ be an orthogonal $C$-space whose underlying orthogonal space is closed.
  Then the following conditions are equivalent.
  \begin{enumerate}[\em (a)]
  \item The orthogonal $C$-space $X$ is coinduced.
  \item For every continuous homomorphism $\alpha\colon G\to C$ of compact Lie groups,
    the map
    \[ \alpha^\flat(\map(p,X(\Uc_G))^G\ : \ \alpha^\flat(X(\Uc_G))^G\ \to \
      \alpha^\flat(\map(E C, X(\Uc_G)))^G =    \map^G(\alpha^*(E C),  \alpha^\flat(X(\Uc_G))) \]
    is a weak equivalence.
  \end{enumerate}
\end{prop}
\begin{proof}
  We start with a preliminary reduction step.
  By choosing an acyclic cofibration $X\to X'$ with fibrant target in the
  $C$-global model structure, we can assume without loss of generality in both parts
  that $X$ is not only closed, but also $C$-fibrant.
  We let $\alpha\colon G\to C$ be a continuous homomorphism,
  and we let $V$ be a faithful $G$-subrepresentation of the universe $\Uc_G$.
  Because $X$ is $C$-fibrant and closed,
  the left vertical map in the commutative diagram
  \begin{equation}    \begin{aligned}\label{eq:colim_diag}
  \xymatrix@C=25mm{
    \alpha^\flat(X(V))^G\ar[r]^-{\alpha^\flat(\map(p,X(V)))^G}_-\sim\ar[d]_\sim
    & \alpha^\flat(\map(E C, X(V)))^G \ar[d]^\sim\\
    \alpha^\flat(X(\Uc_G))^G\ar[r]_-{\alpha^\flat(\map(p,X(\Uc_G)))^G}
    &\alpha^\flat(\map(EC,X(\Uc_G)))^G
    }    
    \end{aligned}  \end{equation}
   is a weak equivalence.
   Because the functor $\map(EC,-)$ preserves the class of $(C\times G)$-maps
   that are graph subgroup equivalences, the right vertical map is a weak equivalence, too.

  (i)$\Longrightarrow$(ii) 
  Because $X$ is coinduced and $C$-fibrant,
  Proposition \ref{prop:fibrant_criterion} shows that
  the morphism $\map(p,X)\colon\allowbreak X\to\map(E C,X)$
  is a $C$-global equivalence.
  Because $X$ is $C$-fibrant, so is $\map(E C,X)$.
  As a $C$-global equivalence between $C$-fibrant objects, the morphism
  $\map(p,X)$ is a $C$-level equivalence, see \cite[Lemma A.19]{barrero}.
  So the upper horizontal map in the diagram \eqref{eq:colim_diag} is 
  a weak equivalence.
  Hence the lower horizontal map in \eqref{eq:colim_diag} is a
  weak equivalence, too, proving condition (ii).
  
  (ii)$\Longrightarrow$(i)
  We turn the previous argument around.
  Because the lower horizontal map in the diagram \eqref{eq:colim_diag} is a
  weak equivalence, so is the upper horizontal map,
  for every faithful $G$-representation $V$.
  So the morphism $\map(p,X)\colon X\to\map(E C,X)$
  is a $C$-level equivalence, and hence a $C$-global equivalence.
  Because $X$ is $C$-fibrant,
  Proposition \ref{prop:fibrant_criterion} shows that $X$ is coinduced.
\end{proof}

\begin{prop}\label{prop:loop_rind}
  Let $X$ be a pointed orthogonal $C$-space whose underlying orthogonal $C$-space
  is coinduced. Then for every finite based $C$-CW-complex $A$,
  also $\map_*(A,X)$ has a coinduced underlying orthogonal $C$-space.
\end{prop}
\begin{proof}
  Since the functor $\map_*(A,-)$ preserves $C$-global equivalences between
  based orthogonal $C$-spaces, we can assume without loss of generality
  that $X$ is $C$-fibrant.
  Then the morphism $p^*\colon X\to\map(E C, X)$
  is a $C$-global equivalence by Proposition \ref{prop:fibrant_criterion}.
  Since $\map_*(A,-)$ preserves $C$-global equivalences, also the morphism
  \[ \map_*(A,p^*) \ : \ \map_*(A,X) \  \to\ \map_*(A,\map(E C, X))   \]
  is a $C$-global equivalence.
  The morphism $p^*\colon\map_*(A,X)\to\map(EC,\map_*(A, X))$
  is isomorphic to $\map_*(A,p^*)$, and hence a $C$-global equivalence.
  Since $X$ is $C$-fibrant, so is $\map_*(A,X)$,
  so another application of Proposition \ref{prop:fibrant_criterion}
  shows that $\map_*(A,X)$ is coinduced.
\end{proof}

\section{Real-global K-theory}\label{app:B}

In this appendix we extend various features
of global K-theory to the Real-global context.
In particular, we provide Real-global generalizations of many
results in Sections 2.5, 6.3 and 6.4 of \cite{schwede:global}.
In Theorem \ref{thm:BUP_represents} we show that
the Real-global space $\bBUP$ represents Real-global K-theory.
In Theorem \ref{thm:BUP2OmegaU},
we establish Real-global Bott periodicity,
an equivalence of Real-global ultra-commutative monoids between $\bBUP$ and $\Omega^\sigma\bU$.
We show in Theorem \ref{thm:KU_deloops_BUP} that
the {\em Real-global K-theory spectrum} $\bKR$ deserves its name:
  for every augmented Lie group $\alpha\colon G\to C=\Gal(\mC/\mR)$,
  the genuine $G$-spectrum  $\alpha^*(\bKR)$ represents
  $\alpha$-equivariant Real K-theory $KR_\alpha$.
  Theorem \ref{thm:U2shKR} shows that (and how) the Real-global space
  $\bU$ is the Real-global infinite loop space of $\bKR\sm S^\sigma$.
  Consequently, the Real-global space $\bBUP\sim \Omega^\sigma\bU$
  `is' the Real-global infinite loop space underlying $\bKR$,
  see Theorem \ref{thm:KU_deloops_BUP} (i).

\subsection{Unstable Real-global K-theory}
In this subsection we extend various unstable features of global K-theory
to the Real-global context. 
The main results are that the orthogonal $C$-space $\bBUP$
represents Real-global K-theory, see Theorem \ref{thm:BUP_represents},
and Real-global Bott periodicity, an equivalence
of Real-global ultra-commutative monoids between $\bBUP$ and $\Omega^\sigma\bU$,
see Theorem \ref{thm:BUP2OmegaU}.
Along the way, we show that the orthogonal $C$-spaces 
$\bBUP$ and $\bU$ are coinduced, see Theorem \ref{thm:coinduced}.

\begin{construction}[$\bGr_k$]\label{con:Gr_k}
  We recall the Grassmannian model $\bGr_k$
  for the Real-global classifying space of
  the extended unitary group $\tilde U(k)=U(k)\rtimes C$,
  augmented by the projection $U(k)\rtimes C\to C$.
  We deviate slightly from the notation of \cite[Section 2.3]{schwede:global},
  where $\bGr$ without a superscript is used for the real version
  of the additive Grassmannian, and where the complex version is denoted $\bGr^\mC$;
  also, the homogeneous summand $\bGr_k$ is written $\bGr^{\mC,[k]}$.
  In the present paper, we shall exclusively work with complex Grassmannians,
  so they will be referred to by the simpler name without superscript.
  The value of $\bGr_k$ at a euclidean inner product space $V$ is 
  \[ \bGr_k(V) \ = \ Gr_k^\mC(V_\mC)\ , \]
  the Grassmannian of complex $k$-planes in the complexification $V_\mC=\mC\tensor_{\mR}V$.
  The structure map $\bGr_k(\varphi)\colon\bGr_k(V)\to \bGr_k(W)$
  induced by a linear isometric embedding $\varphi\colon V\to W$
  takes takes the images under the complexified linear isometric embedding
  $\varphi_\mC\colon V_\mC\to W_\mC$.
  The complexification of an $\mR$-vector space $V$
  comes with a preferred $\mC$-semilinear involution
  \begin{equation} \label{eq:psi_V}
     \psi_V \ : \ V_\mC \ \to \ V_\mC \ , \quad \lambda\tensor v \ \longmapsto \
    \bar\lambda\tensor v \ . 
  \end{equation}
  The involution 
  \[ \psi(V) \ : \ \bGr_k(V)\ \to \ \bGr_k(V) \]
  that makes it an orthogonal $C$-space  takes a $\mC$-subspace $L\subset V_\mC$
  to the conjugate subspace $\bar L=\psi_V(L)$.

  By \cite[Proposition A.31]{schwede:global},
  for every augmented Lie group $\beta\colon K\to C$
  and every complete $K$-universe $\Uc_K$, the complex Stiefel manifold
  $\bL^\mC(\nu_k,\Uc_K^\mC)$ with its  $(K\times_C\tilde U(k))$-action by
  \[ (k,A)\cdot \varphi\ = \ l_k\circ \varphi\circ l_A^{-1} \]
  is a universal $(K\times_C\tilde U(k))$-space for the family
  of those closed subgroups that intersect $1\times U(k)$ trivially.
  A consequence spelled out in \cite[Theorem A.33 (i)]{schwede:global}
  is that $\bGr_k$ receives a $C$-global equivalence from
  \[ B_{\gl}\tilde U(k)\ = \  C\times_{\tilde U(k)}\bL(u(\nu_k),-) \ ,\]
  a $C$-global classifying space, in the sense of Construction \ref{con:B_gl beta},
  of the augmented Lie group $\tilde U(k)$.
\end{construction}

\begin{construction}[$\bGr$ and $\bBUP$]\label{con:Gr2BUP}
  The orthogonal $C$-space
  \[ \bGr \ = \ \coprod_{k\geq 0}\, \bGr_k \]
  is the disjoint union of the Grassmannians from Construction \ref{con:Gr_k}.
  So the value of $\bGr$ at an inner product space $V$
  is the disjoint union of all complex Grassmannians in the complexification $V_\mC$. 
  Direct sum of subspaces, plus identification 
  along the isomorphism $V_\mC\oplus W_\mC\iso(V\oplus W)_\mC$
  provide an ultra-commutative multiplication on $\bGr$, compatible
  with the involution by complex conjugation.
  This structure makes $\bGr$ an ultra-commutative $C$-monoid.
  
  The ultra-commutative monoid $\bBUP$ is the Real-global analog of
  the ultra-commutative monoid $\bBOP$ introduced in \cite[Example 2.4.1]{schwede:global},
  and its underlying global space is the complex periodic Grassmannian
  with the same name from  \cite[Example 2.4.33]{schwede:global}.
  The values of $\bBUP$ are
  \[ \bBUP(V)\ = \ {\coprod}_{n\geq 0} Gr _n^\mC(V_\mC^2)\ ,  \]
  the full Grassmannian of complex subspaces of $V_\mC^2$. The structure map
  $\bBUP(\varphi)\colon\bBUP(V)\to\bBUP(W)$ associated with a linear isometric
  embedding $\varphi\colon V\to W$ is given by
  \[ \bBUP(\varphi)(L)\ = \ \varphi_\mC^2(L)\ + \ ((W-\varphi(V))_\mC\oplus 0) \ .\]
  While the $C$-spaces $\bBUP(V)$ and $\bGr(V^2)$ are equal,
  their structure maps are different, making them distinct Real-global homotopy types.
  
  An ultra-commutative multiplication of $\bBUP$ is given by
  \[ \mu_{V,W} \ : \ \bBUP(V)\times \bBUP(W)\ \to \bBUP(V\oplus W) \ , \quad
    \mu_{V,W}(L,L')\ = \ \kappa_{V,W}(L\oplus L')\ , \]
  where $\kappa_{V,W}\colon V_\mC^2\oplus W_\mC^2\iso (V\oplus W)_\mC^2$ is
  the shuffle isomorphism $\kappa_{V,W}(v,v',w,w')=(v,w,v',w')$.
  An involution
  \[ \psi(V) \ : \ \bBUP(V) \ \to \ \bBUP(V) \]
  is defined in the same way as for $\bGr$ by applying the complex
  conjugation involution $\psi_V^2\colon V_\mC^2\to V_\mC^2$ to complex subspaces.
  All these data make $\bBUP$ an ultra-commutative $C$-monoid.

  We define a morphism of ultra-commutative $C$-monoids 
  \begin{equation}\label{eq:Gr2BUP}
    i \ : \  \bGr\ \to \ \bBUP
  \end{equation}
  at a euclidean inner product space $V$ by
  \begin{align*}
    \bGr(V)\ = \ {\coprod}_{k\geq 0}\,  Gr_k^\mC( V_\mC) \
    &\to \  {\coprod}_{n\geq 0}\, Gr_n^\mC(V^2_\mC) = \bBUP(V)   \ ,\quad L \ \longmapsto \ V_\mC\oplus L\ .
\end{align*}
\end{construction}

In \cite[Proposition 2.4.5]{schwede:global}, we show
in the real (with small `r') context,
that for every compact Lie group $G$ and every $G$-space $A$, the homomorphism
$[A,i^\mR]^G \colon[A,\bGr^\mR]^G\to [A,\bBOP]^G$
is a group completion of abelian monoids.
All arguments in the proof carry over almost literally to our presents
Real-global context, and thereby show the following result:

\begin{prop}\label{prop:pointwise_groupcomplete}
  For every augmented Lie group $\alpha\colon G\to C$ and
  every  $G$-space $A$, the homomorphism
  \[ [A,i]^\alpha \ : \ [A,\bGr]^\alpha\ \to \ [A,\bBUP]^\alpha \]
  is a group completion of abelian monoids.
\end{prop}

\begin{construction}[Real-equivariant vector bundles]\label{con:K_G_via_spc}
  We recall the notion of Real-equivariant vector bundles
  for an augmented Lie group $\alpha\colon G\to C$. This concept encompasses
  real and complex equivariant vector bundles, and Atiyah's Real vector bundles \cite{atiyah:KR}.
  A {\em Real $\alpha$-vector bundle} over a $G$-space $A$ is the data of
  \begin{itemize}
  \item a complex vector bundle $\xi\colon E\to A$,
  \item a continuous $G$-action on the total space $E$,
  \end{itemize}
  such that the projection $\xi$ is $G$-equivariant, and for all
  $(g,a)\in G\times A$, the translation map $g\cdot-\colon E_a\to E_{g a}$
  is $\alpha(g)$-linear.
  In other words, translation by $g$ is $\mC$-linear if $\alpha(g)=1$;
  and translation by $g$ is conjugate-linear if $\alpha(g)\ne 1$.

  A key example is the tautological vector bundle $\gamma_V$ over the Grassmannian
  $\bGr(V)=\coprod_{k\geq 0}G r^\mC_k(V_\mC)$ for an orthogonal $G$-representation $V$.
  Here $G$ acts diagonally on $\bGr(V)$,
  through the complexification of the given action on $V$,
  and through complex conjugation along the augmentation  $\alpha\colon G\to C$.
  In other words, for a complex subspace $L$ of $V_\mC$, we set
  \[ g\cdot L\ = \
    \begin{cases}
  \quad  l_g^\mC(L) & \text{ for $\alpha(g)=1$, and}\\
      \psi_V(l_g^\mC(L))= l_g^\mC(\psi_V(L))
                    & \text{ for $\alpha(g)\ne 1$.}
    \end{cases}  \]
  Here $l_g\colon V\to V$ is translation by $g\in G$, and
  $l_g^\mC\colon V_\mC\to V_\mC$ is its complexification.

  The total space of the tautological bundle $\gamma_V$ over $\bGr(V)$ is
  \[ \{ (w,L)\ \in V_\mC\times \bGr(V)\colon w\in L\} \ ;\]
  this bundle does not have constant rank.
  The group $G$ acts on the total space by
  \[ g\cdot (w,L)\ = \
    \begin{cases}
      ( l_g^\mC(w), l_g^\mC(L)) & \text{ for $\alpha(g)=1$, and}\\
      ( \psi_V(l_g^\mC(w)), \psi_V(l_g^\mC(L))) & \text{ for $\alpha(g)\ne 1$.}
    \end{cases}  \]
  These data make $\gamma_V$ a Real $\alpha$-vector bundle over $\bGr(V)$.  
\end{construction}

We recall in Examples \ref{eg:trivialaug} and \ref{eg:productaug}
how Real-equivariant vector bundles specialize to real and complex
equivariant vector bundles for trivial and product augmentations, respectively.
For the group $C$ augmented by the identity, 
Real-equivariant vector bundles specialize to Atiyah's Real vector bundles
from \cite{atiyah:KR}.
In \cite[Section 5]{atiyah:Bott and elliptic} and
\cite[Section 6]{atiyah-segal:completion}, Atiyah and Segal
consider what they call `Real Lie groups', i.e., Lie groups $G$ equipped with a multiplicative
involution $\tau\colon G\to G$; this corresponds to case of a semi-direct
product $G\rtimes_\tau C$, augmented by the projection to the second factor.
The case of general augmented Lie groups is considered by Karoubi
in \cite{karoubi:algebres_clifford, karoubi:sur la K-theorie}.\medskip

Now we explain in which sense the Real-global space $\bGr$ represents
Real-equivariant vector bundles, and
in which sense the Real-global space $\bBUP$ represents
Real-equivariant K-theory. The constructions and theorems are adaptations
of results from \cite[Section 2.4]{schwede:global} from the
global to the Real-global context.

\begin{construction}
  We let $\alpha\colon G\to C$ be an augmented Lie group, and $k\geq 0$.
  We write $\Vect_\alpha^{k,R}(A)$ for the set of isomorphism classes
  of Real $\alpha$-vector bundles of rank $k$ over a $G$-space $A$.
  We define a map
  \begin{equation}\label{eq:Gr^k_to_Vect^k}
    \td{-} \ :\  [A,\bGr_k]^\alpha \ = \ 
    \colim_{V\in s(\Uc_G)}\, [A,\alpha^\flat(\bGr_k(V)) ]^G \ \to \  \Vect_\alpha^{k,R}(A)
  \end{equation}
  as follows.
  We let $f\colon A\to \alpha^\flat(\bGr_k(V))$ be a continuous $G$-map, for some
  orthogonal $G$-representation~$V$.
  We pull back the tautological Real $\alpha$-vector bundle $\gamma_V^k$ over $\bGr_k(V)$
  and obtain a Real $\alpha$-vector bundle
  $f^\ast(\gamma_V^k)\colon E\to A$ of rank $k$.
  Since the base $\bGr_k(V)$ of the tautological bundle is compact,
  the isomorphism class of the bundle $f^\ast(\gamma_V^k)$ 
  depends only on the $G$-homotopy class of $f$. 
  So the construction yields a well-defined map
  \[ [A,\alpha^\flat(\bGr_k(V)) ]^G \ \to \  \Vect_\alpha^{k,R}(A) \ , \quad
    [f]\ \longmapsto \ [f^\ast(\gamma_V^k)]\ .\]
  If $\varphi\colon V\to W$ is a linear isometric embedding of
  orthogonal $G$-representations,
  then the restriction along $\bGr_k(\varphi)\colon\bGr_k(V)\to\bGr_k(W)$
  of the tautological Real $\alpha$-vector bundle $\gamma_W$ over $\bGr_k(W)$
  is isomorphic to the tautological Real $\alpha$-vector bundle $\gamma_V$ over $\bGr_k(V)$.
  So the two $\alpha$-vector bundles $f^\ast(\gamma_V)$ 
  and $(\bGr_k(\varphi)\circ f)^\ast(\gamma_W)$ over $A$ are isomorphic.
  We can thus pass to the colimit over the poset $s(\Uc_G)$
  of finite-dimensional $G$-subrepresentations of $\Uc_G$,
  and get a well-defined map \eqref{eq:Gr^k_to_Vect^k}.

  We let $\Vect_\alpha^R(A)$ denote the commutative monoid, under Whitney sum, 
  of isomorphism classes of Real $\alpha$-vector bundles over $A$.
  We define a monoid homomorphism
  \begin{equation}\label{eq:Gr_to_Vect}
    \td{-} \ :\  [A,\bGr]^\alpha  \ \to \  \Vect_\alpha^R(A)
  \end{equation}
  in much the same way as the map \eqref{eq:Gr^k_to_Vect^k}, with the main difference
  that now the vector bundles need not have constant rank.
  The map \eqref{eq:Gr_to_Vect} is a monoid homomorphism
  because all additions in sight arise from direct sum of inner product spaces.

Now we group complete the picture.
We denote by $KR_\alpha(A)$ the $\alpha$-equivariant Real K-group of $A$, i.e.,
the group completion (Grothendieck group) of the abelian monoid $\Vect_\alpha^R(A)$.
In some other references, the Real K-group $KR_\alpha(A)$ is denoted $KR_G(A)$,
i.e., only the group $G$, but not the augmentation $\alpha\colon G\to C$,
is recorded in the notation.
The composite
\[ [A,\bGr]^\alpha \ \xra[\eqref{eq:Gr_to_Vect}]{\ \td{-}\ }\ \Vect_\alpha^R(A)\ \to \ \ KR_\alpha(A) \]
is a monoid homomorphism to an abelian group. 
The morphism $[A,i]^\alpha\colon [A,\bGr]^\alpha \to [A,\bBUP]^\alpha$
is a group completion of abelian monoids by Proposition \ref{prop:pointwise_groupcomplete},
where $i\colon \bGr\to\bBUP$ was defined in \eqref{eq:Gr2BUP}.
So there is a unique homomorphism of abelian groups 
\begin{equation}\label{eq:BUP_to_KR_alpha}
 \td{-} \ : \ [A,\bBUP]^\alpha \ \to \  KR_\alpha(A)   
\end{equation}
such that the following square commutes:
\begin{equation}\begin{aligned}\label{eq:bGr to bBRP square}
 \xymatrix@C=10mm{ 
[A,\bGr]^\alpha \ar[d]_{[A,i]^\alpha} \ar[r]^-{\td{-}} &
\Vect_\alpha^R(A)\ar[d]\\
[A,\bBUP]^\alpha \ar[r]_-{\td{-}} & KR_\alpha(A) }     
\end{aligned}\end{equation}
\end{construction}

The next theorem generalizes \cite[Theorem 2.4.10]{schwede:global}
to the Real-equivariant context.

\begin{theorem}\label{thm:BUP_represents}
  Let $\alpha\colon G\to C$ be an augmented Lie group, and let $A$ be
  a finite $G$-CW-complex.
  \begin{enumerate}[\em (i)]
  \item For every $k\geq 0$, the map \eqref{eq:Gr^k_to_Vect^k}
    \[ \td{-} \ :\  [A,\bGr_k]^\alpha \ \to \  \Vect_\alpha^{k,R}(A) \]
    is bijective.
    \item The monoid homomorphism \eqref{eq:Gr_to_Vect}
  \[ \td{-}\ : \  [A,\bGr]^\alpha\ \to \ \Vect^R_\alpha(A) \]
  is an isomorphism.
    \item The group homomorphism \eqref{eq:BUP_to_KR_alpha}
  \[ \td{-}\ : \  [A,\bBUP]^\alpha\ \to \ KR_\alpha(A) \]
  is an isomorphism.
  \end{enumerate}\end{theorem}
\begin{proof}
  (i) This statement is proved by essentially the same arguments
  as its real (with small `r') predecessor in \cite[Theorem 2.4.10]{schwede:global},
  mutatis mutandis.
  The argument uses that the complex Stiefel manifold
  $\bL^\mC(\nu_k,\Uc_G^\mC)$ with its  $(G\times_C\tilde U(k))$-action by
  \[ (g,A)\cdot \varphi\ = \ l_g\circ \varphi\circ l_A^{-1} \]
  is a universal $(G\times_C\tilde U(k))$-space for the family
  of those closed subgroups that intersect $1\times U(k)$ trivially,
  see \cite[Proposition A.31]{schwede:global}. This, in turn, implies
  that the $G$-space $\alpha^\flat( \bGr_k(\Uc_G^\mC))=\bL^\mC(\nu_k,\Uc_G^\mC)/U(k)$
  is a classifying space for rank $k$ Real $\alpha$-equivariant vector bundles
  over compact $G$-spaces.

  (ii) We suppose first that $A$ is $G$-connected, i.e., 
  the group $\pi_0(G)$ acts transitively on $\pi_0(A)$.
  This ensures that every Real $\alpha$-vector bundle over $A$ has constant rank,
  and that every continuous $G$-map $A\to \alpha^\flat(\bGr(V))$ factors through
  $\alpha^\flat(\bGr_k(V))$ for some $k\geq 0$. Hence both vertical maps in the
  following commutative square, induced by the inclusions $\bGr_k\to\bGr$, are bijective:
  \[ \xymatrix@C=20mm{
      \coprod_{k\geq 0}\, [A,\bGr_k]^\alpha\ar[r]_-{\eqref{eq:Gr^k_to_Vect^k}}^-{\td{-}}\ar[d] &   \coprod_{k\geq 0} \Vect^{k,R}_\alpha(A)\ar[d]\\
       [A,\bGr]^\alpha\ar[r]_-{\td{-}}^-{\eqref{eq:Gr_to_Vect}}& \Vect^R_\alpha(A)    } \]
   The upper horizontal map is bijective by (i), hence the lower horizontal map is bijective.

   In general $A=A_1\amalg\ldots\amalg A_m$ is the disjoint union of $G$-connected
   finite $G$-CW-complexes, indexed by the $\pi_0(G)$-orbits of $\pi_0(A)$.
   Since both functors $[-,\bGr]^\alpha$ and $\Vect_\alpha^R$
   take finite disjoint unions to products, the general case follows.
   
  Part (iii) follows from (ii) because in the commutative
  square \eqref{eq:bGr to bBRP square} both vertical maps are group
  completions of abelian monoids, by 
  Proposition \ref{prop:pointwise_groupcomplete} and by definition, respectively.
\end{proof}

As we shall explain in the next two examples,
the isomorphism $[A,\bBUP]^\alpha\iso KR_\alpha(A)$ generalizes
analogous isomorphisms for real K-groups and for complex K-groups
that are already discussed in \cite[Section 2.4]{schwede:global}.

\begin{eg}[Trivially augmented Lie groups]\label{eg:trivialaug}
  For a compact Lie group $G$, we write $G^{\tr}$ for the trivially augmented
  Lie group, i.e., $G$ endowed with the trivial homomorphism to $C$.
  Then Real $G^{\tr}$-vector bundles are nothing but complex vector bundles,
  and thus
  \[ KR_{G^{\tr}}(A)\ =\ K U_G(A)\ . \]
  The underlying orthogonal space of $\bBUP$ is
  the complex periodic Grassmannian with the same name from \cite[Example 2.4.33]{schwede:global}. 
  For trivially augmented Lie groups, Theorem \ref{thm:BUP_represents}
  thus specializes to an isomorphism
  \[ \td{-}\ : \  [A,\bBUP]^{G^{\tr}}\ \xra{\ \iso \ } \ K U_G(A) \ ,\]  
  the complex analog of the isomorphism from \cite[Theorem 2.4.10]{schwede:global}. 
\end{eg}

\begin{eg}[Product augmented Lie groups]\label{eg:productaug}
  We augment the product of a compact Lie group $G$ with the Galois group $C$
  by the projection $G\times C\to C$ to the second factor.
  In the following, we shall leave the projection implicit
  and simply write $G\times C$ for this augmented Lie group.

  We consider a $G$-space $A$, and we let the group $C$ act trivially on $A$.
  Then a Real $(G\times C)$-vector bundle over $A$ is nothing but a real
  $G$-vector bundle (with small `r'). More precisely,
  the action of the element $(1,\psi)\in G\times C$
  on the fiber $E_a=\xi^{-1}(a)$ of a Real $(G\times C)$-vector bundle
  is a conjugate linear involution on the complex vector space $E_a$, i.e.,
  a real structure. The $\mR$-subspaces $(E_a)^\psi$ fixed by
  this involution form a real $G$-vector subbundle
  \[ \xi^\psi\ = \ \ker( \Id- (1,\psi)\cdot-\colon \xi \to \xi) \]
  of $\xi$, and $\xi$ is naturally isomorphic to the complexification of $\xi^\psi$.
  This construction implements an equivalence of categories between 
  Real $(G\times C)$-vector bundles and real $G$-vector bundles over
  any $G$-space with trivial $C$-action.
  The equivalence induces an isomorphism
  \begin{equation}\label{eq:KR_for_GxC}
    KR_{G\times C}(A)\ \iso\ KO_G(A)\ .     
  \end{equation}
  The orthogonal subspace of $\bBUP$ fixed by the involution `is'
  the periodic Grassmannian $\bBOP$ from \cite[Example 2.4.1]{schwede:global}. 
  More precisely, the complexification maps
  \[ \bBOP(V)\ = \ \amalg_{n\geq 0}\, G r_n^\mR(V^2)\ \to \
    \amalg_{n\geq 0}\, G r^\mC_n(V^2_\mC)\ = \ \bBUP(V)\ ,\quad L \ \longmapsto \ \mC\tensor_\mR L   \]
  form an isomorphism of ultra-commutative monoids to the $\psi$-fixed subobject $\bBUP^\psi$.
  For any $G$-space $A$ endowed with trivial $C$-action, it thus induces an isomorphism
  \begin{equation}\label{eq:BOPvsBRP}
    [A,\bBOP]^G\ \xra{\ \iso \ }\ [A,\bBUP]^{G\times C}\ .    
  \end{equation}
  Under the identifications \eqref{eq:KR_for_GxC} and \eqref{eq:BOPvsBRP},
  Theorem \ref{thm:BUP_represents} thus specializes to the isomorphism
  \[ \td{-}\ : \  [A,\bBOP]^G\ \xra{\ \iso \ } \ KO_G(A) \]  
  from \cite[Theorem 2.4.10]{schwede:global}. 
\end{eg}

\begin{construction}[The ultra-commutative $C$-monoid $\bU$]\label{con:U}
  We recall the  ultra-commutative monoid $\bU$ made from unitary groups,
  compare \cite[Example 2.37]{schwede:global}.
  The euclidean inner product $\td{-,-}$ on $V$ induces a hermitian
  inner product $(-,-)$ on the complexification $V_\mC=\mC\tensor_\mR V$,
  defined as the unique sesquilinear form that satisfies
  $(1\tensor v, 1\tensor w) =\td{v,w}$
  for all $v, w\in V$.
  The value of the orthogonal space $\bU$ on $V$ is
    \[ \bU(V)\ = \ U(V_\mC) \ ,\]
  the unitary group of the complexification of $V$.
  The complexification of every $\mR$-linear isometric embedding $\varphi\colon V\to W$ 
  preserves the hermitian inner products, 
  so we can define a continuous monomorphism
  \[ \bU(\varphi)\ :\ \bU(V)\ \to\ \bU(W)  \]
  by conjugation with $\varphi_\mC\colon V_\mC\to W_\mC$
  and the identity on the orthogonal complement of the image of $\varphi_\mC$.
  The commutative multiplication of $\bU$ is given by the direct sum
  of unitary automorphisms
  \[ \bU(V)\times\bU(W)\ \to \ \bU(V\oplus W) \ , \quad (A,B)\ \longmapsto \ A\oplus B\ ,\]
  where we implicitly used the preferred complex isometry $V_\mC\oplus W_\mC\iso(V\oplus W)_\mC$.
  An involutive automorphism of ultra-commutative monoids $\psi\colon \bU \to \bU$
  is given at $V$ by the map
  \[ \psi(V) \ : \ U(V_\mC) \ \to \ U(V_\mC) \ , \quad 
    A \ \longmapsto \ \psi_V\circ A\circ\psi_V \ , \]
  where $\psi_V\colon V_\mC\to V_\mC$
  is the $\mC$-semilinear conjugation involution \eqref{eq:psi_V}.
  This involution makes $\bU$ an orthogonal $C$-space, representing
  an unstable Real-global homotopy type.
\end{construction}

We write $\bBU$ for the orthogonal $C$-subspace of $\bBUP$ with values
\[ \bBU(V)\ = \  Gr _{|V|}^\mC(V_\mC^2)\ ,  \]
where $|V|=\dim_\mR(V)$.
This subobject is the homogeneous summand of degree~0 in the
$\mZ$-grading on $\bBUP$, and closed under the ultra-commutative multiplication.
Our next aim is to show that the ultra-commutative $C$-monoid $\bBU$ is
a $C$-global deloop of $\bU$, as the notation suggests. The delooping
will be witnessed by a zigzag of two Real-global equivalences.

\begin{construction}
  We define an ultra-commutative $C$-monoid $F$ and morphisms of
  ultra-commutative monoids 
  \[ \bU \  \xra{\ g\ }\ F \ \xla{\ h\ } \ \Omega(\bBU)\ .\]
  To this end we introduce an auxiliary ultra-commutative $C$-monoid $\Lc$.
  Its value at an inner product space  is the Stiefel manifold
  \[ \Lc(V)\ = \ \bL^\mC(V_\mC,V_\mC^2) \]
  of $\mC$-linear isometric embeddings of $V_\mC$ into $V_\mC^2$. The groups $O(V)$
  and the non-identity element $\psi\in C$ act by conjugation, i.e., by
  \[ {^A\varphi}\ = \ A_\mC^2\circ \varphi\circ A_\mC^{-1} \text{\qquad and\quad}
    {^\psi\varphi}\ = \ \psi_V^2\circ \varphi\circ \psi_V \]
  for $(A,\varphi)\in O(V)\times \bL^\mC(V_\mC,V_\mC^2)$.
  An ultra-commutative multiplication is defined by  
  \[ \Lc(V)\times \Lc(W)\ \to \ \Lc(V\oplus W) \ , \quad (\varphi,\phi)\ \longmapsto \
    \kappa_{V,W}\circ (\varphi\oplus\phi)\ ,  \]
  where $\kappa_{V,W}\colon V_\mC^2\oplus W^2_\mC \to(V\oplus W)_\mC^2$ is
  $\kappa_{V,W}(v,v',w,w')=(v,w,v',w')$.
  The unit is $i_1\colon V_\mC\to V_\mC^2$, $i_1(v)=(v,0)$,
  which is $(O(V)\times C)$-fixed and multiplicative.

  A morphism of ultra-commutative $C$-monoids $\im\colon\Lc\to \bBU(V)$
  sends $\varphi\in\bL(V_\mC,V_\mC^2)$ to its image $\im(\varphi)\in Gr_{|V|}^\mC(V_\mC^2)=\bBU(V)$.
  We define the ultra-commutative monoid $F$ as the homotopy fiber
  of the morphism $\im\colon\Lc\to\bBU(V)$ over the unit of $\bBU$, i.e., by a pullback diagram in
  the category of ultra-commutative $C$-monoids:
  \begin{equation}
    \begin{aligned}\label{eq:defineP}      
      \xymatrix@C=20mm{ F \ar[d]_q\ar[r]^-p &   \Lc \ar[d]^{\im} \\
      \{V_\mC\oplus 0\}\times_{\bBU}\bBU^{[0,1]}\ar[r]_-{\ev_1} & \bBU }
    \end{aligned}
  \end{equation}
  Explicitly, $F(V)$ is the space of all pairs $(\omega,\varphi)\in\bBU(V)^{[0,1]}\times\Lc(V)$
  consisting of a path $\omega\colon[0,1]\to\bBU(V)=Gr_{|V|}(V_\mC^2)$
  and a linear isometric embedding $\varphi\colon V_\mC\to V_\mC^2$ such that 
  \[ \omega(0)\ =\  V_\mC\oplus 0 \text{\quad and\quad} \omega(1)\ =\ \im(\varphi)\ . \] 
  As a pullback, the $O(V)$-action, structure maps, involution, multiplication and unit are
  all inherited from $\bBU$ and $\Lc$.

  A morphism of ultra-commutative $C$-monoids $f\colon \bU\to F$ is given at $V$ by
  \[ f(V)\ :\ \bU(V) =U(V_\mC)\ \to \ F(V)\ , \quad A\ \longmapsto\ (\text{const}_{V_\mC\oplus 0},i_1\circ A)\ .\]
  Here $\text{const}_{V_\mC\oplus 0}\colon[0,1]\to \bBU(V)$
  is the constant path at $V_\mC\oplus 0$, and $i_1\colon V_\mC\to V_\mC^2$
  is the embedding as the first summand.
  A morphism of ultra-commutative monoids $g\colon\Omega(\bBU) \to F$ is given at $V$ by
  \[ g(V)\ :\ \map_*(S^1,\bBU(V)) \ \to \ F(V)\ , \quad \omega\ \longmapsto\ (\omega\circ t,i_1)\ ,\]
  where $t\colon [0,1]\to S^1$ is the continuous map $t(x)=(2 x-1)/(x(1-x))$
  that factors through a homeomorphism $[0,1]/\{0,1\}\iso S^1$. 
\end{construction}

\begin{prop}\label{prop:BUdeloopsU}
  The morphisms
  \[ f\ :\ \bU \ \to F \text{\qquad and\qquad} g\ :\  \Omega(\bBU)\ \to \ F \]
  are $C$-global equivalences of ultra-commutative $C$-monoids.  
\end{prop}
\begin{proof}
  We start with the morphism $f$, which we show is even a $C$-level equivalence.
  We let $\alpha\colon G\to C$ be an augmented Lie group,
  with graph $\Gamma=\{(\alpha(\gamma),\gamma)\colon \gamma\in G\}$.
  We let $V$ be an orthogonal $G$-representation.
  The map 
  \begin{equation}\label{eq:im(V)^Gamma}
   \im(V)^\Gamma\ : \ \Lc(V)^\Gamma\ = \ (\bL^\mC(V_\mC,V_\mC^2))^\Gamma \ \to \
    (G r_{|V|}^\mC(V_\mC^2))^\Gamma \ = \ \bBU(V)^\Gamma  
  \end{equation}
  is a disjoint union of projections from Stiefel manifolds
  to the associated Grassmannian manifolds.
  Hence this map is a locally trivial fiber bundle, and thus a Serre fibration.
  So the map from the strict fiber of \eqref{eq:im(V)^Gamma}
  over $V_\mC\oplus 0$ to the homotopy fiber
  is a weak homotopy equivalence.
  The right action of $\bU(V)^\Gamma=U(V_\mC)^\Gamma$ on $i_1\colon V_\mC\to V_\mC^2$
  by precomposition
  identifies $\bU(V)^\Gamma$ with the strict fiber of \eqref{eq:im(V)^Gamma}.
  So the map
    \[ f(V)^\Gamma\ : \ \bU(V)^\Gamma\ \to \
      \{V_\mC\oplus 0\}\times_{\bBU(V)^\Gamma}(\bBU(V)^\Gamma)^{[0,1]}\times_{\bBU(V)^\Gamma} \Lc(V)^\Gamma
    \ = \ F(V)^\Gamma\]
  is a weak homotopy equivalence.
  This shows that the morphism $f\colon\bU\to F$ is a $C$-level equivalence in
  the sense of \cite[Theorem A.2]{barrero}, and hence in particular a $C$-global equivalence.

  To deal with the morphism $g$ we show first that the orthogonal $C$-space
  $\Lc$ is $C$-globally trivial.
  For every inner product space $V$, the homotopy
  \begin{align*}
    H \ : \ \bL^\mC(V_\mC, V_\mC^2)\times [0,\pi/2]\ &\to\ \bL^\mC(V^2_\mC, V_\mC^2\oplus V_\mC^2)\\
    H(\varphi,t)(w,w')\ &\ = \ (\cos(t)\cdot \varphi(w), (w',\sin(t)\cdot w))
  \end{align*}
  witnesses that the map
  \[ -\oplus i_1 \ : \ \Lc(V) = \bL^\mC(V_\mC, V_\mC^2)\ \to\ \bL^\mC(V^2_\mC, V_\mC^2\oplus V_\mC^2)\]
  is $(O(V)\times C)$-equivariantly homotopic to a constant map.
  The structure map
  \[ \Lc(i_1)\ : \ \Lc(V)\ \to\ \Lc(V^2) \]
  is the composite of $-\oplus i_1$ and the homeomorphism induced by
  $\kappa_{V,V}\colon V^2_\mC\oplus V^2_\mC\iso (V\oplus V)_\mC^2$,
  so the structure map $\Lc(i_1)$ is also
 $(O(V)\times C)$-equivariantly null-homotopic.
 Hence the underlying orthogonal $C$-space of $\Lc$ is $C$-globally contractible.

  The lower horizontal evaluation morphism in the defining pullback square \eqref{eq:defineP}
  is a $C$-level fibration.
  So also its base change $p\colon F\to\Lc$ is a $C$-level fibration.
  Moreover, the right vertical morphism in the pullback square
  \begin{equation*}
    \xymatrix@C=20mm{ \Omega(\bBU) \ar[d]_g\ar[r] &\ast\ar[d]^\sim \\
      F\ar[r]_-p   & \Lc }
  \end{equation*}  
  is a $C$-global equivalence by the above. The pullback of a $C$-global equivalence
  along a $C$-level fibration is again a  $C$-global equivalence,
  see \cite[Lemma A.18]{barrero}.
  So $g\colon \Omega(\bBU)\to F$ is a $C$-global equivalence.
\end{proof}

Coinduced orthogonal $C$-spaces were introduced in Definition \ref{def:coinduced}.
As we shall now show, the representing $C$-space $\bBUP$ for Real-global K-theory
is an example. This is useful for our purposes because
$C$-global equivalences between coinduced objects can
be detected on underlying global spaces, see Proposition \ref{prop:coinduced_coinduces}.

\begin{theorem}\label{thm:coinduced}
  The orthogonal $C$-spaces $\bBUP$, $\bU$ and $\Omega^\sigma \bU$
  are coinduced.
\end{theorem}
\begin{proof}
  For showing that $\bBUP$ is coinduced we will exploit that it represents
  Real-equivariant K-theory, by Theorem \ref{thm:BUP_represents} (iii).
  We abuse notation and simply write $C$ for augmented Lie group $\Id_C\colon C\to C$.
  We let $\beta\in\widetilde{KR}_C(S^{1+\sigma})$ denote
  the Bott class associated with the 1-dimensional Real $C$-representation on $\mC$,
  see for example \cite[Theorem 2.3]{atiyah:KR} or \cite[Theorem 5.1]{atiyah:Bott and elliptic}.
  Then $i^*(\beta)$ is a class in $\widetilde{KR}_C(S^1)$, where $i\colon S^1\to S^{1+\sigma}$
  is the inclusion of the $C$-fixed points.
  Hence $(i^*(\beta))^3$ lies in the group $\widetilde{KR}_C(S^3)$, which is isomorphic to
  $\widetilde{KO}(S^3)\iso\pi_3(KO)$, and hence trivial.

  Now we let $\alpha\colon G\to C$ be an augmented Lie group.
  We abuse notation and also write $\sigma$ for
  the restriction of the sign $C$-representation along the augmentation.
  Inflation along $\alpha\colon G\to C$ is compatible
  with products in Real-equivariant K-theory, so
  $\alpha^*(i^*(\beta)^3)=(i^*(\beta_{\alpha^*(\mC)}))^3=0$
  in the group $\widetilde{KR}_\alpha(S^3)$, where
  $\beta_{\alpha^*(\mC)}\in\widetilde{KR}_\alpha(S^{1+\sigma})$
  is the Bott class of the Real $\alpha$-representation $\alpha^*(\mC)$.
  So external multiplication by the class $\alpha^*(i^*(\beta)^3)$ is
  trivial on reduced $KR_\alpha$-cohomology of every based finite $G$-CW-complex.
  An instance of such a multiplication  is the composite
  \[    \widetilde{KR}_\alpha(A\sm S^{3\sigma})\ \xra{\ i_3^*\ } \
       \widetilde{KR}_\alpha(A)\ \xra{\beta_{\alpha^*(\mC)}^3\cdot } \
    \widetilde{KR}_\alpha(A\sm S^{3(1+\sigma)})\ ,  \]
  where $i_3\colon S^0\to S^{3\sigma}$ is the inclusion of the points $\{0,\infty\}$.
  Multiplication by $\beta_{\alpha^*(\mC)}$ is an isomorphism,
  see \cite[Theorem (5.1)]{atiyah:Bott and elliptic}, so the restriction map $i^*_3$
  above is trivial.
  Thus for every $n\geq 3$, the long exact sequence in $KR_\alpha$-cohomology of the cofiber sequence
  \[ A\sm S(n\sigma)_+\ \xra{A\sm p_+} \
    A \ \xra{A\sm i_n} \      A\sm S^{n\sigma}  \]
  splits off a short exact sequence:
  \[
    0\ \to \  \widetilde{KR}_\alpha(A)
  \ \xra{(A\sm p)^*} \widetilde{KR}_\alpha(A\sm S(n\sigma)_+) \xra{\ \partial\ }\ 
   \widetilde{KR}_\alpha^1(A\sm S^{n\sigma})\ \to \ 0 \]
 These exact sequences for $n$ and $n+3$ participate in a commutative diagram:
 \[ \xymatrix@C=12mm{
     0\ar[r] &\widetilde{KR}_\alpha(A)\ar[r]^-{(A\sm p_+)^*} \ar@{=}[d]&
     \widetilde{KR}_\alpha(A\sm S((n+3)\sigma)_+) \ar[r]^-\partial \ar[d]^{(A\sm j_+)^*}& 
     \widetilde{KR}_\alpha^1(A\sm S^{(n+3)\sigma})\ar[r] \ar[d]^{(A\sm S^{n\sigma}\sm i_3)^*}& 0 \\
          0\ar[r] &\widetilde{KR}_\alpha(A)\ar[r]_-{(A\sm p_+)^*} &
     \widetilde{KR}_\alpha(A\sm S(n\sigma)_+) \ar[r]_-\partial & 
   \widetilde{KR}_\alpha^1(A\sm S^{n\sigma})\ar[r] & 0 \\
   }  \]
 Here $j\colon S(n\sigma)\to S((n+3)\sigma)$ is $j(x_1,\dots,x_n)=(x_1,\dots,x_n,0,0,0)$.
 The right vertical map is another instance of a restriction map $i^*_3$, and hence trivial.
 So the middle vertical map $(A\sm j_+)^*$ factors through the lower left
 horizontal map $(A\sm p_+)^*$. This shows that the maps $(A\sm p_+)^*$ form a pro-isomorphism
 from the constant tower with value $\widetilde{KR}_\alpha(A)$ to the tower
 \[ \dots\to \widetilde{KR}_\alpha(A\sm S((n+1)\sigma)_+) \to \widetilde{KR}_\alpha(A\sm S(n\sigma)_+) \to \dots \to
   \widetilde{KR}_\alpha(A\sm S(\sigma)_+) \ .\]
 The natural isomorphism $\td{-}\colon[A,\bBUP]^\alpha\iso KR_\alpha(A)$
 from Theorem \ref{thm:BUP_represents}
 passes to reduced groups, and provides a natural isomorphism
 \[ [A,\bBUP]^\alpha_*\ \iso \ \widetilde{KR}_\alpha(A)   \ ,\]
 where $A$ is any finite based $G$-CW-complex.
 So the maps $(A\sm p_+)^*\colon[A,\bBUP]^\alpha_*\to [A\sm S(n\sigma)_+,\bBUP]^\alpha_*$ form a pro-isomorphism
 from the constant tower with value $[A,\bBUP]^\alpha_*$ to the tower
 \[ \dots\to [A\sm S((n+1)\sigma)_+,\bBUP]^\alpha_* \to [A\sm S(n\sigma)_+,\bBUP]^\alpha_* \to \dots \to
   [A\sm S(\sigma)_+,\bBUP]^\alpha_* \ .\]
 In particular, $[A,\bBUP]^\alpha_*$ maps isomorphically to the inverse limit
 \[ [A,\bBUP]^\alpha_*\ \xra{\ \iso \ }\ \lim_n\  [A\sm S(n\sigma)_+,\bBUP]^\alpha_*  \ , \]
 and the derived inverse limit of the tower vanishes.
 The $C$-space $\bigcup_{n\geq 0}S(n\sigma)$ is a model
 for the universal free $C$-space $E C$.
 So the group $[A\sm EC_+,\alpha^\flat(\bBUP(\Uc_G))]^G_*$
 participates in a Milnor short exact sequence
 with the inverse limit and vanishing derived inverse limit.
 We conclude that the map
 \[ (A\sm p_+)^*\ : \ [A,\bBUP(\Uc_G)]^\Gamma_*\ \to \ [A\sm EC_+,\bBUP(\Uc_G)]^\Gamma_*\ \iso\
   [A,\map(EC,\bBUP(\Uc_G))]^\Gamma_* \]
 is an isomorphism for every finite based $G$-CW-complex $A$,
 where $\Gamma$ is the graph of $\alpha\colon G\to C$.
 Taking $A=S^k$, for $k\geq 0$ and with trivial $G$-action shows that the map
 \[ (p^*)^\Gamma \ : \ \bBUP(\Uc_G)^\Gamma\ \to \ \map^\Gamma(EC,\bBUP(\Uc_G)) \]
 induces a bijection on $\pi_0$ and isomorphisms of homotopy groups
 at the distinguished basepoint. Since $\bBUP(\Uc_G)$ is a group-like $G$-$E_\infty$-space,
 $(p^*)^\Gamma$ is a map of non-equivariant group-like $E_\infty$-spaces,
 and thus a weak homotopy equivalence.
 Proposition \ref{prop:closed_criterion} shows that the orthogonal $C$-space
 $\bBUP$ is coinduced.

 Since $\bBUP$ is coinduced, so is $\Omega(\bBUP)=\Omega(\bBU)$
 by Proposition \ref{prop:loop_rind}.
 Since $\bU$ is $C$-globally equivalent to $\Omega(\bBU)$
 by Proposition \ref{prop:BUdeloopsU}, it is coinduced.
 Then $\Omega^\sigma\bU$ is coinduced by Proposition \ref{prop:loop_rind}.
\end{proof}

In \cite[Theorem 2.5.41]{schwede:global}, we establish a global equivariant
form of Bott periodicity by exhibiting a zigzag of two global equivalences
of ultra-commutative monoids between $\bBUP$ and $\Omega\bU$.
We shall now lift this to a Real-global equivalence
between $\bBUP$ and $\Omega^\sigma\bU$.

\begin{construction}
We endow $\Omega^\sigma\bU=\map_*(S^\sigma,\bU)$ with the involution
that is diagonal from the sign action on $S^\sigma$ and the complex conjugation
involution of $\bU$. This becomes an  ultra-commutative $C$-monoid
via the pointwise multiplication inherited from $\bU$.
We let $\sh_\tensor\bU=\sh_\tensor^{\mR^2}\bU$ denote the `multiplicative shift'
of $\bU$ by $\mR^2$ in the sense of \cite[Example 1.1.11]{schwede:global}.
The values of this orthogonal space are thus given by
\[ (\sh_\tensor\bU)(V)\ =\ \bU(V\tensor\mR^2)\ = \ U((V\tensor\mR^2)_\mC)\ . \]
The structure maps, ultra-commutative multiplication and an involution are
inherited from $\bU$, making $\sh_\tensor\bU$ an ultra-commutative $C$-monoid.
The embeddings 
\[ j(V)\ =\ -\tensor (1,0) \ : \ V\ \to\  V\tensor\mR^2 \]
as the first summand induce a morphism of ultra-commutative monoids
\[ \bU\circ j\ :\ \bU\ \to\ \sh_\tensor\bU\ . \]
On \cite[page 224]{schwede:global} we define a morphism of ultra-commutative monoids
$\bar\beta \colon \bBUP \to \Omega(\sh_\tensor \bU)$.
At that point we were not taking any involutions into account.
However, inspection of all definitions shows that this morphism
is $C$-equivariant with respect to the complex conjugation involutions
on $\bBUP$ and $\bU$, and the sign involution in the loop coordinate in the target;
the sign involution arises because the Cayley transform $\mathfrak c\colon S^1\to U(1)$
is $C$-equivariant for the sign involution on the source and complex conjugation on the target.
So $\bar\beta$ is also a morphism of ultra-commutative $C$-monoids
\[ \bar\beta \ : \ \bBUP \ \to \ \Omega^\sigma(\sh_\tensor \bU) \ . \]
\end{construction}

The following result generalizes \cite[Theorem 2.5.41]{schwede:global}
to the Real-global context.

\begin{theorem}[Real-global Bott periodicity]\label{thm:BUP2OmegaU}
  The morphisms of ultra-commutative $C$-monoids
  \[ \bBUP\ \xra{\ \bar\beta\ } \ \Omega^\sigma(\sh_\tensor\bU) \
    \xla{\Omega^\sigma (\bU\circ j)}\ \Omega^\sigma\bU
  \]
  are Real-global equivalences.
\end{theorem}
\begin{proof}
  The morphism of ultra-commutative $C$-monoids $\bU\circ j\colon\bU\to \sh_\tensor \bU$
  is a Real-global equivalence by \cite[Lemma 3.8]{barrero}. Hence 
  $\Omega^\sigma(\bU\circ j)\colon\Omega^\sigma \bU\to\Omega^\sigma(\sh_\tensor \bU)$ 
  is a Real-global equivalence, too.
  The orthogonal $C$-spaces $\bBUP$ and $\Omega^\sigma\bU$
  are coinduced by Theorem \ref{thm:coinduced}.
  The morphism $\Omega^\sigma (\bU\circ j)$ is a Real-global equivalence,
  so $\Omega^\sigma(\sh_\tensor\bU)$ is also coinduced.
  The morphism $\bar\beta\colon\bBUP\to \Omega^\sigma(\sh_\tensor\bU)$
  is a global equivalence of underlying non-Real orthogonal spaces
  by  \cite[Theorem 2.5.41]{schwede:global}.
  So $\bar\beta$ is a Real-global equivalence
  by Proposition \ref{prop:coinduced_coinduces}.
\end{proof}

  Proposition \ref{prop:BUdeloopsU} and  Theorem \ref{thm:BUP2OmegaU}
  provide Real-global equivalences of ultra-commutative $C$-monoids
  \[ \Omega^{\sigma+1}(\bBUP)\ = \ \Omega^\sigma(\Omega(\bBUP))\ = \
    \Omega^\sigma(\Omega(\bBU))\ \sim \   \Omega^\sigma\bU\ \sim \ \bBUP\ .\]
  These Real-global equivalences implement a highly structured
  refinement of the $(\sigma+1)$-periodicity, also called $(2,1)$-periodicity,
  of Real-equivariant K-theory.

\begin{construction}
  We denote by
  \begin{equation}\label{eq:define_gamma}
    \gamma\ : \ \bBUP\ \xra{\ \sim \ } \ \Omega^\sigma\bU    
  \end{equation}
  the unique morphism in the homotopy category of ultra-commutative $C$-monoids
  such that $\Omega^\sigma(\bU\circ j)\circ\gamma=\bar\beta$.
  A morphism of ultra-commutative $C$-monoids $\beta\colon\bGr\to \Omega^\sigma\bU$
  is defined in \cite[(2.5.38)]{schwede:global}; its value 
  \begin{equation}\label{eq:define_beta}
    \beta(V)\ : \  \bGr(V) = G r^\mC(V_\mC)\ \to \
    \Omega^\sigma( U(V_\mC)) = (\Omega^\sigma \bU)(V)   
  \end{equation}
  at an inner product space $V$ sends a complex subspace $L\subset V_\mC$
  to the loop $\beta(V)(L)\colon S^\sigma\to U(V_\mC)$ such that $\beta(V)(L)(x)$ is
  multiplication by $\mathfrak c(x)\in U(1)$ on $L$,
  and the identity on the orthogonal complement of $L$.
  The morphism $\beta$ is $C$-equivariant with respect
  to the complex conjugation involutions on $\bGr$ and $\bU$,
  and the sign involution in the loop coordinate in the target,
  one more time because the Cayley transform $\mathfrak c\colon S^\sigma\to U(1)$
  is $C$-equivariant for the sign involution on the source and complex conjugation on the target.
  Inspection of definitions shows that the following square commutes:
  \[ \xymatrix{
      \bGr \ar[d]_i \ar[r]^-{\beta} & \Omega^\sigma\bU \ar[d]^{\Omega^{\sigma}(\bU\circ j)}_\sim\\
      \bBUP  \ar[r]_-{\bar\beta}^-\sim & \Omega^\sigma(\sh_\tensor\bU)
    } \]
  So
  \[    \Omega^\sigma(\bU\circ j)\circ \gamma\circ i\ = \
    \bar\beta\circ i\ = \ \Omega^\sigma(\bU\circ j)\circ \beta \ . \]
  Since $\Omega^\sigma(\bU\circ j)$ is a Real-global equivalence, this implies the relation
  \begin{equation}\label{eq:gamma_i=beta}
 \gamma\circ i\ = \ \beta \ : \  \bGr\ \to \ \Omega^\sigma\bU
  \end{equation}
  as morphisms in the homotopy category of Real-global ultra-commutative monoids.
\end{construction}

\begin{construction}[The inverse morphism of $\bU$]
  The inverse maps 
  \[ \Im(V)\ : \ \bU(V) = U(V_\mC)\ \to\  U(V_\mC)= \bU(V)\ , \quad\Im(V)(A)=A^{-1}  \]
  are compatible with complex conjugation and make the following diagrams commute:
  \[ \xymatrix@C=20mm{ \bU(V)\times  \bU(W)\ar[r]^-{\Im(V)\times\Im(W)}\ar[d]_{\oplus} &
      \bU(V)\times\bU(W)\ar[d]^{\oplus}\\
      \bU(V\oplus W)\ar[r]_-{\Im(V\oplus W)}&       \bU(V\oplus W)
    } \]
  So they form a morphism of ultra-commutative $C$-monoids $\Im\colon\bU\to\bU$
  that models the inverse map in the ultra-commutative addition.
  Indeed, composition of unitary automorphism
  is a morphism of  ultra-commutative $C$-monoids
  $\circ\colon\bU\times\bU\to\bU$ that makes the following diagram commute:
  \begin{equation} \begin{aligned}\label{eq:Inv_is_Inv}
 \xymatrix{ & \bU\boxtimes \bU\ar@/^1pc/[dr]^(.7){\mu^\bU}\ar[d]_{(\rho_1,\rho_2)}^\sim \\
      \bU \ar[r]_-{(\Id,\Im)} & \bU\times\bU \ar[r]_-\circ & \bU}       
    \end{aligned}
  \end{equation}
  Here $\rho_1,\rho_2\colon \bU\boxtimes\bU\to\bU$ are the projections to the two factors;
  the combined morphism
  $(\rho_1,\rho_2)\colon \bU\boxtimes\bU\to\bU\times\bU$ is a $C$-global equivalence
  by \cite[Proposition 3.9]{barrero}.
 The lower horizontal composite is constant with value the identity
  automorphism, and hence the zero morphism in the ultra-commutative multiplication.
  So this diagram witnesses that $\Im$ is the inverse morphism.
\end{construction}

 As before, we write $\epsilon\colon S^\sigma\to S^\sigma$ for the sign involution, $\epsilon(x)=-x$.
 Precomposition with $\epsilon$ is an involution
 $\epsilon^*\colon \Omega^\sigma\bU \to \Omega^\sigma\bU$.
 Under the Cayley transform \eqref{eq:define_Cayley}, the sign involution
 $\epsilon$ becomes complex conjugation on $U(1)$.
 So the unitary automorphisms $\beta(V)(L)(x)\in \bU(V)$ and $\beta(V)(L)(\epsilon(x))$
 have the same eigenspace $L$, but inverse eigenvalues.
 Hence $\beta(V)(L)(x)$ and $\beta(V)(L)(\epsilon(x))$ are inverse to each other.
 This shows that the following diagram of morphisms of ultra-commutative $C$-monoids commutes:
 \begin{equation}   \begin{aligned}\label{eq:beta_u_Im}
 \xymatrix{
       \bGr\ar[r]^-{\beta} \ar[d]_{\beta}& \Omega^\sigma \bU\ar[d]^{\epsilon^*} \\
       \Omega^\sigma\bU\ar[r]_-{\Omega^\sigma\Im} & \Omega^\sigma \bU     }      
   \end{aligned}\end{equation}

\begin{prop}\label{prop:u_on_Omega^sigma U}
  Let $\alpha\colon G\to C$ be an augmented Lie group, and let $A$
  be a finite $G$-CW-complex.
  Then the map
  $[A,\epsilon^*]^\alpha\colon  [A,\Omega^\sigma\bU]^\alpha\to  [A,\Omega^\sigma\bU]^\alpha$
  is multiplication by $-1$.
\end{prop}
\begin{proof}
 The commutative diagram \eqref{eq:Inv_is_Inv} witnesses that
  the involution of ultra-commutative $C$-monoids
  $\Im\colon\bU\to\bU$ is the inverse morphism; the same is thus true for
  $\Omega^\sigma \Im\colon\Omega^\sigma\bU\to\Omega^\sigma\bU$.
  In particular, the map 
  $[A,\Omega^\sigma\Im]^\alpha\colon  [A,\Omega^\sigma\bU]^\alpha\to  [A,\Omega^\sigma\bU]^\alpha$
  is multiplication by $-1$.
  The commutative diagram \eqref{eq:beta_u_Im}
  shows that the composite
  \[    [A,\bGr]^\alpha\ \xra{[A,\beta]^\alpha}\ [A,\Omega^\sigma\bU]^\alpha\ \xra{[A,\epsilon^*]^\alpha}\
    [A,\Omega^\sigma\bU]^\alpha \]
  equals $[A,(\Omega^\sigma\Im)]^\alpha\circ [A,\beta]^\alpha=-[A,\beta]^\alpha$.
  So $[A,\epsilon^*]^\alpha$  is multiplication by $-1$
  on the image of $[A,\beta]^\alpha$.
  The homomorphism
  \[ [A,i]^\alpha \ : \ [A,\bGr]^\alpha\ \to \ [A,\bBUP]^\alpha \]
  is a group completion of abelian monoids by Proposition \ref{prop:pointwise_groupcomplete}.
  In particular, the group $[A,\bBUP]^\alpha$ is generated by the image of $[A,\beta]^\alpha$.
  The $C$-global equivalence $\gamma\colon \bBUP\xra{\sim}\Omega^\sigma\bU$
  from \eqref{eq:define_gamma} satisfies $\gamma\circ i=\beta\colon\bGr\to\Omega^\sigma\bU$,
  see \eqref{eq:gamma_i=beta}.
  So the group
  $[A,\Omega^\sigma\bU]^\alpha$ is generated by the image of
  $[A,\beta]^\alpha$. Since $[A,\epsilon^*]^\alpha$ is a group homomorphism
  and multiplication by $-1$ on a generating set,
   it is  multiplication by $-1$ on all elements.
\end{proof}

\begin{rk}
  Something stronger than Proposition \ref{prop:u_on_Omega^sigma U} is actually true:
  the morphism $\epsilon^*\colon \Omega^\sigma\bU\to\Omega^\sigma\bU$
  equals $\Omega^\sigma\Im$ in the homotopy category of ultra-commutative $C$-monoids,
  and thus $\epsilon^*=-\Id_{\Omega^\sigma\bU}$.  
  Here is a sketch of the proof.
  An ultra-commutative $C$-monoid $R$ is {\em group-like} if for every augmented Lie
  group $\alpha\colon G\to C$, the abelian monoid $\pi_0^\alpha(R)$ is a group.
  One can adapt the proof of \cite[Theorem 2.5.33]{schwede:global}
  to the Real-global context and show that the morphism $i\colon\bGr\to \bBUP$
  is a Real-global group completion, i.e., initial in the homotopy category
  of ultra-commutative $C$-monoids among morphisms from $\bGr$ to a group-like object.
  The Real-global equivalence $\gamma\colon\bBUP\to\Omega^\sigma\bU$ satisfies
  $\gamma\circ i=\beta$, so the morphism $\beta\colon\bGr\to\Omega^\sigma\bU$
  is also a Real-global group completion.
  Since $\Omega^\sigma\bU$ is group-like and
   $\epsilon^*\circ\beta=(\Omega^\sigma\Im)\circ \beta$
  by \eqref{eq:beta_u_Im},
  the universal property of a 
   Real-global group completion shows that $\epsilon^*=\Omega^\sigma\Im$.
\end{rk}

\subsection{Connective Real-global K-theory}
The connective global K-theory spectrum $\bku$ is defined in
\cite[Construction 3.6.9]{schwede:global},
adapting a configuration space model of Segal \cite{segal:K-homology}
to the global equivariant context.
In this subsection we recall the definition,
along with the involution $\psi$ by complex conjugation,
which enhances it to the connective Real-global K-theory spectrum $\bkr=(\bku,\psi)$.
In \eqref{eq:eig} we recall the eigenspace morphism
$\eig \colon\bU \to\Omega^\bullet(\sh^\sigma  \bkr)$,
which we show to be a $C$-global H-map in Proposition \ref{prop:eig_additive} (i).

\begin{construction}[The connective Real-global K-theory spectrum $\bkr$]\label{con:kr}
  We let $\Uc$ be a hermitian inner product space of countable dimension (finite or infinite).
  We recall the $\Gamma$-space $\bk(\Uc)$ of `orthogonal subspaces in $\Uc$'.
  For a finite based set $A$ we let $\bk(\Uc,A)$ be the space of tuples
  $(E_a)_{a\in A\setminus\{0\}}$, indexed by the non-basepoint elements of $A$,
  of finite-dimensional, pairwise orthogonal $\mC$-subspaces of $\Uc$.
  The topology on $\bk(\Uc,A)$ 
  is that of a disjoint union of subspaces of a product of Grassmannians.
  The basepoint of $\bk(\Uc,A)$ is the tuple where $E_a=\{0\}$
  for all $a\in A\setminus\{0\}$.
  For a based map $\alpha\colon A\to  B$ the induced map
  $\bk(\Uc,\alpha)\colon\bk(\Uc,A)\to \bk(\Uc,B)$
  sends $(E_a)$ to $(F_b)$ where
  \[ F_b \ = \ {\bigoplus}_{\alpha(a)=b}\, E_a \ .\]
  Every $\Gamma$-space can be evaluated on a based space 
  by a coend construction, see for example \cite[(4.5.14)]{schwede:global}.
  Categorically speaking, this coend realizes the enriched Kan extension along
  the inclusion of $\Gamma$ into the category of based spaces.
  We write $\bk(\Uc,K)=\bk(\Uc)(K)$ for the value of the $\Gamma$-space
  $\bk(\Uc)$ on a based space $K$.
  Elements of $\bk(\Uc,K)$ can be interpreted as `labeled configurations': 
  a point is represented by an unordered tuple 
  \[ [ E_1,\dots,E_n;\,k_1,\dots,k_n ] \]
  where $(E_1,\dots, E_n)$ is an $n$-tuple of finite-dimensional,
  pairwise orthogonal subspaces of $\Uc$, 
  and $k_1,\dots,k_n$ are points of $K$, for some $n$. 
  The topology is such that,
  informally speaking, the labels sum up whenever two points collide,
  and a label disappears whenever a point approaches the basepoint of $K$.
  
  The value of the orthogonal spectrum $\bkr$ on a euclidean inner product space $V$ is
  \[  \bkr(V)\ = \ \bk(\Sym(V_\mC),S^V)\ ,   \]
  the value of the $\Gamma$-space $\bk(\Sym(V_\mC))$
  on the sphere $S^V$;
  the inner product on the symmetric algebra
  is described in \cite[Proposition 6.3.8]{schwede:global}.
  The action of the orthogonal group $O(V)$ on $V$ then extends
  to a unitary action on $\Sym(V_\mC)$. 
  We let $O(V)$ act diagonally, 
  via the action on the sphere $S^V$ and the action 
  on the $\Gamma$-space $\bk(\Sym(V_\mC))$. For the structure maps we refer to
  \cite[Construction 6.3.9]{schwede:global}.  
  The involution $\psi(V)\colon\bkr(V)\to\bkr(V)$
  is induced by complex conjugation on $\Sym(V_\mC)$,
  which preserves orthogonality of subspaces.
\end{construction}

\begin{construction}[The eigenspace morphism]
  We recall from  \cite[(6.3.26)]{schwede:global} the {\em eigenspace morphism} 
  \begin{equation}\label{eq:eig}
    \eig \ : \ \bU \ \to \ \Omega^\bullet(\sh^\sigma  \bkr) \ .
  \end{equation}
  Its value at an inner product space $V$
  \begin{align*}
       \eig(V)\ : \ \bU(V) = U(V_\mC) \ \to \ \map_*(S^V,&\bk(\Sym((V\oplus\sigma)_\mC),S^{V\oplus\sigma}))\\
   &= \map_*(S^V,\bkr(V\oplus\sigma)) = \Omega^\bullet(\sh^\sigma  \bkr)(V)
  \end{align*}
  is defined as follows.
  For a unitary endomorphism $A\in U(V_\mC)$,
  we let $\lambda_1,\dots,\lambda_n\in U(1)\setminus\{1\}$ be the eigenvalues 
  different from $1$, and  we let $E(\lambda_j)$ be the eigenspace of $A$
  for the eigenvalue $\lambda_j$.
  We let 
  \begin{equation} \label{eq:define_c_inv}
    \mathfrak c^{-1}\ : \  U(1)\ \xra{\ \iso\ }\   S^\sigma\ , \quad
    \mathfrak c^{-1}(\lambda)=i\cdot(\lambda+1)(\lambda-1)^{-1}      
  \end{equation}
  be the inverse of the Cayley transform \eqref{eq:define_Cayley}.
  Then $\eig(V)$ is defined by
  \[ \eig(V)(A)(v)\ = \
    [E(\lambda_1),\dots,E(\lambda_n);\, (v,\mathfrak c^{-1}(\lambda_1)),\dots,(v,\mathfrak c^{-1}(\lambda_n))]\ .  \]
  In other words, $\eig(V)(A)(v)$ is the configuration
  of the points $(v,\mathfrak c^{-1}(\lambda_i)) \in S^{V\oplus \sigma}$
  labeled by the eigenspace $E(\lambda_i)$ of $A$, whence the name.
  Strictly speaking, $E(\lambda_i)$ is a subspace of $V_\mC$,
  which we embed into the linear summand of $\Sym((V\oplus\sigma)_\mC)$.

  If $A\in\bU(V)$ has eigenspaces $E_1,\dots,E_n$ for eigenvalues
  $\lambda_1,\dots,\lambda_n$, then the eigenspaces of ${^\psi A}$ are
  $\psi(E_1),\dots,\psi(E_n)$ for eigenvalues
  $\bar\lambda_1,\dots,\bar\lambda_n$.
  The involution on $(\sh^\sigma  \bkr)(V)= \bkr(V\oplus\sigma)$
  is by complex conjugation on the labeling vector spaces,
  and by sign on the $\sigma$-coordinate that stores the eigenvalues.
  The inverse Cayley transform $\mathfrak c^{-1}\colon U(1)\to S^\sigma$
  is $C$-equivariant in the sense that $\mathfrak c^{-1}(\bar\lambda)=-\mathfrak c^{-1}(\lambda)$.
  So the map 
  $\eig(V)\colon \bU(V)\to \Omega^\bullet(\sh^\sigma  \bkr)(V)$
  commutes with the involutions.
  The upshot is that the eigenspace morphism \eqref{eq:eig}
  is a morphism of orthogonal $C$-spaces.
\end{construction}

We consider an ultra-commutative $C$-monoid $R$ and an orthogonal $C$-spectrum $Y$.
On several occasions we will need to know that
the map $[A,f]^\alpha\colon[A,R]^\alpha\to [A,\Omega^\bullet(Y)]^\alpha$
induced by a certain morphism
of orthogonal $C$-spaces $f\colon R\to \Omega^\bullet(Y)$ is additive.
To this end, we now introduce the notion of a `$C$-global H-map'.

\begin{construction}\label{con:oplus_stable}
 For every orthogonal $C$-spectrum $Y$, the stable structure
  provides an abelian group structure on the object $\Omega^\bullet(Y)$
  in the $C$-global homotopy category, as follows.
  The canonical morphism $\kappa\colon Y\vee Y\to Y\times Y$
  from the wedge to the product is a $C$-global equivalence of orthogonal $C$-spectra,
  so it induces a $C$-global equivalence of orthogonal $C$-spaces
  \[ \Omega^\bullet(\kappa)\ : \ \Omega^\bullet(Y\vee Y)\ \xra{\ \sim \ }\ \Omega^\bullet(Y\times Y)\iso
    \Omega^\bullet(Y)\times\Omega^\bullet(Y)\ .\]
  This morphism becomes an isomorphism in the based $C$-global homotopy category,
  where we can thus form the composite
  \[ \oplus \ : \  \Omega^\bullet(Y)\times\Omega^\bullet(Y)\ \xra[\iso]{\Omega^\bullet(\kappa)^{-1}} \
    \Omega^\bullet(Y\vee Y) \xra{\Omega^\bullet(\nabla)}\ 
    \Omega^\bullet(Y) \ .
  \]
  Here $\nabla\colon Y\vee Y\to Y$ is the fold morphism.
  The morphism $\oplus$ makes $\Omega^\bullet(Y)$ into an abelian group
  object in the $C$-global homotopy category;
  and for every augmented Lie group $\alpha\colon G\to C$,
  its effect on $\pi_0^\alpha(\Omega^\bullet Y)=\pi_0^G(\alpha^*Y)$
  is the addition on stable homotopy groups.
\end{construction}

\begin{defn}
  Let $R$ be an ultra-commutative $C$-monoid, and let $Y$ be an orthogonal $C$-spectrum.
  A morphism $f\colon R\to \Omega^\bullet Y$ of based orthogonal $C$-spaces
  is a {\em $C$-global H-map} if the following diagram commutes in the
  based unstable $C$-global homotopy category:
  \[ \xymatrix@C=20mm{
      R \boxtimes R \ar[r]^-{(f \rho_1,f\rho_2)}\ar[d]_{\mu_R} & (\Omega^\bullet Y)\times (\Omega^\bullet Y)\ar[d]^\oplus\\
      R\ar[r]_-f & \Omega^\bullet Y} \]
Here $\rho_1,\rho_2\colon R\boxtimes R\to R$ denote the projections to the two factors.
\end{defn}

The ultra-commutative multiplication on $\bU$ induces a
$C$-ring spectrum structure on the suspension spectrum $\Sigma^\infty_+\bU$.
This, in turn, induces a product structure on
the equivariant homotopy groups of  $\Sigma^\infty_+\bU$.
The morphism $\varrho\colon \Sigma^\infty_+\bU\to\mS$
arising from the unique morphism of orthogonal spaces $\bU\to \ast$
is a morphism of ultra-commutative $C$-ring spectra, and thus induces a
morphism of equivariant homotopy rings
\[ \varrho_* \ : \  \pi_\star^G(\Sigma^\infty_+ \bU)\ \to \ \pi_\star^G(\mS)\]
for every augmented Lie group.
The morphism $\bU\to \ast$ has a section given by the multiplicative unit of $\bU$,
so $\varrho_*$ is surjective.
The {\em augmentation ideal} of $\pi_\star^G(\Sigma^\infty_+ \bU)$
is the kernel of this homomorphism $\varrho_*$.
The morphism $q\colon \Sigma^\infty_+ \bU\to \Sigma^\infty \bU$
is induced by the based morphism $\bU_+\to \bU$ that maps the external basepoint
of the source with the multiplicative unit in the target.
  
\begin{prop}\label{prop:eig_additive}\
  \begin{enumerate}[\em (i)]
  \item 
    The morphism $\eig\colon\bU\to\Omega^\bullet(\sh^\sigma\bkr)$ is a $C$-global H-map.
  \item 
    For every augmented Lie group $\alpha\colon G\to C$, the composite
      \[ \pi_\star^G(\Sigma^\infty_+ \bU)\ \xra{\ q_*\ }\  \pi_\star^G(\Sigma^\infty\bU)
        \ \xra{\eig^\flat_*}\ \pi_\star^G(\sh^\sigma\bkr)  \]
   annihilates the square of the augmentation ideal.  
  \end{enumerate}
\end{prop}
\begin{proof}
  For use in both parts we introduce an orthogonal $C$-spectrum $\bkr^{[2]}$
  that is Real-globally equivalent to a product of two copies of $\bkr$.
  The value of the orthogonal spectrum $\bkr^{[2]}$ on a euclidean inner product space $V$ is
  the configuration space
  \[  \bkr^{[2]}(V)\ = \ \bk(\Sym(V_\mC),S^V\vee S^V)\ ,   \]
  the value of the $\Gamma$-space $\bk(\Sym(V_\mC))$
  on the wedge of two copies of $S^V$.
  The action of $O(V)$, the structure maps of $\bkr^{[2]}$
  and the involution are defined in much the same way as for $\bkr$.
  The projection and fold maps 
  \[ p_1 ,\,  p_2,\,  \nabla \ : \ S^V\vee S^V\ \to \ S^V  \]
  induce continuous, based and $(C\times O(V))$-equivariant maps of configuration spaces
  \[ p_1(V),\, p_2(V),\, \nabla(V)\ : \
    \bkr^{[2]}(V)\ = \ \bk(\Sym(V_\mC),S^V\vee S^V)\ \to \ \bk(\Sym(V_\mC),S^V)\ = \ \bkr(V)\ .\]
  For varying inner product spaces $V$,
  these maps assemble into morphisms of orthogonal $C$-spectra 
  \[ p_1,\, p_2,\, \nabla\ :\ \bkr^{[2]}\ \to\ \bkr\ . \]
  We claim that the morphism
  \[  (p_1, p_2)\ : \   \bkr^{[2]}\ \to \ \bkr\times\bkr \]
  is a $C$-global equivalence.
  To see that, we let $\alpha\colon G\to C$ be an augmented Lie group.
  We call an orthogonal $G$-representation $V$  {\em ample} if the complex symmetric algebra
  $\Sym(V_\mC)$ is a complete Real $(G,\alpha)$-universe. If $V$ is ample, then
  the $G$-$\Gamma$-space $\bk(\alpha^\flat(\Sym(V_\mC)),-)$ is special
  by \cite[Theorem 6.3.19 (i)]{schwede:global},
  and it is $G$-cofibrant by \cite[Example 6.3.16]{schwede:global}.
  So the $G$-$\Gamma$-space $\bk(\alpha^\flat(\Sym(V_\mC)),S^1\sm -)$ is very special
  and cofibrant, and thus takes wedges of finite based $G$-CW-complexes to products,
  up to $G$-weak equivalence, by \cite[Theorem B.61 (i)]{schwede:global}.
  In particular, if $V$ is ample and $V^G\ne 0$, then the map
  \begin{align*}
     (p_1(V), p_2(V))\ : \
    \alpha^\flat(\bkr^{[2]}(V)) =  \bk
    &(\alpha^\flat(\Sym(V_\mC)),S^V\vee S^V)\ \to \\
    &\bk(\alpha^\flat(\Sym(V_\mC)),S^V)\times \bk(\alpha^\flat(\Sym(V_\mC)),S^V)
                                           = \alpha^\flat(\bkr(V)\times  \bkr(V)) 
  \end{align*}
  is a $G$-weak equivalence.
  The ample $G$-representations with nonzero $G$-fixed points
  are cofinal in all orthogonal $G$-representations, so this proves
  the claim that the morphism $(p_1, p_2)$ is a $C$-global equivalence.

  (i) We define a morphism of orthogonal $C$-spaces
  $\eig_{[2]}\colon \bU\boxtimes\bU\to \Omega^\bullet(\sh^\sigma  \bkr^{[2]})$,
  a variation of the eigenspace morphism \eqref{eq:eig},
  but with two (instead of one) unitary parameters.
  Its value at an inner product space $V$ is the map
  \begin{align*}
    \eig_{[2]}(V)\ : \
    (\bU\boxtimes\bU)(V) \ \to \ \map_*(S^V,&\bk(\Sym( (V\oplus\sigma)_\mC),S^{V\oplus\sigma}\vee S^{V\oplus\sigma}))\\
    &= \map_*(S^V,\bkr^{[2]}(V\oplus \sigma)) = \Omega^\bullet(\sh^\sigma  \bkr^{[2]})(V)
      \end{align*}
  defined as follows.
  Elements of $(\bU\boxtimes\bU)(V)$ are pairs $(A,B)$
  of unitary endomorphisms $A,B\in U(V)$ that are {\em transverse}
  in the following sense:
  there exists an orthogonal direct sum decomposition
  $V=V'\oplus V''$ such that $A$ is the identity on $V'$, and $B$ is the identity on $V''$.
  The transversality hypothesis in particular means that $A$ and $B$ commute
  (but it is stronger than that),
  so $A$ and $B$ are simultaneously diagonalizable.
  We let $\lambda_1,\dots,\lambda_n\in U(1)\setminus\{1\}$ be the set eigenvalues 
  of $A$ and $B$ different from $1$,
  we let $E(\lambda_j)$ be the eigenspace of $A$ for the eigenvalue $\lambda_j$,
  and we let $F(\lambda_j)$ be the eigenspace of $B$ for the eigenvalue $\lambda_j$.
  By the transversality hypothesis, all these eigenspaces $E(\lambda_i)$ and $F(\lambda_j)$
  are pairwise orthogonal.
  We can then define the map $\eig_{[2]}(V)$ by
  \[ \eig_{[2]}(V)(A,B)(v)\ = \
    \{ E(\lambda_j), i_1(v,\mathfrak c^{-1}(\lambda_j)\}_{1\leq j\leq n}\ \cup\
    \{ F(\lambda_j), i_2(v,\mathfrak c^{-1}(\lambda_j)\}_{1\leq j\leq n}\ .  \]
  As before, $\mathfrak c^{-1}$ is the inverse Cayley transform \eqref{eq:define_c_inv}.
  And $i_1,i_2\colon S^{V\oplus\sigma}\to S^{V\oplus\sigma}\vee S^{V\oplus\sigma}$
  denote the embeddings of the two wedge summands.  
  In other words, the eigenspace $E(\lambda_j)$ of $A$ is attached to
  the point $(v,\mathfrak c^{-1}(\lambda_j))$ in the first wedge summand,
  and the eigenspace $F(\lambda_j)$ of $B$ is attached to
  the point $(v,\mathfrak c^{-1}(\lambda_j))$ in the second wedge summand.

  The following diagram of orthogonal $C$-spaces then commutes:
  \begin{equation} \begin{aligned}  \label{eq:U_kr_Hmap}
       \xymatrix@C=12mm{
    && \Omega^\bullet(\sh^\sigma \bkr)\times \Omega^\bullet(\sh^\sigma \bkr) \\
    \bU\boxtimes\bU\ar[d]_{\mu}
    \ar@/^2pc/[urr]^(.2){(\eig \rho_1,\eig \rho_2)}\ar[rr]_-{\eig_{[2]}}
   &&  \Omega^\bullet(\sh^\sigma \bkr^{[2]})\ar[d]_{\Omega^\bullet(\sh^\sigma \nabla)}
      \ar[u]_\sim^{(\Omega^\bullet(\sh^\sigma  p_1),\Omega^\bullet(\sh^\sigma  p_2))}
    &\Omega^\bullet( (\sh^\sigma \bkr)\vee(\sh^\sigma \bkr))
      \ar[l]_-\sim^-{\Omega^\bullet(\iota)}\ar@/_1pc/[ul]^(.4)\sim_(.3){\Omega^\bullet(\kappa)}
    \ar@/^1pc/[dl]^{\Omega^\bullet(\nabla)}\\
    \bU \ar[rr]_-{\eig} &&\Omega^\bullet(\sh^\sigma \bkr)
}
     \end{aligned}  \end{equation}
  The morphism $(p_1,p_2)\colon\bkr^{[2]}\to\bkr\times\bkr$ is a Real-global equivalence.
  The functors $\sh^\sigma$ and $\Omega^\bullet$ preserve Real-global equivalences and products,
  so the upwards middle vertical morphism is a Real-global equivalence of orthogonal $C$-spaces.
  Hence the morphisms decorated with a tilde are Real-global equivalences.
  After passing to the Real-global homotopy category, we can invert the
  Real-global equivalences, and the resulting diagram witnesses
  that $\eig\colon\bU\to\Omega^\bullet(\sh^\sigma\bkr)$ is a $C$-global H-map.

  (ii) In the commutative diagram \eqref{eq:U_kr_Hmap},
    we pass to adjoints for the adjunction $(\Sigma^\infty,\Omega^\bullet)$.
    This yields the commutativity  in the $C$-global stable homotopy category
    of the right part of the following diagram:
\[ 
      \xymatrix@C=15mm{
      ( \Sigma^\infty_+\bU )\times (\Sigma^\infty_+ \bU) \ar[r]^{q\times q}&
                                                                             ( \Sigma^\infty \bU)\times (\Sigma^\infty \bU)
    \ar[r]^-{\eig^\flat\times \eig^\flat} & (\sh^\sigma \bkr)\times (\sh^\sigma \bkr) \\
    \Sigma^\infty_+ (\bU \boxtimes \bU ) \ar[u]^{(\Sigma^\infty_+ \rho_1,\Sigma^\infty_+ \rho_2)}
    \ar[r]_-{q_{[2]}}\ar[d]_{\Sigma^\infty_+\mu}
    &  \Sigma^\infty (\bU \boxtimes \bU)\ar[d]_{\Sigma^\infty\mu}\ar[r]_-{\eig^\flat_{[2]}}
      \ar[u]^{(\Sigma^\infty \rho_1,\Sigma^\infty \rho_2)} &
       \sh^\sigma(\bkr^{[2]})\ar[d]^{\sh^\sigma \nabla}\ar[u]_{(\sh^\sigma  p_1,\sh^\sigma  p_2)}^\sim\\
    \Sigma^\infty_+ \bU  \ar[r]_-q & \Sigma^\infty \bU   \ar[r]_-{\eig^\flat} & \sh^\sigma \bkr
}
 \]
  Given representation-graded homotopy classes $x\in\pi_V^G(\Sigma^\infty_+\bU)$ and
  $y\in\pi_W^G(\Sigma^\infty_+\bU)$, we write $x\boxtimes y$ for the image of $(x,y)$ under
  the composite
  \begin{align*}
    \pi_V^G(\Sigma^\infty_+ \bU)\times\pi_W^G(\Sigma^\infty_+ \bU)\ \to \
    \pi_{V\oplus W}^G((\Sigma^\infty_+ \bU)\sm(\Sigma^\infty_+ \bU))\ \iso \ 
    \pi_{V\oplus W}^G(\Sigma^\infty_+ (\bU\boxtimes \bU))\ ,
  \end{align*}
  where the first map is the exterior pairing \cite[(3.5.13)]{schwede:global},
  and the isomorphism is induced by the strong symmetric monoidal structure on the
  unreduced suspension spectrum functor, see \cite[(4.1.17)]{schwede:global}.
  Then
  \[  (\Sigma^\infty_+ \mu)_*(x\boxtimes y)\ = \ x\cdot y \ . \]
  The morphism $\Sigma^\infty_+\rho_1\colon  \Sigma^\infty_+(\bU\boxtimes \bU)\to  \Sigma^\infty_+ \bU$
  factors as the composite
  \[ \Sigma^\infty_+ (\bU\boxtimes \bU)\ \iso
    (\Sigma^\infty_+ \bU)\sm(\Sigma^\infty_+\bU)\ \xra{\Id\sm\varrho}\ 
    (\Sigma^\infty_+ \bU)\sm\mS\ \iso\  \Sigma^\infty_+ \bU \ ,  \]
  where the final isomorphism is the unit isomorphism of the smash product.
  So if $y$ belongs to the augmentation ideal, then
  \[    (\Sigma^\infty_+ \rho_1)_*(x\boxtimes y)\ = \ x\cdot \varrho_*(y)\ = \ 0 \ ,\]
  and hence
  \begin{align*}
    (\sh^\sigma p_1)_*((\eig^\flat_{[2]}\circ q_{[2]})_*(x\boxtimes y))\
    &= \     (\eig^\flat\circ q\circ (\Sigma^\infty_+ \rho_1))_*(x\boxtimes y)\ = \ 0\ .
  \end{align*}
  Similarly, $(\sh^\sigma p_2)_*((\eig^\flat_{[2]}\circ q^{[2]})_*(x\boxtimes y))=0$
  whenever $x$ belongs to the augmentation ideal.

  The morphism $(\sh^\sigma  p_1,\sh^\sigma  p_2)$ is a $C$-global equivalence,
  so the induced map on equivariant homotopy groups is an isomorphism.
  So if both $x$ and $y$ lie in the augmentation ideal, we conclude that
  $(\eig^\flat_{[2]}\circ q_{[2]})_*\allowbreak(x\boxtimes y)=0$.
  The commutativity of the above diagram then yields the desired relation:
\[  (\eig^\flat\circ q)_*(x\cdot y)\ =
  \   (\eig^\flat\circ q\circ (\Sigma^\infty_+\mu))_*(x\boxtimes y)\ = \ 
\   ((\sh^\sigma\nabla)\circ\eig^\flat_{[2]}\circ q_{[2]})_*(x\boxtimes y)\ = \ 0 \ . \qedhere\]
\end{proof}

\subsection{Periodic Real-global K-theory}
The global K-theory spectrum $\bKU$
was introduced by Joachim \cite[Definition 4.3]{joachim:coherences}
as a commutative orthogonal ring spectrum,
see also \cite[Construction 6.4.9]{schwede:global}.
Joachim showed in \cite[Theorem 4.4]{joachim:coherences}
that the genuine $G$-spectrum underlying the global spectrum $\bKU$ represents
$G$-equivariant complex K-theory. A different proof can be found in 
\cite[Corollary 6.4.23]{schwede:global}.
Joachim's orthogonal ring spectrum can be enhanced to an orthogonal $C$-spectrum $\bKR$
by suitably incorporating complex conjugation; to my knowledge,
the additional involution was first discussed in the literature
by Halladay and Kamel \cite[Section 6]{halladay-kamel}, who use the notation $K U_\mR$
for the resulting orthogonal $C$-ring spectrum.
In \cite[Proposition 6.2]{halladay-kamel},
Halladay and Kamel identify the genuine $C$-fixed point spectrum of $\bKR$
with $KO$, thereby verifying that $K U_\mR$ is a model for  Atiyah's Real K-theory spectrum.

We recall the definition of the orthogonal $C$-spectrum $\bKR$
in Construction \ref{con:KR}, and we refer to it as
the {\em Real-global K-theory spectrum}.
Theorem \ref{thm:KU_deloops_BUP} justifies the name:
for every augmented Lie group $\alpha\colon G\to C$,
the orthogonal $G$-spectrum  $\alpha^*(\bKR)$ represents
$\alpha$-equivariant Real K-theory $KR_\alpha$.
Our treatment proceeds from a self-contained proof
in Theorem \ref{thm:a_nilpotent} that the pre-Euler class
of the $C$-sign representation is nilpotent in $\pi_\star^C(\bKR)$,
which immediately implies that
the geometric fixed point homotopy groups $\Phi^G_*(\alpha^*(\bKR))$
vanish for all {\em surjective} augmentations $\alpha\colon G\to C$.
  This implies that the Real-global homotopy type of
  $\bKR$ is coinduced (or `relative $C$-Borel')
  from the underlying global spectrum.
  Theorem \ref{thm:U2shKR} shows that (and how) the Real-global space
  $\bU$ is the Real-global infinite loop space underlying $\bKR\sm S^\sigma$.
  Consequently, the Real-global space $\bBUP\sim \Omega^\sigma\bU$
  `is' the Real-global infinite loop space underlying $\bKR$,
  see Theorem \ref{thm:KU_deloops_BUP} (i).

  \begin{construction}[Clifford algebras]
  We let $V$ be a euclidean inner product space. We define the complex
  Clifford algebra $\mCl(V)$ by
  \[ \mCl(V) \ = \ (T V)_\mC / ( v\tensor v - |v|^2\cdot 1) \ , \]
  the quotient of the complexified tensor algebra of $V$
  by the ideal generated by the elements $v\tensor v -  |v|^2\cdot 1$
  for all $v\in V$.  We write $[-]\colon V\to \mCl(V)$ for the
  $\mR$-linear and injective composite
  \[ V \ \xra{\text{linear summand}} \  T V\ \xra{1\tensor- }\ (T V)_\mC
  \ \to \  \mCl(V)  \ .\]
  With this notation the relation $[v]^2=|v|^2\cdot 1$ holds in $\mCl(V)$
  for all $v\in V$.
  The Clifford algebra construction is functorial for $\mR$-linear isometric embeddings,
  so in particular $\mCl(V)$ inherits an action of the orthogonal group $O(V)$. 
  The Clifford algebra is $\mZ/2$-graded, coming from the grading of
  the tensor algebra by even and odd tensor powers.

  The complex Clifford algebra is in fact a $\mZ/2$-graded $C^\ast$-algebra.
  The $\ast$-involution on $\mCl(V)$ is defined by declaring $[v]^\ast=[v]$
  for all $v\in V$ and extending this to a $\mC$-semilinear anti-automorphism. 
  This makes the elements $[v]$ for $v\in S(V)$ into unitary elements of $\mCl(V)$.
  The norm on $\mCl(V)$ arises from the operator norm on the endomorphism
  algebra of the exterior algebra of $V_\mC$, as explained, for example,
  in \cite[Construction 6.4.5]{schwede:global}.

  The complex Clifford algebra $\mCl(V)$
  supports three different involutions relevant for our purposes:
\begin{itemize}
\item The grading involution $\alpha\colon \mCl(V)\to \mCl(V)$;
  it is $\mC$-linear and multiplicative, and satisfies $\alpha[v]=-[v]$.
  The $+1$ and $-1$ eigenspaces of $\alpha$ are, respectively,
  the even and odd summands of $\mCl(V)$.
\item The conjugation involution $(-)^*\colon \mCl(V)\to \mCl(V)$;
  it is conjugate-linear and anti-multiplicative, and satisfies $[v]^*=[v]$.
\item We define $\psi\colon \mCl(V)\to \mCl(V)$ as the unique
  conjugate-linear and multiplicative map satisfying $\psi[v]=[v]$.
  The fixed points of this involution are thus the real Clifford algebra
  associated with the positive definite form, i.e.,
  \[ \mCl(V)^{\psi}\ = \  (T V)/(v\tensor v - |v|^2\cdot 1) \ . \]
\end{itemize}

The three involutions $\alpha$, $(-)^*$ and $\psi$ of $\mCl(V)$
commute with each other.
The pair $((-)^*,\alpha)$ makes $\mCl(V)$ into a complex, $\mZ/2$-graded $C^*$-algebra,
and only these data enter into the definition of $\bKU$ as an orthogonal spectrum.
The conjugate-linear involution $\psi$ of $\mC(V)$ will enter in the definition
of the involution of $\bKR$ which makes it an orthogonal $C$-ring spectrum.
\end{construction}

\begin{construction}[The Real-global K-theory spectrum]\label{con:KR}
  For a euclidean inner product space $V$,
  we write $\Hc_V$ for the Hilbert space completion of
  the complexified symmetric algebra $\Sym(V_\mC)$,
  with respect to the inner product specified in \cite[Proposition 6.3.8]{schwede:global}.
  Then $\Kc_V$ denotes the $C^*$-algebra of compact operators on $\Hc_V$,
  concentrated in even grading.
  The value of the commutative orthogonal ring spectrum $\bKU$
  at $V$ is defined by
  \[ \bKU(V)\ = \ C^*_{\gr}(s,\mCl(V)\tensor_\mC\Kc_V) \ .\]
  The unit map of the ring spectrum $\bKU$
  uses the `functional calculus'
  \begin{equation}\label{eq:functional_calculus}
    \fc \ : \ S^V \ \to \  C^\ast_{\gr}(s,\mCl(V)) \ , \quad v\ \longmapsto \ (-)[v] \ ,
  \end{equation}
  For $v\in V$ the $\ast$-homomorphism $\fc(v)$ is given on homogeneous elements
  of $s$ by
  \[ f[v]\ = \ \fc(v)(f)\ = \ 
  \begin{cases}
    \ f(|v|)\cdot 1 & \text{ when $f$ is even, and}\\
    \ \frac{ f(|v|)}{|v|}\cdot [v] & \text{ when $f$ is odd.}
  \end{cases}
  \]
  For $v=0$ the formula for odd functions is to be interpreted as $f[0]=0$; 
  this is continuous because for $v\ne 0$, 
  the norm of $f(|v|)/|v|\cdot [v]$ is $f(|v|)$, which tends to $f(0)=0$ 
  if $v$ tends to~0. If the norm of $v$ tends to infinity, then $f(|v|)$ tends to~0, so 
  $\fc(v)$ tends to the constant $\ast$-homomorphism 
  with value~0, the basepoint of $C^\ast_{\gr}(s,\mCl(V))$.
  The unit map
  \begin{equation}\label{eq:unit_of_KR}
 \eta_V \ : \ S^V \ \to \  C^\ast_{\gr}(s,\mCl(V)\tensor_\mC \Kc_V) \ = \ \bKU(V)     
  \end{equation}
  is then defined as the composite
  \[ \ S^V \ \xra{\ \fc\ } \ C^\ast_{\gr}(s,\mCl(V))\ 
\xra{(-\tensor p_0)_*} \ C^\ast_{\gr}(s,\mCl(V)\tensor_\mC\Kc_V)\ = \ \bKU(V) \ ,  \]
where $p_0\in\Kc_V$ is the orthogonal projection onto the constant summand
in the symmetric algebra. The multiplicativity of the unit maps follows from 
the multiplicativity property \cite[(6.4.8)]{schwede:global}
of the functional calculus maps.

Now we explain the involution on $\bKU(V)$,
see also \cite[Definition 6.2]{halladay-kamel}.
Complex conjugation is an conjugate-linear isometry
$\psi\colon \Sym(V_\mC)\to \Sym(V_\mC)$;
so it extends uniquely to a continuous conjugate-linear isometry
$\psi\colon \Hc_V\to \Hc_V$ of the Hilbert space completion.
Conjugation by $\psi$ is then a conjugate-linear involutive automorphism 
\[ (-)^\psi \ : \ \Kc_V \ \to \ \Kc_V \]
  of the $C^*$-algebra of compact operators. For example,
  if $E$ is a finite-dimensional $\mC$-subspace of $\Sym(V_\mC)$,
  and $p_E\in \Kc_V$ the associated orthogonal projection, then
  \[  (p_E)^\psi \ = \ \psi\circ p_E\circ \psi \ = \ p_{\psi(E)}\ .    \]
  We combine the conjugate-linear and multiplicative
  involutions $\psi\colon\mCl(V)\to\mCl(V)$ and  
  $(-)^\psi \colon\Kc_V \to \Kc_V$ into one
  conjugate-linear and multiplicative involution
  \[ \psi\tensor (-)^\psi \ : \ \mCl(V)\tensor_\mC \Kc_V \ \to \ \mCl(V)\tensor_\mC \Kc_V\ . \]
  Finally, we define an involution
  \[  \psi\ : \ \bKU(V)\ = \ C^*_{\gr}(s,\mCl(V)\tensor_\mC\Kc_V) \ \to\
    C^*_{\gr}(s,\mCl(V)\tensor_\mC\Kc_V) \ = \ \bKU(V)\]
  by sending a homomorphism $h\colon s\to\mCl(V)\tensor_\mC\Kc_V$ to the composite
  \[ s \ \xra{\ (-)^*\ }\ s \ \xra{\ h\ }\ \mCl(V)\tensor_\mC \Kc_V\
    \xra{\psi\tensor (-)^\psi} \  \mCl(V)\tensor_\mC \Kc_V \ . \]
  The first map is conjugation on $s$, i.e., pointwise complex conjugation
  of functions.
  If $h$ is a $\ast$-homomorphism, then in particular $h\circ(-)^*=(-)^*\circ h$.
  So the involution on $\bKU(V)$ is also given by postcomposition of
  $h\colon s\to\mCl(V)\tensor_\mC\Kc_V$ with the map
  \[  \mCl(V)\tensor_\mC \Kc_V\ \xra{(-)^*}\  \mCl(V)\tensor_\mC \Kc_V\
    \xra{\psi\tensor (-)^\psi} \  \mCl(V)\tensor_\mC \Kc_V \ . \]
  This composite is $\mC$-linear and anti-multiplicative.
  We omit the verification that these involutions are compatible
  with the multiplication and the structure maps of the orthogonal spectrum $\bKU$.
  We write $\bKR=(\bKU,\psi)$ for the {\em periodic Real-global $K$-theory spectrum},
  i.e., the commutative orthogonal $C$-ring spectra $\bKU$ with $C$-action by this involution.
\end{construction}

We write $a=a_\sigma\in\pi_{-\sigma}^C(\mS)$ for the pre-Euler class \eqref{eq:preEuler}
of the sign representation of the group $C$,
and also for its image in $\pi_{-\sigma}^C(\bKR)$ under the unit morphism $\mS\to\bKR$.
In the $C$-spectrum that represents Atiyah's Real K-theory,
the third power of $a$ vanishes.
Our proof that the orthogonal $C$-spectrum $\bKR$ models the `correct' Real-global homotopy type
proceeds by proving this property for $\bKR$ from scratch.

\begin{theorem}\label{thm:a_nilpotent}
  The relation $a^3=0$ holds in $\pi_{-3\sigma}^C(\bKR)$.
\end{theorem}
\begin{proof}
 For every $n\geq 0$, the class $a^n$ is represented by the composite
 \[ S^0 \ \xra{\ a^n\ }\ S^{n\sigma}\ \xra{\eta_{n\sigma}}\ \bKR(n\sigma) \ .\]
 By definition \eqref{eq:unit_of_KR} of the unit maps of $\bKR$, that composite
 factors through $\fc\circ a^n\colon S^0\to C^*_{\gr}(s,\mCl(n\sigma))$.
 So it suffices to show that for $n=3$,
 the $C$-fixed point $\fc(0)$ lies in the same path component of
 \[ \left(C^*_{\gr}(s,\mCl(3\sigma))\right)^C  \]
 as the basepoint, the zero homomorphism.
 We will show that this fixed point space is  homeomorphic to a circle,
 and so in particular path-connected.

 For any $n$, the $C$-action on the representation $n\sigma$ is by multiplication by $-1$.
 Hence the induced involution on the Clifford algebra
 \[ \mCl(-\Id)\ : \ \mCl(n\sigma)\ \to \  \mCl(n\sigma) \]
 is a $\mC$-algebra automorphism that negates the element $[v]$
 for all $v\in S(n\sigma)$.
 So for the $C$-representation $n\sigma$, the involution $\mCl(-\Id)$
 induced by functoriality of $\mCl(-)$  coincides with the grading involution $\alpha$.
 The involution on $C^*_{\gr}(s,\mCl(n\sigma))$ relevant for our
 present purpose is thus given by
 \begin{itemize}
 \item conjugating a graded $\ast$-homomorphism $f\colon s\to \mCl(n\sigma)$
   by the complex conjugations $(-)^*\colon s\to s$ (pointwise complex conjugation)
   and $\psi\colon \mCl(n\sigma)\to \mCl(n\sigma)$, and
 \item applying the involution $\mCl(-\Id_{n\sigma})=\alpha$ coming from
   functoriality of $\mCl(-)$.   
 \end{itemize}
 In other words, we need to identify the fixed points of the involution
 \[ C^*_{\gr}(s,\mCl(n\sigma))\ \to \ C^*_{\gr}(s,\mCl(n\sigma)) \ , \quad
   f \ \longmapsto \ \alpha\circ \psi\circ f\circ (-)^* \ . \]
 For every $\mZ/2$-graded complex $C^*$-algebra $A$,
 evaluation at the function
 \[ r\ : \ \mR \ \to \ \mC \ , \quad r(x)\ = \ \frac{2 i}{x-i} \]
 in the $C^*$-algebra $s$ is a homeomorphism
 \[ C^*_{\gr}(s,A)\ \iso \ \{x\in A\colon x x^*=x^* x=-x-x^*\ , \ \alpha(x)=x^*\}\ ,\]
 compare \cite[(6.4.3)]{schwede:global}.
 For $A=\mCl(n\sigma)$, the relevant involution
 on the left hand side corresponds to the involution on the right hand
 side sending $x\in \mCl(n\sigma)$ to $\psi(\alpha(x^*))=(\psi(\alpha(x)))^*$.
The previous map thus restricts to a homeomorphism
\begin{align*}
   \left( C^*_{\gr}(s,\mCl(n))\right)^C\
  &\iso \ \{x\in\mCl(n) A\colon x x^*=x^* x=-x-x^*\ , \ \alpha(x^*)=x= \psi(x) \}\\
    &\iso \ \{x\in C(n)\colon x x^*=x^* x=-x-x^*\ , \ \alpha(x^*)=x\}\ .
\end{align*}
The second homeomorphism uses that $\mCl(n)$
is the complexification of the real Clifford algebra
\[ C(n)\ = \ T(\mR^n)/ ([v]\tensor [v] -|v|^2\cdot 1) \ ,\]
with $\psi$ corresponding to the complex conjugation automorphism
of $\mC\tensor_\mR C(n)$.

Now we specialize to $n=3$. 
We write $e=[1,0,0]$, $f=[0,1,0]$ and $g=[0,0,1]$ for the
multiplicative generators of  $C(3)$ coming from the standard
orthonormal basis of $\mR^3$. Then  $e$, $f$ and $g$ are odd,
pairwise anticommuting, fixed by the conjugation $(-)^*$, and satisfy
\[ e^2\ =\ f^2\ =\ g^2\ =\ 1\ .\]
We want to find all $x\in C(3)$ that satisfy 
\[  x\cdot x^*\ =\ x^*\cdot x\ =\ -x-x^*\text{\quad and\quad} \alpha(x^*)\ =\  x^*\ .\]
The involution $x\mapsto \alpha(x^*)=\alpha(x)^*$ fixes the basis elements
$1$ and $efg$, and it negates  $e$, $f$, $g$, $e f$, $e g$, and $f g$.
The condition $\alpha(x^*) =x$ thus forces $x$ to be an $\mR$-linear
combination of $1$ and $e f g$. So  $x=y + z\cdot e f g$ for some   $y,z\in\mR$.
The relations $x\cdot x^*=x^*\cdot x=-x-x^*$ then amount to
\[ (y+1)^2+z^2\ =  \ 1 \ .\]
The solutions in $\mR^2$ to this equation form a circle, as claimed.
\end{proof}

We write $\tilde EC=(E C)^\diamond$ for the unreduced suspension
of the universal free $C$-space $E C$.

\begin{theorem}\label{thm:Omega^bullet coinduced}
  Let $M$ be an orthogonal $C$-spectrum that admits
  a $\bKR$-module structure in the Real-global stable homotopy category.
  \begin{enumerate}[\em (i)]
  \item
    The Real-global spectrum $M\sm \tilde EC$ is trivial.
  \item
    The orthogonal $C$-space $\Omega^\bullet(M)$ is coinduced.
  \end{enumerate}
\end{theorem}
\begin{proof}
  (i) We let $\alpha\colon G\to C$ be any augmented Lie group.
  If the augmentation $\alpha$ is trivial, then $G$ acts trivially on
  $\alpha^*(\tilde EC)$, which is thus $G$-equivariantly contractible.
  In particular, the geometric fixed point homotopy groups $\Phi^G_*(M\sm\tilde EC)$
  vanish.

  If the augmentation $\alpha$ is surjective, then we need to argue differently.
  Because $a^3=0$, also $\alpha^*(a)^3=0$ in $\pi_{-3\alpha^*(\sigma)}^G(\alpha^*(\bKR))$.
  Because $(\alpha^*(\sigma))^G=\{0\}$, we have 
  So $\Phi^G(\alpha^*(a))=1$ in the geometric fixed point ring $\Phi^G_0(\alpha^*(\bKR))$.
  So $1=\Phi^G(\alpha^*(a))^3=\Phi^G(\alpha^*(a^3))=0$, and so 
  the ring $\Phi^G_0(\alpha^*(\bKR))$ is trivial.
  Since $M$ is a $\bKR$-module spectrum, the group
  $\Phi^G_k(\alpha^*(M))$ admits a module structure over $\Phi^G_0(\alpha^*(\bKR))=0$,
  and thus  $\Phi^G_k(\alpha^*(M))=0$ for all $k\in\mZ$.
  Since $\alpha\colon G\to C$ is surjective, we have $(\alpha^*(\tilde E C))^G=S^0$,
  and thus
  \[ \Phi^G_*(\alpha^*(M\sm\tilde EC))\ \iso \ \Phi^G_*(\alpha^*(M))\ = \ 0 \ . \]
  So all geometric fixed point homotopy groups of $M\sm\tilde E C$ vanish,
  for all augmented Lie groups, and thus $M\sm\tilde E C$ is trivial in
  the Real-global stable homotopy category.

  (ii) Since $M$ is a $\bKR$-module, also the Real-global spectrum
  $\map(E C,M)$ admits a $\bKR$-module structure.
  So the Real-global spectra $M\sm\tilde E C$ and
  $\map(E C,M)\sm \tilde E C$ are trivial by part (i).
  Hence the two horizontal morphisms in the following commutative square
  are Real-global equivalences:
  \[ \xymatrix@C=22mm{
      M\sm E C_+\ar[r]^-{M\sm p_+}_-\sim\ar[d]_{p^*\sm EC_+}& M\ar[d]^{p^*} \\
     \map(E C,M)\sm E C_+\ar[r]_-{\map(E C,M)\sm p_+}^-\sim &\map(E C,M)
    } \]
  Since the morphism $p^*\colon M\to\map(E C,M)$ is global equivalence
  of underlying non-Real global spectra, the left horizontal morphism
  is also a Real-global equivalence, by the stable analog of Proposition \ref{prop:times_free}.
  Hence the right vertical morphism
  $p^*\colon M\to\map(E C,M)$ is a Real-global equivalence.   
  The functor $\Omega^\bullet$ takes Real-global equivalences
  of orthogonal $C$-spectra to Real-global equivalences
  of orthogonal $C$-spaces; and it commutes with $\map(E C,-)$. So the morphism
  \[ 
    \Omega^\bullet(M)\ \xra{ \Omega^\bullet(p^*)} \Omega^\bullet(\map(E C,M)) \iso
    \map(E C,\Omega^\bullet M) \] 
  is a Real-global equivalence.
  Proposition \ref{prop:fibrant_criterion} then shows that $\Omega^\bullet(M)$
  is coinduced.
\end{proof}

\begin{construction}\label{con:kr2KR}
  We recall from \cite[Construction 6.4.13]{schwede:global}
  the construction of the morphism of commutative orthogonal $C$-ring spectra
  \[ j \ : \ \bkr \ \to \ \bKR \]
  from connective to periodic Real-global K-theory.
  Its value at a euclidean inner product space $V$ is the map
  \[ j(V)\ : \ \bkr(V)\ = \ \bk(\Sym(V_\mC),S^V)\ \to \
    C^*_{\gr}(s,\mCl(V)\tensor_\mC\Kc_V) \ = \ \bKR(V)  \]
  defined by
  \[ j(V)[E_1,\dots,E_n;\,v_1,\dots,v_n](f)\ = \ \sum_{i=1}^n f[v_i]\tensor p_{E_i} \ .\]
  Here $f[v_i]=\fc(v_i)(f)$ is the functional calculus map \eqref{eq:functional_calculus},
  and $p_{E_i}$ is the orthogonal projection onto the subspace $E_i$ of $\Sym(V_\mC)\subset \Hc_V$.
  Already in \cite[Construction 6.4.13]{schwede:global} we omitted the detailed
  verification that these maps indeed form a morphism of orthogonal ring spectra.
  We shall honor this tradition here, and also refrain from checking that
  the map $j(V)$ commutes with the complex-conjugation involutions. 
\end{construction}

The following theorem generalizes \cite[Theorem 6.4.21]{schwede:global}
from the global to the Real-global context; it says that
$\bU$ is a Real-global genuine infinite loop space,
with delooping the $\sigma$-suspension of $\bKR$.

\begin{theorem}\label{thm:U2shKR}
  The morphism of based orthogonal $C$-spaces
  \[   \Omega^\bullet(\sh^\sigma j)\circ\eig\ : \ \bU \ \to \ \Omega^\bullet(\sh^\sigma\bKR) \]
  is a Real-global equivalence and a $C$-global H-map.  
\end{theorem}
\begin{proof}
  The orthogonal $C$-space $\bU$ is coinduced by Theorem \ref{thm:coinduced}.
  The orthogonal $C$-space $\Omega^\bullet(\sh^\sigma\bKR)$
  is coinduced by Theorem \ref{thm:Omega^bullet coinduced}.
  The morphism $\Omega^\bullet(\sh^\sigma j)\circ \eig\colon\bU\to \Omega^\bullet(\sh^\sigma \bKR)$
  is a global equivalence of underlying non-Real orthogonal spaces
  by \cite[Theorem 6.4.21]{schwede:global}.  
  So $\Omega^\bullet(\sh^\sigma j)\circ \eig$ is a Real-global equivalence
  by Proposition \ref{prop:coinduced_coinduces}.

  The eigenspace morphisms is a $C$-global H-map by Proposition \ref{prop:eig_additive} (i).
  Since $\Omega^\bullet(\sh^\sigma j)$ arises from a morphism of orthogonal $C$-spectra,
  also the composite
   $\Omega^\bullet(\sh^\sigma j)\circ \eig$ is a $C$-global H-map.
\end{proof}

Now we proceed to justify that $\bKR$ deserves its name:
for every augmented Lie group $\alpha\colon G\to C$,
the orthogonal $G$-spectrum $\alpha^*(\bKR)$ represents $\alpha$-equivariant
Real $K$-theory $KR_\alpha$.
If the augmentation $\alpha\colon G\to C$ is trivial, this amounts to
the fact that the underlying $G$-spectrum of
the global spectrum $\bKU$ represents $G$-equivariant
complex $K$-theory, see  \cite[Theorem 4.4]{joachim:coherences} or
\cite[Corollary 6.4.23]{schwede:global}.
For $G=C$ augmented by the identity, this amounts to the fact
that the genuine $C$-spectrum underlying $\bKR$
is a model for Atiyah's Real K-theory spectrum, see \cite[Proposition 6.2]{halladay-kamel}.

\begin{construction}
We let $\Theta\colon\bBUP\to \Omega^\bullet(\bKR)$ be the unique morphism
in the unstable based $C$-global homotopy category that makes the following diagram
commute:
\begin{equation}\begin{aligned}  \label{eq:define_Theta}
     \xymatrix{ \bBUP\ar[d]_\gamma^\sim \ar[rrrr]^-{\Theta}&&&&
    \Omega^\bullet(\bKR)\ar[d]^{\Omega^\bullet(\tilde\lambda^\sigma_{\bKR})}_\sim\\
    \Omega^\sigma\bU \ar[rrr]^-\sim_-{\Omega^\sigma(\Omega^\bullet(\sh^\sigma j)\circ \eig)} &&&
    \Omega^\sigma(\Omega^\bullet(\sh^\sigma\bKR)) \ar[r]_-\iso &
                                                                 \Omega^\bullet(  \Omega^\sigma(\sh^\sigma\bKR)) }
\end{aligned}\end{equation}
The Real-global equivalences $\gamma$ and $\tilde\lambda^\sigma_{\bKR}$
are defined in \eqref{eq:define_gamma}
and \cite[Proposition 3.1.25]{schwede:global}, respectively,
and the unnamed isomorphism interchanges loop coordinates.
The lower left horizontal morphism is a Real-global equivalence by Theorem \ref{thm:U2shKR}.
\end{construction}

\begin{theorem}\label{thm:KU_deloops_BUP}\
  \begin{enumerate}[\em (i)]
  \item The morphism $\Theta\colon\bBUP\to \Omega^\bullet(\bKR)$ is a Real-global
    equivalence and a $C$-global H-map.
\item   For every augmented Lie group $\alpha\colon G\to C$
  and every finite $G$-CW-complex $A$, the composite
  \[       KR_\alpha(A)\ \iso_{\eqref{eq:BUP_to_KR_alpha}}\
  [A,\bBUP]^\alpha \ \xra{[A,\Theta]^\alpha}\ \bKR^0_\alpha(A) \]
  is an isomorphism of rings.
  \end{enumerate}
\end{theorem}
\begin{proof}
  (i) Since all other morphisms in the defining commutative diagram \eqref{eq:define_Theta}
  for $\Theta$ are Real-global equivalences, so is $\Theta$.
  For the second claim we exploit that the morphism 
  $\Omega^\bullet(\sh^\sigma j)\circ\eig\colon\bU\to \Omega^\bullet(\sh^\sigma\bKR)$
  is a $C$-global H-map, by Theorem \ref{thm:U2shKR}.
  Hence also the lower horizontal composite in \eqref{eq:define_Theta} is a $C$-global H-map.
  Since $\gamma$ is a morphism of ultra-commutative $C$-monoids, the
  counter clockwise composite in \eqref{eq:define_Theta} is also a $C$-global H-map.
  Since $\Omega^\bullet(\tilde\lambda^\sigma_{\bKR})$ arises
  from a stable $C$-global morphism, it commutes with the `stable H-space structures'
  from Construction \ref{con:oplus_stable}. So $\Theta$ is a $C$-global H-map, as claimed.  

  (ii) Since $\Theta$ is both a Real-global equivalence and a $C$-global H-map,
  the induced map $[A,\Theta]^\alpha$ is a group isomorphism.
  The map \eqref{eq:BUP_to_KR_alpha} is a group isomorphism by 
  Theorem \ref{thm:BUP_represents} (iii).
  So it remains to show that the composite in the statement of the theorem is also
  multiplicative, and hence an isomorphism of rings.
  By design, the composite factors as
  \[       KR_\alpha(A)\ \iso_{\eqref{eq:BUP_to_KR_alpha}}\
    [A,\bBUP]^\alpha \ \xra{[A,\vartheta]^\alpha}\ \bkr^0_\alpha(A) \ \xra{\ j_* \ }\
    \bKR^0_\alpha(A) \ ,\]
  where $\vartheta\colon\bBUP\to \Omega^\bullet(\bkr)$ is the unique morphism
  in the unstable based $C$-global homotopy category that makes the following diagram
  commute:
  \[  \xymatrix{ \bBUP\ar[d]_\gamma^\sim \ar[rrr]^-{\vartheta}&&&  \Omega^\bullet(\bkr)
      \ar[d]^{\Omega^\bullet(\tilde\lambda^\sigma_{\bkr})}_\sim\\
      \Omega^\sigma\bU \ar[rr]_-{\Omega^\sigma\eig} &&
      \Omega^\sigma(\Omega^\bullet(\sh^\sigma\bkr)) \ar[r]_-\iso &
       \Omega^\bullet(\Omega^\sigma(\sh^\sigma\bkr)) }\]
  Since $j\colon\bkr\to\bKR$ is a morphism of ultra-commutative $C$-ring spectra,
  the induced map $j_*\colon\bkr^0_\alpha(A) \to\bKR^0_\alpha(A)$ of equivariant cohomology
  theories is multiplicative.
  We claim that $[A,\vartheta]^\alpha\circ \eqref{eq:BUP_to_KR_alpha}\colon KR_\alpha(A)\to \bkr^0_\alpha(A)$ is also multiplicative.
  If the augmentation $\alpha$ is trivial, this is shown in
  \cite[Theorem 6.3.31 (ii)]{schwede:global}.
  The proof for surjective augmentations is analogous, and we omit it.  
\end{proof}

\end{appendix}

\medskip

\end{document}